\documentclass[11pt,a4paper,leqno]{amsart}

\title
[The filter of singularities in global anisotropic microlocal analysis]
{The filter of singularities in global anisotropic microlocal analysis}

\author[L. Rodino]{Luigi Rodino}

\address{Department of Mathematics, Universit\`a di Torino, Via Carlo Alberto 10,
10123 Torino, Italy}

\email{luigi.rodino[AT]unito.it}

\author[P. Wahlberg]{Patrik Wahlberg}

\address{Dipartimento di Scienze Matematiche, Politecnico di Torino, Corso Duca degli Abruzzi 24,
10129 Torino, Italy}

\email{patrik.wahlberg[AT]polito.it}

\usepackage{latexsym}
\usepackage[a4paper]{geometry}
\usepackage{amsmath}
\usepackage{amssymb}
\usepackage{amsthm}
\usepackage{bbm}
\usepackage{amsfonts}
\usepackage{mathrsfs}
\usepackage{calc}
\usepackage{cite}
\usepackage{color}
\usepackage{graphicx}
\usepackage{subcaption}
\usepackage{caption}
\usepackage{enumerate}
\usepackage{stackrel}

\numberwithin{equation}{section}          

\newtheorem{thm}{Theorem}
\numberwithin{thm}{section}

\newcommand{\rubrik}{}
\newtheorem{prop}[thm]{Proposition}
\newtheorem{cor}[thm]{Corollary}
\newtheorem{lem}[thm]{Lemma}

\theoremstyle{definition}

\newtheorem{defn}[thm]{Definition}
\newtheorem{example}[thm]{Example}

\theoremstyle{remark}

\newtheorem{rem}[thm]{Remark}              

\newcommand{\scal}[2]{\langle #1,#2\rangle}
\newcommand{\pd}[1] {\partial ^#1}
\newcommand{\pdd}[2] {\partial_{#1} ^{#2}}

\newcommand{\ro}{\mathbf R}
\newcommand{\no}{\mathbf N}
\newcommand{\qo}{\mathbf Q}
\newcommand{\rr}[1]{\mathbf R^{#1}}
\newcommand{\sr}[1]{\mathbf S^{#1}}
\newcommand{\sro}[1]{\mathbf S}
\newcommand{\nn}[1]{\mathbf N^{#1}}

\newcommand{\co}{\mathbf C}

\newcommand{\dd}{\mathrm {d}}

\newcommand{\ep}{\varepsilon}
\newcommand{\fy}{\varphi}

\newcommand{\cdo}{\, \cdot \, }

\newcommand{\supp}{\operatorname{supp}}

\newcommand{\wpr}{{\text{\footnotesize $\#$}}}

\newcommand{\eabs}[1]{\langle #1\rangle}

\newcommand{\Sp}{\operatorname{Sp}}

\newcommand{\GL}{\operatorname{GL}}

\newcommand{\Mp}{\operatorname{Mp}}
\newcommand{\charac}{\operatorname{char}}

\newcommand{\rB}{\operatorname{B}}

\newcommand{\WF}{\mathrm{WF}}
\newcommand{\WFg}{\mathrm{WF_{\rm g}}}
\newcommand{\WFgs}{\mathrm{WF_{g}^{\it \sigma}}}

\newcommand{\rP}{\mathbb P}

\newcommand{\cS}{\mathscr{S}}

\newcommand{\cF}{\mathscr{F}}
\newcommand{\fF}{\mathcal{F}}
\newcommand{\gG}{\mathcal{G}}

\newcommand{\cK}{\mathscr{K}}

\newcommand{\J}{\mathcal{J}}
\newcommand{\wt}{\widetilde}
\newcommand{\wh}{\widehat}

\newcommand{\im}{{\rm Im}}
\newcommand{\dom}{{\rm dom}}
\def\la{\langle}
\def\ra{\rangle}
\newcommand{\leqs}{\leqslant}
\newcommand{\geqs}{\geqslant}

\begin{document}

\begin{abstract}
We define a filter of time-frequency anisotropic global singularities
of phase space for tempered distributions. 
The filter contains information from the corresponding anisotropic Gabor wave front set
and admits propagation results for the Cauchy problem for 
certain linear evolution equations of Schr\"odinger type that generalize the 
harmonic oscillator.
\end{abstract}


\maketitle

\section{Introduction}\label{sec:intro}

In this paper we develop tools for linear evolution equations of the form 
\begin{equation}\label{eq:LinEvEq}
\begin{cases} 
\partial_t u (t,x)+ i  a^w(x,D_x) u (t,x )= 0, \qquad x \in \rr d, \qquad t \in [-T,T] \setminus 0,  \quad T > 0, \\ 
u(0,\cdot) = u_0 \in \cS'(\rr d)
\end{cases}. 
\end{equation}	
where $a^w(x,D_x)$ is a Weyl pseudodifferential operator acting on $x \in \rr d$. 
We address propagation of singularities in the $x$-variables from the initial datum $u_0$
to the solution $u(t,\cdot)$ at time $t \in [-T,T]$. 

The harmonic oscillator hamiltonian $a^w(x,D) = |x|^2 - \Delta_x$ has Weyl symbol $a(x,\xi) = |x|^2 + |\xi|^2$. 
For a more general homogeneous real-valued quadratic form $a(x,\xi) = \la (x,\xi), Q (x,\xi) \ra$ 
defined by a symmetric matrix $Q \in \rr {2d \times 2d}$,
a natural concept of singularities is the Gabor wave front set $\WFg(u) \subseteq T^* \rr d \setminus \{ 0 \}$
of a tempered distribution $u \in \cS'(\rr d)$. 
Indeed we have \cite{Hormander2,PravdaStarov1}
\begin{equation}\label{eq:PropWFg}
\WFg \left( u(t,\cdot)\right) = \chi_t \left(  \WFg( u_0) \right), \quad t \in \ro,
\end{equation}
where $\chi_t(x,\xi)$ is the hamiltonian flow, that is the solution to Hamilton's system of equations
\begin{equation*}
\begin{cases} 
x'(t)  = \nabla_\xi a( x(t), \xi(t) ) \\ 
\xi'(t)  = -\nabla_x a( x(t), \xi(t) )  \\ 
x(0)=x  \\ 
\xi(0)=\xi
\end{cases} 
\end{equation*}
with $(x,\xi) \in T^* \rr d$. 
If $a(x,\xi) = \la (x,\xi), Q (x,\xi) \ra$ then the flow is linear as $\chi_t = e^{2 t \J Q}$ 
with
\begin{equation*}
\J =
\left(
\begin{array}{cc}
0 & I_d \\
- I_d & 0
\end{array}
\right) \in \rr {2d \times 2d}.
\end{equation*}

The Gabor wave front set $\WFg(u)$ of $u \in \cS'(\rr d)$ is the closed conic subset of $T^* \rr d \setminus \{ 0 \}$ in the complement 
of which each point has a conic neighborhood where the short-time Fourier transform 
$T^* \rr d \ni (x,\xi) \mapsto \cF\left( u T_x \fy \right)(\xi)$
decays super-polynomially.
Here $\fy \in \cS(\rr d)$ is a nonzero Schwartz function and $T_x$ denotes translation.

Our aim is to study equations with hamiltonian Weyl symbols that are more general than $a(x,\xi) = |x|^2 + |\xi|^2$. 
In particular we are interested in anisotropic Weyl symbols $a(x,\xi) = |x|^{2 k} + |\xi|^{2 m}$ with $k,m \in \no \setminus \{ 0 \}$
where particularly $m=1$ attracts interest in mathematical physics under the term anharmonic oscillator \cite{Bambusi1,Turbiner1}.
It is then natural to use an anisotropic version of the Gabor wave front set \cite{Rodino4,Wahlberg4}.

Then the cones in phase space are replaced by sets invariant under the anisotropic dilation
\begin{equation}\label{eq:anisodilation}
T^* \rr d \setminus \{ 0 \} \ni (x,\xi) \mapsto  (\lambda x, \lambda^\sigma \xi), \quad \lambda > 0,
\end{equation}
for a real anisotropy parameter $\sigma > 0$. The anisotropic Gabor wave front set $\WFgs(u)$ generalizes the Gabor wave front where $\sigma = 1$. 

In \cite{Cappiello6,Cappiello7} we have studied propagation of anisotropic Gabor wave front sets for correspondingly 
anisotropic hamiltonian Weyl symbols. 
In particular we have found a generalization of \eqref{eq:PropWFg} as
\begin{equation}\label{eq:PropWFgs}
\WFgs \left( u(t,\cdot)\right) = \chi_t \left(  \WFgs( u_0) \right), \quad t \in \ro, 
\end{equation}
where $\chi_t$ the hamilton flow corresponding to a hamiltonian Weyl symbol
\begin{equation}\label{eq:anisohampower}
a(x,\xi) = \left( |x|^{2 k} + |\xi|^{2 m} \right)^p,
\end{equation}
with critical exponent
\begin{equation}\label{eq:pcrit0}
p = p_c = \frac12 \left( \frac{1}{k} + \frac{1}{m}\right)
\end{equation}
and $\sigma = \frac{k}{m}$. 
We have also shown the non-propagation result 
$\WFgs \left( u(t,\cdot)\right) = \WFgs( u_0)$
when $0 < p < p_c$.

Similar to many other wave front sets the anisotropic Gabor wave front set can be defined as the intersection
\begin{equation}\label{eq:WFGaniso}
\WFgs(u) = \bigcap_{a^w(x,D) u \in \cS} \charac (a).
\end{equation}
Here $a$ is an anisotropic Shubin symbol
and the characteristic set $\charac (a)$ is the complement of points in $T^* \rr d \setminus \{ 0 \}$ in a neighborhood invariant as in \eqref{eq:anisodilation} of which $|a|$ is lower bounded by its weight
\begin{equation*}
(1 + |x| + |\xi|^{\frac1\sigma} )^r,
\end{equation*}
where $r \in \ro$ is the order, for large $|x| + |\xi|$. 

Thus regular subsets $\Gamma \subseteq T^* \rr d \setminus \left( \WFgs(u) \cup \{ 0\} \right)$ 
are open sets invariant as in \eqref{eq:anisodilation} such that there exists a symbol $a$ 
so that $a^w(x,D) u \in \cS$ and $a$ is elliptic in $\Gamma$ in the sense of 
\begin{equation*}
|a(x,\xi)| \geqs C(1 + |x| + |\xi|^{\frac1\sigma} )^r, \quad (x,\xi) \in \Gamma \setminus \rB_R
\end{equation*}
for some $C, R > 0$ where $\rB_R$ denotes a ball of radius $R$.

This paper introduces a new concept of global microlocal singularities for tempered distributions and positive rational anisotropy parameter. 
It is a \emph{filter} of subsets of phase space $T^* \rr d$ as opposed to notions of wave front sets that are closed subsets of phase space. 
In the definition of filter we imitate the regular subsets $\Gamma \subseteq T^* \rr d$ determined by the anisotropic wave front set above, but we relax the invariance \eqref{eq:anisodilation}. 
The filter of anisotropic Gabor singularities thus consists of subsets $T^* \rr d \setminus \Omega$ such that 
there exists a symbol $a$ so that $a^w(x,D) u \in \cS$ and $a$ is elliptic in $\Omega$ as
\begin{equation*}
|a(x,\xi)| \geqs C(1 + |x| + |\xi|^{\frac1\sigma} )^r, \quad (x,\xi) \in \Omega \setminus \rB_R
\end{equation*}
for some $C, R > 0$. 
Here a filter means a non-empty class of subsets closed under enlargement and finite intersection. 

The filter of global anisotropic singularities is more flexible than the anisotropic Gabor wave front set
and contains more information.
Indeed we may reconstruct the anisotropic Gabor wave front set as the intersection of filter sets that are closed and invariant as 
 \eqref{eq:anisodilation} (see Proposition \ref{prop:filterWF}).
It also allows us to relax the conditions of the power $p$ of the hamiltonian 
\eqref{eq:anisohampower} and still obtain propagation results corresponding to \eqref{eq:PropWFgs}. 

The filter of global anisotropic singularities is a relaxation of the filter of anisotropic singularities that we have studied in \cite{Cappiello7}. There the sets in the filter are restricted to have the anisotropic annular form 
\begin{equation*}
\{ (x,\xi) \in \rr {2d}: \ |x|^{2k} + |\xi|^{2m} \in \wt \Sigma \}
\end{equation*}
where $\wt \Sigma \subseteq \ro_+$. This filter is invariant under propagation for hamiltonian Weyl symbols of the form 
\eqref{eq:anisohampower} with $p = 1$ provided $\frac1k + \frac1m > \frac23$, cf. \cite[Theorem~1.5]{Cappiello7}. 

The idea to study singularities as filters goes back to 
\cite{Rodino1} where it is used mostly for the frequency variables (covariables), 
with further results in this direction in \cite{Garello1}, 
as opposed to our filters which are classes of subsets of phase space $T^* \rr d$.
Wave front sets anisotropic in the covariables only were introduced in \cite{Lascar1}.

Our main results are on one hand generalized versions of \cite[Theorem~8.3]{Cappiello6} and \cite[Theorem~1.4]{Cappiello7}
respectively
which we formulate for filters of singularities. See Theorems 6.2 and 6.5.
On the other hand we deduce a new result on propagation of filters of singularities for $d=1$
and hamiltonian Weyl symbols of the form \eqref{eq:anisohampower}
in the supercritical case $p_c < p \leqs p_c + \frac14 \min \left( \frac{1}{k}, \frac{1}{m} \right)$.
See Theorem \ref{thm:propagationfilter2}. 

As preparation for these results we devote a large part of the paper to the study of filters of anisotropic global singularities, which we hope can be of independent interest. 
A main result is the characterization of the filter by means of the short-time Fourier transform (see Theorem \ref{thm:characsing}). 

The paper is organized as follows. 
Section \ref{sec:prelim} specifies notation, and the framework of the short-time Fourier transform, 
pseudodifferential calculus with anisotropic Shubin symbols and corresponding modulation spaces,
and anisotropic Gabor wave front sets. 
Section \ref{sec:microlocalcutoff} introduces anisotropic neighborhoods in phase space and 
shows existence of microlocal excision symbols when the anisotropy parameter $\sigma > 0$ is rational. 
In Section \ref{sec:filtersing} we define filters of global anisotropic singularities and deduce some of their properties. 
Section \ref{sec:microlocmicroell} is devoted to notions of microlocality and microellipticity in this framework, 
and the final Section \ref{sec:schrodinger} to applications to propagation of singularities for evolution equations.

\section{Preliminaries}\label{sec:prelim}

The power set of a set $\Omega$ is denoted $\rP(\Omega)$.
The symbols $\ro_+$ and $\qo_+$ denote the positive real and rational numbers respectively, 
and the unit sphere in $\rr d$ is denoted $\sr {d-1} \subseteq \rr d$. 
An open ball of radius $r > 0$ centered in $x \in \rr d$ is denoted $\rB_r (x)$, and $\rB_r(0) = \rB_r$.  
We write $f (x) \lesssim g (x)$ provided there exists $C>0$ such that $f (x) \leqs C \, g(x)$ for all $x \in \dom (f) \cap \dom (g)$. 
If $f (x) \lesssim g (x) \lesssim f(x)$ then we write $f \asymp g$. 
We use the partial derivative $D_j = - i \partial_j$, $1 \leqs j \leqs d$, acting on functions and distributions on $\rr d$, 
with extension to multi-indices. 
The bracket $\eabs{x} = (1 + |x|^2)^{\frac12}$ for $x \in \rr d$ satisfies
Peetre's inequality \cite[Lemma~2.1]{Rodino3}, that is
\begin{equation}\label{eq:Peetre}
\eabs{x+y}^s \leqs \left( \frac{2}{\sqrt{3}} \right)^{|s|} \eabs{x}^s\eabs{y}^{|s|}\qquad x,y \in \rr d, \quad s \in \ro. 
\end{equation}
%

\subsection{Anisotropic weight functions}\label{subsec:weights}

For $\sigma \in \ro_+$ we use 
\begin{equation}\label{eq:quasitriangleineq}
|x + y|^{\frac1\sigma} \leqs C_\sigma ( |x|^{\frac1\sigma} + |y|^{\frac1\sigma}), \quad x,y \in \rr d, 
\end{equation}
where 
\begin{equation}\label{eq:qtrieqconstant}
C_\sigma= 
\left\{
\begin{array}{ll}
1 & \mbox{if} \ \sigma \geqs 1 \\
2^{\frac1\sigma-1} & \mbox{if} \ 0 < \sigma < 1
\end{array}
\right. .
\end{equation}

If $\sigma \in \ro_+$ the anisotropic weight function is
\begin{equation}\label{eq:weightanisotrop1}
\theta_\sigma(x,\xi) = 1 + |x| + |\xi|^{\frac1\sigma}, \quad x, \xi \in \rr d, 
\end{equation}
and if $k,m \in \no \setminus 0$ we set
\begin{equation}\label{eq:weightanisotrop2}
w_{k,m} (x, \xi) 
= \left( 1 + | x |^{2k} + | \xi |^{2 m} \right)^{\frac12}, \quad x , \xi \in \rr d. 
\end{equation}
When $k,m \in \no \setminus 0$ are given by the context we abbreviate 
$w = w_{k,m}$. 
If $k,m \in \no \setminus 0$ and $\sigma = \frac{k}{m}$ then
by \eqref{eq:quasitriangleineq} there exists $c_k \geqs 1$ such that 
\begin{equation}\label{eq:weightequivalence}
c_k^{-1} \theta_{\sigma}^{k} (x,\xi) \leqs w_{k,m} (x,\xi) \leqs \theta_{\sigma}^{k} (x,\xi), \quad x,\xi \in \rr d. 
\end{equation}
In fact by \eqref{eq:qtrieqconstant} we may take $c_k = 2^{2k-1}$. 
The weight $w_{k,m}$ is smooth everywhere as opposed to $\theta_\sigma$, in general. 

By \cite[Lemma~2.1]{Rodino4} we have the anisotropic Peetre inequality for any $\sigma \in \ro_+$
\begin{equation}\label{eq:Peetreaniso}
\theta_\sigma(X+Y)^s 
\lesssim  \theta_\sigma (X)^s \theta_\sigma (Y)^{|s|}, \qquad X,Y \in \rr {2d}, \quad s \in \ro. 
\end{equation}

We also need the following inequality.

\begin{lem}\label{lem:triangleineq}
If $\sigma \in \ro_+$ then there exists $B_\sigma > 0$ such that 
\begin{equation}\label{eq:triangle}
\theta_\sigma(x + y,\xi + \eta) 
\leqs B_\sigma \Big(  \theta_\sigma(x,\xi) + \theta_\sigma(y,\eta) \Big), \quad x,y,\xi,\eta \in \rr d.
\end{equation}
\end{lem}

\begin{proof}
This is an immediate consequence of the triangle inequality, \eqref{eq:quasitriangleineq} and \eqref{eq:qtrieqconstant}.
The constant may be chosen as $B_\sigma = 2^{\max \left(0, \frac1\sigma - 1\right)}$.
\end{proof}

The weight $\theta_\sigma$, $\sigma \in \ro_+$, is related to $\eabs{\cdot}$ as \cite[Eq.~(3.4)]{Rodino4}
\begin{equation}\label{eq:sobolevweightestimate1}
\eabs{(x,\xi)}^{\min\left( 1, \frac1\sigma\right)} \lesssim \theta_\sigma (x,\xi) \lesssim \eabs{(x,\xi)}^{\max\left( 1, \frac1\sigma\right)}, \quad (x,\xi) \in T^* \rr d. 
\end{equation}
%

\subsection{The short-time Fourier transform and a class of modulation spaces}\label{subsec:STFT}

The Fourier transform is 
\begin{equation*}
 \cF f (\xi )= \widehat f(\xi ) = (2\pi )^{-\frac d2} \int _{\rr
{d}} f(x)e^{-i\scal  x\xi } \, \dd x, \qquad \xi \in \rr d, 
\end{equation*}
for $f\in \cS(\rr d)$ (the Schwartz space), where $\scal \cdo \cdo$ denotes the scalar product on $\rr d$. 
The conjugate linear action of a distribution $u$ on a test function $\phi$ is written $(u,\phi)$, consistent with the $L^2$ inner product $(\cdo ,\cdo ) = (\cdo ,\cdo )_{L^2}$ which is conjugate linear in the second argument. 

Denote translation by $T_x f(y) = f( y-x )$ and modulation by $M_\xi f(y) = e^{i \scal y \xi} f(y)$ 
for $x,y,\xi \in \rr d$, where $f$ is a function or distribution defined on $\rr d$. 
Their composition is denoted $\Pi(x,\xi) = M_\xi T_x$ for $(x, \xi) \in \rr {2d}$. 
If $\fy \in \cS(\rr d) \setminus \{0\}$ (``window function'') then
the short-time Fourier transform (STFT) of a tempered distribution $u \in \cS'(\rr d)$ is defined as
\begin{equation*}
V_\fy u (x,\xi) = (2\pi )^{-\frac d2} (u, M_\xi T_x \fy) = \cF (u T_x \overline \fy)(\xi), \quad x,\xi \in \rr d. 
\end{equation*}
Then $V_\fy u \in C^\infty(\rr {2d})$ and by \cite[Theorem~11.2.3]{Grochenig1}
there exists $r \geqs 0$ such that 
\begin{equation}\label{eq:STFTtempered}
|V_\fy u (x,\xi)| \lesssim \eabs{(x,\xi)}^{r}, \quad (x,\xi) \in T^* \rr d.  
\end{equation}
We have $u \in \cS(\rr d)$ if and only if
\begin{equation}\label{eq:STFTschwartz}
|V_\fy u (x,\xi)| \lesssim \eabs{(x,\xi)}^{-N}, \quad (x,\xi) \in T^* \rr d, \quad \forall N \geqs 0.  
\end{equation}

The inverse transform to the STFT is given by
\begin{equation}\label{eq:STFTinverse}
u = (2\pi )^{-\frac d2} \| \fy \|_{L^2}^{-2} \iint_{\rr {2d}} V_\fy u (x,\xi) M_\xi T_x \fy \, \dd x \, \dd \xi
\end{equation}
with action under the integral understood, that is 
\begin{equation}\label{eq:moyal}
(u, f) = \| \fy \|_{L^2}^{-2} (V_\fy u, V_\fy f)_{L^2(\rr {2d})} = \| \fy \|_{L^2}^{-2} (V_\fy^* V_\fy u, f)
\end{equation}
for $u \in \cS'(\rr d)$ and $f \in \cS(\rr d)$, cf. \cite[Theorem~11.2.5]{Grochenig1}. 
With the same interpretation we write consistently for a polynomially bounded measurable function $F$
defined on $\rr {2d}$
\begin{equation}\label{eq:STFTadjoint}
V_\fy^* F = (2\pi )^{-\frac d2} \int_{\rr {2d}} F(z) \Pi(z) \fy \, \dd z
\end{equation}
which is an element in $\cS'(\rr d)$ \cite{Grochenig1}.

We use the following parametrized family of Hilbert modulation spaces \cite{Cappiello6,Grochenig1,Rodino4}.

\begin{defn}\label{def:Sobolevaniso}
Let $\fy \in \cS(\rr d) \setminus 0$. 
The anisotropic Shubin--Sobolev modulation space $M_{\sigma,s} (\rr d)$ with anisotropy parameter $\sigma \in \ro_+$ and order $s \in \ro$ is the Hilbert subspace of $\cS'(\rr d)$
defined by the norm
\begin{equation}\label{eq:SSmodnorm}
\| u \|_{M_{\sigma,s}} = \left( \iint_{\rr {2 d}} |V_\fy u (x,\xi)|^2 \, \theta_\sigma (x,\xi)^{2 s} \, \dd x \, \dd \xi \right)^{\frac12}. 
\end{equation}
\end{defn}

For any $\sigma \in \ro_+$ we have $M_{\sigma,0} (\rr d) = L^2(\rr d)$ \cite{Grochenig1}, 
and $M_{\sigma,s_1} (\rr d) \subseteq M_{\sigma,s_2}(\rr d)$ is a continuous inclusion when $s_1 \geqs s_2$. 
It holds
\begin{equation}\label{eq:SSSchwartz}
\cS(\rr d) = \bigcap_{s \in \ro} M_{\sigma,s} (\rr d), \quad \cS'(\rr d) = \bigcup_{s \in \ro} M_{\sigma,s} (\rr d),
\end{equation}
and $\{ \| \cdot \|_{M_{\sigma,s}}, s \geqs 0\}$ is a family of seminorms that defines the Fr\'echet space topology on $\cS(\rr d)$ \cite{Grochenig1}.

\subsection{Pseudodifferential calculus with anisotropic symbols}\label{subsec:psidiocalc}

In this paper we use anisotropic Shubin symbols defined as follows, generalizing \cite[Definition~3.1]{Rodino4}. 
Cf. also \cite{Chatzakou1}.

\begin{defn}\label{def:symbol}
Let $\sigma \in \ro_+$, $0 < \rho \leqs 1$ and $r \in \ro$. 
The space of ($\sigma$-)anisotropic Shubin symbols $G_\rho^{r,\sigma}$ of order $r$ and parameter $\rho$ consists of functions $a \in C^\infty(\rr {2d})$ 
that satisfy the estimates
\begin{equation}\label{eq:symbolderivative1}
|\pdd x \alpha \pdd \xi \beta a(x,\xi)|
\lesssim ( 1 + |x| + |\xi|^{\frac1\sigma} )^{r - \rho \left(  |\alpha| + \sigma |\beta| \right)}, \quad (x,\xi) \in T^* \rr d, \quad \alpha, \beta \in \nn d. 
\end{equation}
\end{defn}

The space $G_\rho^{r,\sigma}$ is a Fr\'echet space with respect to the seminorms on $a \in G^{r,\sigma}$ indexed by $j \in \no$
\begin{equation*}
\| a \|_j = \max_{|\alpha + \beta| \leqs j}
\sup_{(x,\xi) \in \rr {2d} } \theta_\sigma(x,\xi)^{- r + \rho \left( |\alpha| + \sigma |\beta| \right)} \left| \pdd x \alpha \pdd \xi \beta a(x,\xi ) \right|. 
\end{equation*}

If $\sigma = 1$ then $G_\rho^{r,\sigma} = G_\rho^r$ is the space of isotropic Shubin symbols with parameter $\rho$ \cite{Nicola1,Shubin1}.
Recall that the isotropic Shubin symbols of order $m$ and parameter $0 \leqs \rho \leqs 1$, denoted $a\in G_\rho^r$, satisfy
\begin{equation*}
|\partial_x^\alpha \partial_\xi^\beta a(x,\xi)| \lesssim \langle (x,\xi)\rangle^{r - \rho|\alpha + \beta|}, \quad (x,\xi) \in T^* \rr d, \quad \alpha, \beta \in \nn d.
\end{equation*}
We write $G_1^{r,\sigma} = G^{r,\sigma}$ and $G_1^r = G^r$. 

We have $G_\rho^{r,\sigma} \subseteq G_{\rho_0}^{r_0}$,  
where $r_0 = \max \left(r, \frac{r}{\sigma} \right)$ and $\rho_0 = \rho \min \left( \sigma, \frac1\sigma \right)$,  
and
\begin{equation}\label{eq:symbolintersection}
\bigcap_{r \in \ro} G_\rho^{r,\sigma} = \cS(\rr {2d}). 
\end{equation}

Given a sequence of symbols $a_j \in G_\rho^{r_j,\sigma}$, $j=1,2,\dots$, such that $r_j \to - \infty$ as $j \to \infty$
we write 
\begin{equation*}
a \sim \sum_{j = 1}^\infty a_j
\end{equation*}
provided that for any $n \geqs 2$
\begin{equation*}
a - \sum_{j = 1}^{n-1} a_j \in G_\rho^{t_n,\sigma}
\end{equation*}
where $t_n = \max_{j \geqs n} r_j$. 
By a natural generalization of the proof of \cite[Lemma~3.2]{Rodino4}, which treats the case $\rho = 1$, there exists a symbol
$a \in G_\rho^{r,\sigma}$ where $r = \max_{j \geqs 1} r_j$ such that $a \sim \sum_{j = 1}^\infty a_j$
under the stated assumptions. 
The symbol $a$ is unique modulo $\cS(\rr {2d})$. 

\begin{rem}\label{rem:WeylHormander}
The symbol class $G_\rho^{r,\sigma}$ defined in Definition \ref{def:symbol} may be considered 
as a particular case of more general classes $G_{\rho_1, \rho_2}^{r,\sigma}$ defined by estimates 
\begin{equation}\label{eq:symbolderivative2}
|\pdd x \alpha \pdd \xi \beta a(x,\xi)|
\lesssim ( 1 + |x| + |\xi|^{\frac1\sigma} )^{r - \rho_1  |\alpha| - \rho_2 |\beta|}, \quad (x,\xi) \in T^* \rr d, \quad \alpha, \beta \in \nn d.  
\end{equation}
Then $G_\rho^{r,\sigma} = G_{\rho, \rho \sigma}^{r,\sigma}$. 
The symbol class $G_{\rho_1, \rho_2}^{r,\sigma}$ may be seen 
as a Weyl--H\"ormander class $S(m,g)$ \cite[Chapter~18.4]{Hormander1}, \cite{Lerner1} with metric 
\begin{equation*}
g = \frac{\dd x^2}{ \theta_\sigma(x,\xi)^ {2 \rho_1} } + \frac{\dd \xi^2}{ \theta_\sigma(x,\xi)^{2 \rho_2} }
\end{equation*}
and weight $m(x,\xi) = \theta_\sigma (x,\xi)^r$.

Then $g$ is \emph{slowly varying} \cite[Definition~18.4.1]{Hormander1} and $m$ is $g$ continuous \cite[Definition~18.4.2]{Hormander1}
provided $\rho_1 \leqs 1$ and $\rho_2 \leqs \sigma$.
The \emph{Planck function} \cite[Eq.~18.4.19]{Hormander1} is 
\begin{equation*}
h(x,\xi) = \theta_\sigma(x,\xi)^{- \rho_1 - \rho_2}.
\end{equation*}
We have $g \leqs g^{\rm symp}$ (where $g^{\rm symp}$ is the symplectically dual metric cf. \cite{Lerner1}) if and only if $\rho_1 + \rho_2 \geqs 0$, 
and the \emph{strong uncertainty principle} \cite[Eq.~(1.1.10)]{Nicola1} is satisfied 
if and only if $\rho_1 + \rho_2 > 0$. Note that the conditions 
$\rho_1 \leqs 1$, $\rho_2 \leqs \sigma$ and $\rho_1 + \rho_2 > 0$
admit $\rho_1 < 0 < \rho_2$ or vice versa. 

Our symbol classes $G_\rho^{r,\sigma}$ have $\rho_1 + \rho_2 = \rho(1+\sigma) > 0$
due to the assumption $\rho > 0$. 
The strong uncertainty principle is hence valid. 
\end{rem}

For $a \in G_\rho^{r,\sigma}$ and $\tau \in \ro$ a pseudodifferential operator in the $\tau$-quantization is defined by
\begin{equation}\label{eq:tquantization}
a_\tau(x,D) f(x)
= (2\pi)^{-d}  \int_{\rr {2d}} e^{i \langle x-y, \xi \rangle} a ( (1-\tau) x + \tau y,\xi ) \, f(y) \, \dd y \, \dd \xi, \quad f \in \cS(\rr d),
\end{equation}
when $r < - d \sigma$. The definition extends to $r \in \ro$ if the integral is viewed as an oscillatory integral.
If $\tau = 0$ we get the Kohn--Nirenberg quantization $a_0(x,D) = a(x,D)$ and if $\tau = \frac12$ the Weyl quantization $a_{1/2}(x,D) = a^w(x,D)$. 
The Weyl quantization enjoys a simple formal adjoint relation: $a^w(x,D)^* = \overline{a}^w(x,D)$.  
For brevity we often write $a^w = a^w(x,D)$.

By a miniscule modification of the proof of \cite[Proposition~3.3 (i)]{Rodino4} the symbol classes $G_\rho^{r,\sigma}$ are homeomorphically invariant under change of quantization parameter $\tau \in \ro$, for any $\sigma \in \ro_+$, $r \in \ro$ and $0 \leqs \rho \leqs 1$. 
If $a \in G_\rho^{r,\sigma}$
then the operator $a^w(x,D)$ acts continuously on $\cS(\rr d)$ and extends uniquely by duality to a continuous operator on $\cS'(\rr d)$ \cite{Rodino4,Shubin1}. 
If $a \in \cS'(\rr {2d})$ then $a^w(x,D)$ extends to a continuous operator $a^w(x,D): \cS(\rr d) \to \cS'(\rr d)$, 
and if $a \in \cS(\rr {2d})$ then $a^w(x,D)$ is regularizing, in the sense that it is continuous $a^w(x,D): \cS'(\rr d) \to \cS(\rr d)$ with $\cS'(\rr d)$ equipped with the strong topology \cite{Cappiello5}. 

The bilinear Weyl product $a \wpr b$ of two symbols $a \in G_\rho^{r,\sigma}$ and $b \in G_\rho^{t,\sigma}$
is defined as the product of symbols corresponding to operator composition in the Weyl quantization: 
$( a \wpr b)^w(x,D) = a^w(x,D) b^w (x,D)$. 
Again by a small modification of the proof of \cite[Proposition~3.3 (ii)]{Rodino4} the Weyl product is continuous $\wpr: G_\rho^{r,\sigma} \times G_\rho^{t,\sigma} \to G_\rho^{r+t,\sigma}$ when $0 \leqs \rho \leqs 1$.
If $0 < \rho \leqs 1$ the asymptotic expansion formula for the Weyl product \cite{Hormander1,Shubin1} is
\begin{equation}\label{eq:calculuscomposition1}
a \wpr b(x,\xi) \sim \sum_{\alpha, \beta \geqs 0} \frac{(-1)^{|\beta|}}{\alpha! \beta!} \ 2^{-|\alpha+\beta|}
D_x^\beta \pdd \xi \alpha a(x,\xi) \, D_x^\alpha \pdd \xi \beta b(x,\xi). 
\end{equation}
If $a \in G_\rho^{r,\sigma}$ and $b \in G_\rho^{t,\sigma}$
then each term in the sum belongs to $G_\rho^{r + t - \rho(1+\sigma)|\alpha+\beta|,\sigma}$. 

To each symplectic matrix $\chi \in \Sp(d,\ro)$ is associated a unitary operator $\mu(\chi)$ on $L^2(\rr d)$, determined up to a complex factor of modulus one, such that
\begin{equation}\label{symplecticoperator}
\mu(\chi)^{-1} a^w(x,D) \, \mu(\chi) = (a \circ \chi)^w(x,D), \quad a \in \cS'(\rr {2d})
\end{equation}
(cf. \cite{Folland1,Hormander1}).
The operator $\mu(\chi)$ is a homeomorphism on $\cS$ and on $\cS'$.

The mapping $\Sp(d,\ro) \ni \chi \mapsto \mu(\chi)$ is called the \emph{metaplectic representation} \cite{Folland1}.
More precisely it is a representation of the so called $2$-fold covering group of $\Sp(d,\ro)$, which is called the metaplectic group 
and denoted $\Mp(d,\ro)$.
The metaplectic representation satisfies the homomorphism relation modulo a change of sign:
\begin{equation*}
\mu( \chi_1 \chi_2) = \pm \mu(\chi_1 ) \mu(\chi_2 ), \quad \chi_1, \chi_2 \in \Sp(d,\ro).
\end{equation*}

There is an asymptotic expansion formula connecting $\tau$-quantizations for different $\tau \in \ro$ \cite[Theorem~1.2.4]{Nicola1}. 
In fact let $0 \leqs \rho \leqs 1$, $\tau_1, \tau_2 \in \ro$, 
$a_{\tau_1} \in G_\rho^{r,\sigma}$ and $a_{\tau_1}(x,D) = a_{\tau_2}(x,D)$. 
Then $a_{\tau_2} \in G_\rho^{r,\sigma}$, and
provided $0 < \rho \leqs 1$ we have
\begin{equation}\label{eq:calculusquantization}
a_{\tau_2}(x,\xi) \sim \sum_{\alpha \geqs 0} \frac{ ( \tau_1 - \tau_2 )^{|\alpha|} }{\alpha!} 
D_x^\alpha \pdd \xi \alpha a_{\tau_1}(x,\xi). 
\end{equation}

For $\sigma \in \ro_+$ a $\sigma$-conic subset $\Omega \subseteq T^* \rr d \setminus 0$  is invariant as
\begin{equation}\label{eq:sigmaconic}
(x,\xi) \in \Omega 
\quad \Longrightarrow \quad 
( \lambda x, \lambda^\sigma \xi ) \in \Omega \quad \forall \lambda \in \ro_+. 
\end{equation}
If $\sigma = 1$ this reduces to conic subsets. 

By \cite[Definition~3.4 and Lemma~3.5]{Rodino4} (cf. also \cite[Remark~3.4]{Wahlberg4}) it is possible to construct $\sigma$-conic open subsets of given points in $T^* \rr d \setminus 0$, and corresponding cutoff functions.

Let $a \in G_\rho^{r,\sigma}$ for $r \in \ro$, $\sigma \in \ro_+$ and $0 \leqs \rho \leqs 1$. 
By \cite[Proposition~4.2]{Cappiello6} the map
\begin{equation}\label{eq:ShubinSobolevCont}
a^w(x,D) : M_{\sigma,s + r} (\rr d) \to M_{\sigma,s} (\rr d)
\end{equation}
is continuous for all $s \in \ro$. 
More precisely the proof of \cite[Proposition~4.2]{Cappiello6} treats only the case $\rho = 1$, but 
the argument extends to $0 \leqs \rho \leqs 1$. 
From \eqref{eq:ShubinSobolevCont}, \eqref{eq:SSSchwartz} and the ensuing phrase about seminorms on $\cS(\rr d)$ we obtain another proof of the continuity of the operator
\begin{equation}\label{eq:SchwartzCont}
a^w(x,D) : \cS (\rr d) \to \cS (\rr d), \quad a \in G_\rho^{r,\sigma}, \quad r \in \ro.
\end{equation}
%

\subsection{Anisotropic Gabor wave front sets}\label{subsec:WFgaboraniso}

For $\sigma \in \ro_+$ the anisotropic Gabor wave front set of $u \in \cS'(\rr d)$ is a closed $\sigma$-conic subset of $T^* \rr d \setminus 0$ defined as follows, cf. \cite[Definition~4.1]{Rodino4}. 
Let $\fy \in \cS(\rr d) \setminus 0$ be arbitrary. 
For $z_0 \in T^* \rr d \setminus 0$ we have $z_0 \notin \WFgs(u)$ if there exists an open set $U \subseteq T^* \rr d$ such that $z_0 \in U$ and 
\begin{equation*}
\sup_{(x,\xi) \in U, \ \lambda > 0} \lambda^N | V_\fy u( \lambda x, \lambda^\sigma \xi)| < + \infty \quad \forall N \geqs 0.
\end{equation*}
Thus the anisotropic Gabor wave front set is a closed subset of $T^* \rr d \setminus 0$ consisting of $\sigma$-conic curves 
$\ro_+ \ni \lambda \mapsto (\lambda x, \lambda^\sigma \xi)$ where the STFT does not decay superpolynomially. 
For $u \in \cS'(\rr d)$ we have $\WFgs(u) = \emptyset$ if and only if $u \in \cS(\rr d)$ \cite{Rodino4}. 
When $\sigma = 1$ then $\WFgs(u) = \WFg(u)$ which denotes the (isotropic) Gabor wave front set \cite{Hormander2,Rodino2}.

By \cite[Proposition~5.1]{Rodino4} we have the microlocal inclusion
\begin{equation}\label{eq:microlocalityWFs}
\WFgs( a^w(x,D) u) \subseteq \WFgs( u), \quad u \in \cS'(\rr d), 
\end{equation}
for $a \in G_\rho^{r,\sigma}$ for any $r \in \ro$, $0 < \rho \leqs 1$ and $\sigma \in \ro_+$. 
Under the same assumptions the microelliptic inclusion \cite[Theorem~6.4]{Rodino4} implies 
\begin{equation}\label{eq:microellipticityWFs}
\WFgs( u) \subseteq  \WFgs( a^w(x,D) u) \bigcup \charac_\sigma (a), 
\quad u \in \cS'(\rr d), 
\end{equation}
where the characteristic set $\charac_\sigma(a) \subseteq T^* \rr d \setminus 0$ of the symbol $a$ is defined as follows: 
$0 \neq z_0 \notin \charac_\sigma(a)$ if there exists an open $\sigma$-conic neighborhood $\Gamma \subseteq T^* \rr d \setminus 0$ and $C, R > 0$ such that $z_0 \in \Gamma$ and 
\begin{equation*}
|a(x,\xi)| \geqs C \theta_\sigma(x,\xi)^r, \quad (x,\xi) \in \Gamma \setminus \rB_R.
\end{equation*}
More precisely the quoted result \cite[Theorem~6.4]{Rodino4} treats the case $\rho = 1$ but the proof extends to $0 < \rho \leqs 1$.

\section{Microlocal excision symbols}\label{sec:microlocalcutoff}

Let $\emptyset \neq \Omega \subseteq \rr {2d}$, let 
$\sigma \in \ro_+$, 
let $0 < \rho \leqs 1$ and let $\ep > 0$. 
We use anisotropic open neighborhoods of $\Omega$ defined as
\begin{equation}\label{eq:anisoneighb}
\begin{aligned}
\Omega_{\rho,\ep} 
& = \{ (x, \xi) \in \rr {2d}: \exists (y,\eta) \in \Omega: |x - y| < \ep \theta_\sigma^\rho(y,\eta) \mbox{ and } |\xi - \eta| < \ep \theta_\sigma^{\rho \sigma}(y,\eta) \} \\
& = \bigcup_{(y,\eta) \in \Omega} \rB_{\ep \theta_\sigma^\rho(y,\eta)}(y) \times \rB_{\ep \theta_\sigma^{\rho \sigma}(y,\eta)} (\eta).
\end{aligned}
\end{equation}
The dependence of $\Omega_{\rho,\ep}$ on $\sigma$ is thus suppressed in the notation for simplicity. 
When we need it we write $\Omega_{\rho,\ep}^\sigma$. 

If $\Omega$ is $\sigma$-conic, cf. \eqref{eq:sigmaconic}, 
then we have the invariance
\begin{align*}
(x,\xi) \in \Omega_{\rho,\ep} 
\quad \Longrightarrow \quad 
 ( \lambda x, \lambda^\sigma \xi ) \in \Omega_{\rho, \ep} \quad \forall \lambda \in ( 0, 1].
\end{align*}
Thus $\Omega_{\rho,\ep}$ is $\sigma$-conic ``inwards''.

First we show a few results that will be useful. 

\begin{lem}\label{lem:boundaniso}
Let $\sigma \in \ro_+$, $0 < \rho \leqs 1$ and define $C_\sigma$ by \eqref{eq:qtrieqconstant}. 
Suppose $\ep > 0$ satisfies
\begin{align}
& \ep + \ep^{\frac1\sigma} C_\sigma < 1 \label{eq:cond1ep},  \\
& \frac{\ep}{( 1 - \ep )^\sigma} C_{\frac1\sigma}^2 < 1 \label{eq:cond2ep}, 
\end{align}
and suppose $x,y,\xi,\eta \in \rr d$ satisfy 
\begin{equation}\label{eq:neibhassump}
|x - y| < \ep \theta_\sigma^\rho(y,\eta) 
\quad \mbox{ and } \quad
|\xi - \eta| < \ep \theta_\sigma^{\rho \sigma}(y,\eta). 
\end{equation}
Then there exists $C = C_{\sigma,\ep} \geqs 1$ such that 
\begin{equation}\label{eq:boundyeta}
\theta_\sigma (y, \eta) \leqs C \theta_\sigma(x,\xi). 
\end{equation}
\end{lem}

\begin{proof}
First we combine the assumption \eqref{eq:neibhassump} and $\rho \leqs 1$ 
which gives
\begin{align}
|y| & \leqs |x| + \ep ( 1 + |y| + |\eta|^{\frac1\sigma}), \label{eq:boundy1}\\
|\eta| & \leqs |\xi| + \ep ( 1 + |y| + |\eta|^{\frac1\sigma})^{\sigma}. \label{eq:boundeta1}
\end{align}
From \eqref{eq:quasitriangleineq} and \eqref{eq:boundeta1} we obtain 
\begin{equation*}
\ep |\eta|^{\frac1\sigma} 
\leqs \ep C_\sigma \left( |\xi|^{\frac1\sigma} + \ep^{\frac1\sigma} ( 1 + |y| + |\eta|^{\frac1\sigma}) \right)
\end{equation*}
that is 
\begin{equation*}
\ep |\eta|^{\frac1\sigma} \left( 1 - \ep^{\frac1\sigma} C_\sigma \right)
\leqs \ep C_\sigma \left( |\xi|^{\frac1\sigma} + \ep^{\frac1\sigma} ( 1 + |y| ) \right)
\end{equation*}
where $1 - \ep^{\frac1\sigma} C_\sigma > 0$ is a consequence of the assumption \eqref{eq:cond1ep}.
Insertion into \eqref{eq:boundy1} yields
\begin{equation}\label{eq:boundy2}
|y| \left( 1 - \ep - \frac{\ep^{1 + \frac1\sigma} C_\sigma}{1 - \ep^{\frac1\sigma} C_\sigma} \right)
\leqs 
|x| + \ep + 
\frac{\ep C_\sigma}{1 - \ep^{\frac1\sigma} C_\sigma} | \xi |^{\frac1\sigma} 
+ \frac{\ep^{1 + \frac1\sigma} C_\sigma}{1 - \ep^{\frac1\sigma}C_\sigma}
\end{equation}
where $1 - \ep - \ep^{1 + \frac1\sigma} C_\sigma (1 - \ep^{\frac1\sigma} C_\sigma)^{-1} > 0$ again follows from \eqref{eq:cond1ep}. 
  
Likewise \eqref{eq:boundy1} gives $|y| (1 - \ep  ) \leqs |x| + \ep ( 1 + |\eta|^{\frac1\sigma})$ with $1-\ep > 0$ due to \eqref{eq:cond1ep}. 
We insert into \eqref{eq:boundeta1} and use \eqref{eq:quasitriangleineq} to infer
\begin{align*}
|\eta| & \leqs |\xi| + \ep \left( 1 + \frac{1}{1-\ep} \left( |x| + \ep ( 1 + |\eta|^{\frac1\sigma}) \right) + |\eta|^{\frac1\sigma} \right)^\sigma \\
& = |\xi| + \frac{\ep}{(1-\ep)^\sigma} \left( 1 + |x| + |\eta|^{\frac1\sigma} \right)^\sigma \\
& \leqs |\xi| + \frac{\ep}{(1-\ep)^\sigma} C_{\frac1\sigma}^2 \left( 1 + |x|^\sigma + |\eta| \right)
\end{align*}
which finally gives, again using \eqref{eq:quasitriangleineq}
\begin{equation}\label{eq:boundeta2}
\begin{aligned}
|\eta|^{\frac1\sigma} \left( 1 - \frac{\ep}{(1-\ep)^\sigma} C_{\frac1\sigma}^2 \right)^{\frac1\sigma}
& \leqs C_\sigma \left( |\xi|^{\frac1\sigma} + \frac{\ep^{\frac1\sigma}}{1-\ep} C_{\frac1\sigma}^{\frac2\sigma} \left( 1 + |x|^\sigma \right)^{\frac1\sigma} \right) \\
& \leqs C_\sigma \left( |\xi|^{\frac1\sigma} + \frac{\ep^{\frac1\sigma}}{1-\ep} C_{\frac1\sigma}^{\frac2\sigma} C_\sigma \left( 1 + |x| \right) \right). 
\end{aligned}
\end{equation}
We have $1 - \frac{\ep}{(1-\ep)^\sigma} C_{\frac1\sigma}^2 > 0$ due to assumption \eqref{eq:cond2ep}. 
The claim \eqref{eq:boundyeta} now follows from \eqref{eq:boundy2} and \eqref{eq:boundeta2}. 
\end{proof}

\begin{rem}\label{rem:constant}
It follows from the proof that the constant $C_{\sigma,\ep} > 0$ decreases as $\ep > 0$ decreases. 
\end{rem}

\begin{rem}\label{rem:lowerboundsmall}
There exists a constant $0 < b_\sigma < 1$ that depends on $\sigma$ 
such that $\ep < b_\sigma$ implies that \eqref{eq:cond1ep} and \eqref{eq:cond2ep}
are fulfilled. 
Indeed it suffices to take (cf. \eqref{eq:qtrieqconstant})
\begin{equation}\label{eq:bsigma}
b_\sigma = \left( 1 + \max \left( C_\sigma, C_\frac1\sigma^\frac2\sigma \right) \right)^{- \max(1,\sigma)}
= 
\left\{
\begin{array}{ll}
\left( 1 + 2^{\frac1\sigma-1} \right)^{-1} & \mbox{if} \ \  0 < \sigma \leqs 1 \\
\left( 1 + 2^{2 \left( 1 - \frac1\sigma \right)} \right)^{-\sigma} & \mbox{if} \ \ \sigma > 1
\end{array}
\right. .
\end{equation}
\end{rem}

\begin{lem}\label{lem:closureinclusion}
Suppose $\emptyset \neq \Omega \subseteq \rr {2d}$,  $\sigma \in \ro_+$, 
$0 < \rho \leqs 1$, and let $0 < b_\sigma < 1$ be defined by \eqref{eq:bsigma}. 
If $0 < \ep < \delta < 1$
and $\ep < b_\sigma$ then 
\begin{equation}\label{eq:closincl}
\overline{\left( \Omega_{\rho,\ep} \right)}
\subseteq \Omega_{\rho,\delta}.
\end{equation}
\end{lem}

\begin{proof}
Let $0 < \ep < \mu < \delta < 1$. 
First we show 
\begin{equation}\label{eq:closincl1}
\overline{\left( \Omega_{\rho,\ep} \right)}
\subseteq \left( \overline{\Omega} \right)_{\rho,\mu}.
\end{equation}
If $(x,\xi) \in \overline{\left( \Omega_{\rho,\ep} \right)}$ then there exists a sequence 
$(x_n,\xi_n)_{n \in \no} \subseteq \Omega_{\rho,\ep}$ such that $\lim_{n \to \infty} (x_n, \xi_n) = (x,\xi)$. 
For each $n \in \no$ there exists $(y_n, \eta_n) \in \Omega$ such that 
\begin{equation}\label{eq:ballsn}
|x_n - y_n| < \ep \theta_\sigma^\rho(y_n,\eta_n)
\quad \mbox{and} \quad 
|\xi_n - \eta_n| < \ep \theta_\sigma^{\rho \sigma}(y_n,\eta_n). 
\end{equation}

If $0 < \ep < b_\sigma$ then it follows from Lemma \ref{lem:boundaniso} and Remark \ref{rem:lowerboundsmall} that 
\begin{equation*}
\theta_\sigma (y_n, \eta_n) \leqs C_1 \theta_\sigma(x_n,\xi_n) \leqs C_2
\end{equation*}
for some $C_1, C_2 > 0$, in the second inequality using the boundedness of $(x_n,\xi_n)_{n \in \no}$.
Hence also the sequence $(y_n,\eta_n)_{n \in \no}$ is bounded. 
A subsequence, still denoted $(y_n,\eta_n)_{n \in \no}$, then converges, that is 
$\lim_{n \to \infty} (y_n, \eta_n) = (y,\eta) \in \overline{\Omega}$.

It follows that 
\begin{align*}
\frac{|x - y|}{\theta_\sigma^\rho(y,\eta)}
& = \limsup_{n \to \infty} \frac{|x_n - y_n|}{\theta_\sigma^\rho(y_n,\eta_n)}
\leqs \ep < \mu, \\
\frac{|\xi - \eta|}{\theta_\sigma^{\rho \sigma}(y,\eta)}
& = \limsup_{n \to \infty} \frac{|\xi_n - \eta_n|}{\theta_\sigma^{\rho \sigma}(y_n,\eta_n)}
\leqs \ep < \mu, 
\end{align*}
which shows that $(x,\xi) \in \left( \overline{\Omega} \right)_{\rho,\mu}$.
We have shown \eqref{eq:closincl1}.

To show \eqref{eq:closincl} it remains to show
\begin{equation}\label{eq:closincl2}
\left( \overline{\Omega} \right)_{\rho,\mu} 
\subseteq \Omega_{\rho,\delta}.
\end{equation}
If $(x,\xi) \in \left( \overline{\Omega} \right)_{\rho,\mu}$ then
there exists $(y,\eta) \in \overline \Omega$ such that 
\begin{equation*}
|x - y| < \mu \theta_\sigma^\rho(y,\eta)
\quad \mbox{and} \quad 
|\xi - \eta| < \mu \theta_\sigma^{\rho \sigma}(y,\eta). 
\end{equation*}
Due to the continuity of $\theta_\sigma$ and the assumption $\mu < \delta$ there exists $(y',\eta') \in \Omega$
such that 
\begin{equation*}
|x - y'| < \delta \theta_\sigma^\rho(y',\eta')
\quad \mbox{and} \quad 
|\xi - \eta'| < \delta \theta_\sigma^{\rho \sigma}(y',\eta') 
\end{equation*}
which shows that $(x,\xi) \in \Omega_{\rho,\delta}$. 
This proves \eqref{eq:closincl2}.
\end{proof}

\begin{prop}\label{prop:disjoint}
Suppose $\emptyset \neq \Omega \subseteq \rr {2d}$, $\sigma \in \ro_+$, $0 < \rho \leqs 1$ and $0 < \ep < \delta < 1$. 
There exists $\mu > 0$ such that 
\begin{equation}\label{eq:emptyintersection1}
\left( \rr {2d} \setminus \Omega_{\rho,\delta} \right)_{\rho,\mu} \bigcap \left( \Omega_{\rho,\ep} \right)_{\rho,\mu} = \emptyset.
\end{equation}
\end{prop}

\begin{proof}
Let $\ep < \gamma < \delta$. 
First we show that there exists $\mu_0 > 0$ such that
\begin{equation}\label{eq:doubleinclusion}
\left( \Omega_{\rho,\ep} \right)_{\rho,\mu} \subseteq \Omega_{\rho,\gamma}
\end{equation}
provided $\mu \leqs \mu_0$. 

Let $(x,\xi) \in \left( \Omega_{\rho,\ep} \right)_{\rho,\mu}$. Then there exists $(y,\eta) \in \Omega_{\rho,\ep}$ such that 
\begin{align}
|x - y| & < \mu \theta_\sigma^\rho(y,\eta) \label{eq:neibhassump1a}, \\
|\xi - \eta| & < \mu \theta_\sigma^{\rho \sigma}(y,\eta) \label{eq:neibhassump1b}, 
\end{align}
and there exists $(z,\nu) \in \Omega$ such that 
\begin{align}
|y-z| & < \ep \theta_\sigma^{\rho}(z,\nu) \label{eq:neibhassump2a}, \\
|\eta - \nu| & < \ep \theta_\sigma^{\rho \sigma}(z,\nu). \label{eq:neibhassump2b}
\end{align}

We obtain from \eqref{eq:neibhassump2a} and \eqref{eq:neibhassump2b}
\begin{equation*}
\theta_\sigma(y-z,\eta-\nu) 
= 1 + |y-z| + |\eta-\nu|^{\frac1\sigma}
< \theta_\sigma^\rho(z,\nu) \left( 1+ \ep + \ep^{\frac1\sigma} \right). 
\end{equation*}
From this we deduce by means of \eqref{eq:neibhassump1a}, \eqref{eq:neibhassump2a}, 
\eqref{eq:quasitriangleineq}, 
Lemma \ref{lem:triangleineq} and $\rho \leqs 1$
\begin{align*}
|x - z| & < \ep \theta_\sigma^\rho (z,\nu) + \mu \theta_\sigma^\rho (y,\eta) \\
& \leqs \ep \theta_\sigma^\rho (z,\nu) + \mu B_\sigma^\rho \Big( \theta_\sigma(z,\nu) + \theta_\sigma(y-z,\eta-\nu) \Big)^\rho \\
& \leqs \ep \theta_\sigma^\rho (z,\nu) + \mu B_\sigma^\rho C_{\frac1\rho} \Big( \theta_\sigma^\rho(z,\nu) + \theta_\sigma(y-z,\eta-\nu) \Big) \\
& \leqs \theta_\sigma^\rho (z,\nu) \left( \ep + \mu B_\sigma^\rho C_{\frac1\rho} \Big( 2 + \ep + \ep^{\frac1\sigma} \Big) \right) \\
& < \gamma \theta_\sigma^\rho (z,\nu)
\end{align*}
provided $\mu \leqs \mu_0$ for $\mu_0$ sufficiently small, where we use $\gamma > \ep$.

Likewise we infer starting from \eqref{eq:neibhassump1b} and \eqref{eq:neibhassump2b}
\begin{align*}
|\xi - \nu| & < \ep \theta_\sigma^{\rho \sigma} (z,\nu) + \mu \theta_\sigma^{\rho \sigma}(y,\eta) \\
& \leqs \ep \theta_\sigma^{\rho \sigma} (z,\nu) + \mu B_\sigma^{\rho \sigma} \Big( \theta_\sigma(z,\nu) + \theta_\sigma(y-z,\eta -\nu)\Big)^{\rho \sigma} \\
& \leqs \ep \theta_\sigma^{\rho \sigma} (z,\nu) + \mu B_\sigma^{\rho \sigma} C_{\frac{1}{\rho \sigma}} \Big( \theta_\sigma^{\rho \sigma}(z,\nu) + \theta_\sigma^\sigma (y-z,\eta -\nu)\Big) \\
& \leqs \ep \theta_\sigma^{\rho \sigma} (z,\nu) + \mu B_\sigma^{\rho \sigma} C_{\frac{1}{\rho \sigma}} \Big( \theta_\sigma^{\rho \sigma}(z,\nu) + 
\theta_\sigma^{\rho \sigma}(z,\nu) \left( 1+ \ep + \ep^{\frac1\sigma} \right)^\sigma \Big) \\
& = \theta_\sigma^{\rho \sigma} (z,\nu) \left( \ep + \mu B_\sigma^{\rho \sigma} C_{\frac{1}{\rho \sigma}} \left( 1 + \left( 1+ \ep + \ep^{\frac1\sigma} \right)^\sigma \right) \right) \\
& < \gamma \theta_\sigma^{\rho \sigma}  (z,\nu)
\end{align*}
again provided $\mu \leqs \mu_0$ for $\mu_0$ sufficiently small, using $\gamma > \ep$.
We have shown the inclusion \eqref{eq:doubleinclusion} for $\mu \leqs \mu_0$ and $\mu_0$ sufficiently small. 

Let $\mu \leqs \mu_0$. 
To prove the claim it suffices to show 
\begin{equation*}
\left( \rr {2d} \setminus \Omega_{\rho,\delta} \right)_{\rho,\mu} \bigcap \Omega_{\rho,\gamma} = \emptyset
\end{equation*}
or equivalently 
\begin{equation}\label{eq:inclusionclaim1}
\Omega_{\rho,\gamma} \subseteq \rr {2d} \setminus \left( \rr {2d} \setminus \Omega_{\rho,\delta} \right)_{\rho,\mu} 
\end{equation}
for some small $\mu > 0$.
We show this inclusion by contradiction. Thus suppose that there exists $(x,\xi) \in \Omega_{\rho,\gamma}$ 
such that
\begin{equation*}
(x,\xi) \in \left( \rr {2d} \setminus \Omega_{\rho,\delta} \right)_{\rho,\frac1n}
\end{equation*}
for all $n \in \no \setminus 0$.
Then for all $n \in \no \setminus 0$ there exists $(y_n,\eta_n) \in \rr {2d} \setminus \Omega_{\rho,\delta}$ such that 
\begin{equation*}
|x - y_n| < \frac1n \theta_\sigma^\rho (y_n,\eta_n) 
\quad \mbox{and} \quad
|\xi - \eta_n| < \frac1n \theta_\sigma^{\rho \sigma} (y_n,\eta_n).  
\end{equation*}

By Lemma \ref{lem:boundaniso} we have $\theta_\sigma (y_n,\eta_n) \leqs C \theta_\sigma (x,\xi)$
if $n \geqs N$ for $N \in \no$ sufficiently large, where $C > 0$ does not depend on $n$, cf. Remark \ref{rem:constant}. 
Hence
\begin{align*}
|x - y_n| & < \frac1n C^\rho \theta_\sigma^\rho (x,\xi) , \\
|\xi - \eta_n| & < \frac1n C^{\rho \sigma} \theta_\sigma^{\rho \sigma} (x,\xi), 
\end{align*}
and it follows that 
\begin{equation*}
(x,\xi) = \lim_{n \to \infty} (y_n, \eta_n).
\end{equation*}

The set $\rr {2d} \setminus \Omega_{\rho,\delta}$ is closed in $\rr {2d}$ so we have $(x,\xi) \in \rr {2d} \setminus \Omega_{\rho,\delta}$.
Since $(x,\xi) \in \Omega_{\rho,\gamma} \subseteq \Omega_{\rho,\delta}$ due to $\gamma < \delta$ we get a contradiction. 
It follows that \eqref{eq:inclusionclaim1} must hold for some $\mu > 0$.  
\end{proof}

In the sequel we introduce cutoff symbols in $G_\rho^{0,\sigma}$ adapted to neighborhoods of the form \eqref{eq:anisoneighb}.
Here we restrict to rational anisotropy parameters $\sigma \in \qo_+$. 
Let $\psi \in C_c^\infty(\ro)$ satisfy $\psi \geqs 0$, $\supp \psi \subseteq [-\frac14,\frac14]$ and $\int_{\rr d} \psi(|x|^2) \dd x = 1$. 
Then $\fy(x) = \psi(|x|^2)$ for $x \in \rr d$ satisfies $\fy \in C_c^\infty(\rr d)$ , $\fy \geqs 0$, $\supp \fy \subseteq \rB_{\frac12}$
and $\int_{\rr d} \fy (x) \dd x = 1$. 

Given $\emptyset \neq \Omega \subseteq \rr {2d}$, $\sigma \in \qo_+$, $0 < \rho \leqs 1$ and $0 < \ep < \delta < 1$.
Then $\sigma = \frac{k}{m}$ with $k,m \in \no \setminus 0$ coprime. 
Take $0 < \mu \leqs 1$ so that \eqref{eq:emptyintersection1} is satisfied which is possible by Proposition \ref{prop:disjoint}. 
We define the function $q_{\ep,\delta,\rho,\Omega}: \rr {2d} \to \ro$ as 
\begin{equation}\label{eq:cutoffsymboldef}
\begin{aligned}
& q_{\ep,\delta,\rho,\Omega}(x,\xi) \\ 
& = \mu^{-2d} w^{- d \rho \left( \frac1k + \frac1m \right) }(x,\xi) 
\iint_{\left( \Omega_{\rho,\ep} \right)_{\rho,\mu} }
\psi \left( \left| \frac{x-y}{\mu w^\frac{\rho}{k} (x,\xi)} \right|^2 \right) 
\psi \left( \left| \frac{\xi-\eta}{\mu w^\frac{\rho}{m} (x,\xi)} \right|^2 \right) 
\dd y \, \dd \eta
\end{aligned}
\end{equation}
where $w = w_{k,m}$ is defined by \eqref{eq:weightanisotrop2}.
For simplicity of notation we here suppress the dependence on $\mu$ and $\sigma$, that is  
$q_{\ep,\delta,\rho,\Omega} = q_{\ep,\delta,\rho,\sigma,\mu,\Omega}$, unless it is explicitly needed. 

\begin{prop}\label{prop:cutoffsymbol}
Suppose $\emptyset \neq \Omega \subseteq \rr {2d}$, 
$\sigma \in \qo_+$, 
$0 < \rho \leqs 1$ and $0 < \ep < \delta < 1$. 
If $q = q_{\ep,\delta,\rho,\Omega}$ is defined by \eqref{eq:cutoffsymboldef} 
with $\sigma = \frac{k}{m}$ and $k,m \in \no \setminus 0$ coprime,
then 
$0 \leqs q \leqs 1$, 
$q \in G_\rho^{0,\sigma}$, 
$q|_{\rr {2d} \setminus \Omega_{\rho,\delta}} \equiv 0$,
$\supp q \subseteq \overline{\left(\Omega_{\rho,\delta} \right)}$ and $q|_{\Omega_{\rho,\ep}} \equiv 1$. 
\end{prop}

\begin{proof}
The claim $q \geqs 0$ everywhere is an immediate consequence of the construction. 
By a change of variables we also have for all $(x,\xi) \in \rr {2d}$
\begin{align*}
q(x,\xi) 
& \leqs 
\mu^{-2d} w^{- d \rho \left( \frac1k + \frac1m \right)}(x,\xi) \iint_{\rr {2d}} \psi \left( \left| \frac{x-y}{\mu w^\frac{\rho}{k} (x,\xi)} \right|^2 \right) \psi \left( \left| \frac{\xi-\eta}{\mu w^\frac{\rho}{m} (x,\xi)} \right|^2 \right) \dd y \, \dd \eta \\
& = \left( \int_{\rr d} \psi(|y|^2) \dd y \right)^2 = 1
\end{align*}
which is the claimed upper bound. 

Let $(x,\xi) \in \Omega_{\rho,\ep}$. The function
\begin{equation}\label{eq:integrand1}
\rr {2d} \ni (y,\eta) \mapsto 
\psi \left( \left| \frac{x-y}{\mu w^\frac{\rho}{k} (x,\xi)} \right|^2 \right) \psi \left( \left| \frac{\xi-\eta}{\mu w^\frac{\rho}{m} (x,\xi)} \right|^2 \right)
\end{equation}
is supported in $\overline{\rB_{\frac12 \mu w^\frac{\rho}{k} (x,\xi)} (x)} \times \overline{\rB_{\frac12 \mu w^\frac{\rho}{m} (x,\xi)} (\xi)}\subseteq \left( \Omega_{\rho,\ep} \right)_{\rho,\mu}$
due to \eqref{eq:weightequivalence} and \eqref{eq:anisoneighb}. 
It follows that 
\begin{equation*}
q(x,\xi) 
= \mu^{-2d} w^{- d \rho \left( \frac1k + \frac1m \right)}(x,\xi) \iint_{\rr {2d}}
\psi \left( \left| \frac{x-y}{\mu w^\frac{\rho}{k} (x,\xi)} \right|^2 \right) \psi \left( \left| \frac{\xi-\eta}{\mu w^\frac{\rho}{m} (x,\xi)} \right|^2 \right) \dd y \dd \eta 
= 1
\end{equation*}
as above, which proves the claim $q|_{\Omega_{\rho,\ep}} \equiv 1$. 

Next let 
$(x,\xi) \in \rr {2d} \setminus \Omega_{\rho,\delta}$. The integrand \eqref{eq:integrand1} is then supported in 
$\overline{\rB_{\frac12 \mu w^\frac{\rho}{k} (x,\xi)} (x)} \times \overline{\rB_{\frac12 \mu w^\frac{\rho}{m} (x,\xi)} (\xi)}\subseteq
\left( \rr {2d} \setminus \Omega_{\rho,\delta} \right)_{\rho,\mu}$. 
From \eqref{eq:emptyintersection1} it follows that $q(x,\xi) = 0$ which shows the claims 
$q|_{\rr {2d} \setminus \Omega_{\rho,\delta}} \equiv 0$
and $\supp q \subseteq \overline{\left(\Omega_{\rho,\delta} \right)}$. 

It remains to prove $q \in G_\rho^{0,\sigma}$. 
According to \cite[Lemma~3.6]{Cappiello6} we have $w^{- d \rho \left( \frac1k + \frac1m \right)} \in G^{- d \rho (1+\sigma),\sigma} \subseteq G_\rho^{- d \rho (1+\sigma),\sigma}$ so it suffices to show 
$a \in G_\rho^{d \rho (1+\sigma),\sigma}$ where 
\begin{equation}\label{eq:adef}
a(x,\xi) = \iint_{\left( \Omega_{\rho,\ep} \right)_{\rho,\mu} } 
\psi \left( \left| \frac{x-y}{\mu w^\frac{\rho}{k} (x,\xi)} \right|^2 \right) \psi \left( \left| \frac{\xi-\eta}{\mu w^\frac{\rho}{m} (x,\xi)} \right|^2 \right) \dd y \, \dd \eta.  
\end{equation}

Leibniz' rule yields
\begin{align*}
& \left| \partial_{x,\xi}^{\alpha,\beta} \left(w^{-\frac{2 \rho}{k} }(x,\xi) | x - y |^2 \right) \right|
= \left| \sum_{\gamma \leqs \alpha, \, |\gamma| \leqs 2} \binom{\alpha}{\gamma} 
\partial_{x,\xi}^{\alpha-\gamma ,\beta} \left(w^{-\frac{2 \rho}{k}}(x,\xi) \right) \partial_x^\gamma \left( | x - y |^2 \right) \right| \\
& \lesssim \sum_{\gamma \leqs \alpha, \, |\gamma| \leqs 2} \binom{\alpha}{\gamma} 
\left| \partial_{x,\xi}^{\alpha-\gamma ,\beta} \left(w^{-\frac{2 \rho}{k}}(x,\xi) \right) \right| \left( \frac{| x - y |}{\mu w^\frac{\rho}{k}(x,\xi) }  \right)^{2 - |\gamma|} 
w^{\frac{\rho}{k} \left(2 - |\gamma| \right)} (x,\xi). 
\end{align*}
Again by \cite[Lemma~3.6]{Cappiello6} we have $w^{-\frac{2 \rho}{k}} \in G_\rho^{-2 \rho,\sigma}$ which inserted above 
gives the estimate 
\begin{equation}\label{eq:derivativeestimate1}
\left| \partial_{x,\xi}^{\alpha,\beta} \left(w^{ -\frac{2 \rho}{k} }(x,\xi) | x - y |^2 \right) \right|
\lesssim 
\theta_\sigma(x,\xi)^{- \rho \left(|\alpha| + \sigma |\beta| \right)} 
\end{equation}
provided that
\begin{equation}\label{eq:factorbound1}
| x - y | \leqs \mu w^\frac{\rho}{k}(x,\xi).
\end{equation}

Likewise we have 
\begin{align*}
& \left| \partial_{x,\xi}^{\alpha,\beta} \left(w^{-\frac{2 \rho}{m}}(x,\xi) | \xi - \eta |^2 \right) \right| \\
& \lesssim \sum_{\gamma \leqs \beta, \, |\gamma| \leqs 2} \binom{\beta}{\gamma} 
\left| \partial_{x,\xi}^{\alpha,\beta -\gamma } \left(w^{-\frac{2 \rho}{m}}(x,\xi) \right) \right| \left( \frac{| \xi - \eta |}{\mu w^\frac{\rho}{m} (x,\xi) }  \right)^{2-|\gamma|} 
w^{\frac{\rho}{m} \left(2-|\gamma| \right)} (x,\xi)
\end{align*}
which gives the estimate 
\begin{equation}\label{eq:derivativeestimate2}
\left| \partial_{x,\xi}^{\alpha,\beta} \left(w^{-\frac{2 \rho}{m}}(x,\xi) | \xi - \eta |^2 \right) \right|
\lesssim 
\theta_\sigma(x,\xi)^{- \rho \left(|\alpha| + \sigma |\beta| \right)} 
\end{equation}
if
\begin{equation}\label{eq:factorbound2}
| \xi - \eta | \leqs \mu w^\frac{\rho}{m}(x,\xi).
\end{equation}

Next we invoke Fa\` a di Bruno's formula in the following form (see \cite[Eq. (2.3)]{Gramchev1}).  
If $\psi: \ro \to \ro$ and $g: \rr {2d} \to \ro$
with $\psi \in C^\infty(\ro)$ and $g \in C^\infty(\rr {2d})$ then for any $(\alpha,\beta) \in (\nn d \oplus \nn d) \setminus \{ 0 \}$
we have 
\begin{equation}\label{eq:faadibruno1}
\begin{aligned}
\partial_{x,\xi}^{\alpha,\beta} \big( \psi (g(x,\xi) ) \big)
= \sum_{j = 1}^{|\alpha+\beta|} \frac{ \psi^{(j)}( g(x,\xi) )}{j!} \sum_{ \stackrel[ | \alpha_\ell + \beta_\ell | \geqs 1, \ 1 \leqs \ell \leqs j ]{ \alpha_1 + \cdots + \alpha_{j} = \alpha }{ \beta_1 + \cdots + \beta_{j} = \beta } } \frac{\alpha!\beta!}{\alpha_1! \cdots \alpha_j! \beta_1! \cdots \beta_j!} \prod_{\ell = 1}^{j} \partial_{x,\xi}^{\alpha_\ell,\beta_\ell} g(x,\xi). 
\end{aligned}
\end{equation}

Thus under the assumptions \eqref{eq:factorbound1} and \eqref{eq:factorbound2} we infer from \eqref{eq:faadibruno1}, \eqref{eq:derivativeestimate1}, and 
$ \psi^{(j)} \in L^\infty(\ro)$ for any $j \in \no$,
for $(\alpha,\beta) \in (\nn d \oplus \nn d) \setminus \{ 0 \}$
\begin{equation}\label{eq:estimatefactor1}
\begin{aligned}
& \left| \partial_{x,\xi}^{\alpha,\beta} \left( \psi \left( \left| \frac{x-y}{\mu w^\frac{\rho}{k} (x,\xi)} \right|^2 \right) \right) \right| \\
& \lesssim \sum_{j = 1}^{|\alpha+\beta|} \left| \psi^{(j)} \left( \left| \frac{x-y}{\mu w^\frac{\rho}{k} (x,\xi)} \right|^2 \right) \right| \sum_{ \stackrel[ | \alpha_\ell + \beta_\ell | \geqs 1, \ 1 \leqs \ell \leqs j ]{ \alpha_1 + \cdots + \alpha_{j} = \alpha }{ \beta_1 + \cdots + \beta_{j} = \beta } } 
\prod_{\ell = 1}^{j} \left| \partial_{x,\xi}^{\alpha_\ell,\beta_\ell} \left(w^{- \frac{2 \rho}{k} }(x,\xi) | x - y |^2 \right) \right| \\
& \lesssim \theta_\sigma(x,\xi)^{- \rho \left( |\alpha| + \sigma |\beta| \right) }
\end{aligned}
\end{equation}
and likewise using \eqref{eq:derivativeestimate2}
\begin{equation}\label{eq:estimatefactor2}
\begin{aligned}
& \left| \partial_{x,\xi}^{\alpha,\beta} \left( \psi \left( \left| \frac{\xi-\eta}{\mu w^\frac{\rho}{m} (x,\xi)} \right|^2 \right) \right) \right| \\
& \lesssim \sum_{j = 1}^{|\alpha+\beta|} \left| \psi^{(j)} \left( \left| \frac{\xi-\eta}{\mu w^\frac{\rho}{m} (x,\xi)} \right|^2 \right) \right| \sum_{ \stackrel[ | \alpha_\ell + \beta_\ell | \geqs 1, \ 1 \leqs \ell \leqs j ]{ \alpha_1 + \cdots + \alpha_{j} = \alpha }{ \beta_1 + \cdots + \beta_{j} = \beta } } 
\prod_{\ell = 1}^{j} \left| \partial_{x,\xi}^{\alpha_\ell,\beta_\ell} \left(w^{-\frac{2 \rho}{m}}(x,\xi) | \xi - \eta |^2 \right) \right| \\
& \lesssim \theta_\sigma(x,\xi)^{- \rho \left( |\alpha| + \sigma |\beta| \right)}. 
\end{aligned}
\end{equation}
The estimates \eqref{eq:estimatefactor1} and \eqref{eq:estimatefactor2} clearly hold also when $(\alpha,\beta) = (0,0)$.

Finally, using \eqref{eq:estimatefactor1} and \eqref{eq:estimatefactor2} in \eqref{eq:adef}, 
and the support properties of $\psi$ which guarantee \eqref{eq:factorbound1} and \eqref{eq:factorbound2},
it follows from Leibniz' rule and the fact that the support of the function \eqref{eq:integrand1}
has Lebesgue measure upper bounded by $C w^{d \rho \left( \frac1k + \frac1m \right)}(x,\xi)$ for some $C > 0$
that we get for $(\alpha,\beta) \in \nn d \oplus \nn d$
\begin{align*}
\left| \partial_{x,\xi}^{\alpha,\beta} a(x,\xi) \right| 
& \leqs 
\iint_{\left( \Omega_{\rho,\ep} \right)_{\rho,\mu}} 
\left|
\partial_{x,\xi}^{\alpha,\beta} \left( \psi \left( \left| \frac{x-y}{\mu w^\frac{\rho}{k} (x,\xi)} \right|^2 \right) \psi \left( \left| \frac{\xi-\eta}{\mu w^\frac{\rho}{m} (x,\xi)} \right|^2 \right) \right)
\right| \dd y \, \dd \eta \\
& \lesssim \theta_\sigma(x,\xi)^{- \rho \left( |\alpha| + \sigma |\beta| \right)}
w^{d \rho \left( \frac1k + \frac1m \right)}(x,\xi) \\
& \lesssim \theta_\sigma(x,\xi)^{d \rho (1+\sigma) - \rho \left( |\alpha| + \sigma |\beta| \right) }. 
\end{align*}
We have shown that $a \in G_\rho^{d \rho (1+\sigma), \sigma}$ which as observed implies the final claim $q \in G_\rho^{0,\sigma}$. 
\end{proof}

\begin{rem}\label{rem:supportintersection}
Since $\Omega_{\rho,\ep} \subseteq \rr {2d}$ is open, it follows from the 
the proof that 
$\supp (1-q) \subseteq \rr {2d} \setminus \Omega_{\rho,\ep}$. 
Suppose that $\emptyset \neq \Omega \subseteq \rr {2d}$
and define $q_j = q_{\ep_j,\delta_j,\rho,\Omega}$ for $j = 1,2$, where 
$0 < \ep_2 < \delta_2 < \ep_1 < \delta_1 < 1$ and $\delta_2 < b_\sigma$ where $b_\sigma$ is the constant \eqref{eq:bsigma}.
Then by Lemma \ref{lem:closureinclusion} and Proposition \ref{prop:cutoffsymbol}
\begin{equation*}
\supp q_2 \cap \supp (1 - q_1)
\subseteq \overline{\left( \Omega_{\rho,\delta_2} \right)} \cap \left( \rr {2d} \setminus \Omega_{\rho,\ep_1} \right)
\subseteq \Omega_{\rho,\ep_1} \cap \left( \rr {2d} \setminus \Omega_{\rho,\ep_1} \right)
= \emptyset
\end{equation*}
which will be useful in Sections \ref{sec:microlocmicroell} and  \ref{sec:schrodinger}.
\end{rem}

\section{The filter of global anisotropic singularities}\label{sec:filtersing}

\subsection{Elliptic symbols and tempered distributions that are smooth in a subset of phase space}\label{subsec:ellipticsmooth}

Let $\sigma \in \ro_+$ and $0 < \rho \leqs 1$. 
Given $\emptyset \neq \Omega \subseteq \rr {2d}$ and $r \in \ro$ we say that $a \in G_\rho^{r,\sigma}$ is elliptic in $\Omega$ if 
\begin{equation}\label{eq:ellipticLambda}
|a(x,\xi)| \geqs C \, \theta_\sigma (x,\xi)^r, 
\qquad (x,\xi) \in \Omega \setminus \rB_R,
\end{equation}
for some $C, R >0$ that are independent of $(x,\xi) \in \Omega$. 
As a convention we say that any symbol in $\cS'(\rr {2d})$ is elliptic in $\emptyset$. 
A symbol which is elliptic in $\rr {2d}$ is called elliptic. 

We need the following lemma.

\begin{lem}\label{lem:distanceaniso}
If $\sigma \in \ro_+$ then there exist $\mu > 0$ such that 
\begin{equation}\label{eq:lowerbounddiff}
\inf_{ \stackrel{|y| + |\eta|^{\frac1\sigma} = 1, \, |x| \leqs 2 \mu, \, |\xi| \leqs 2^\sigma \mu}{y, \eta , x, \xi \in \rr d}} \left| y - x \right| + \left| \eta - \xi \right|^{\frac1\sigma} \geqs \frac12. 
\end{equation}
\end{lem}

\begin{proof}
Suppose that there exist sequences $\left( (y_n, \eta_n) \right)_{n \in \no \setminus 0} \subseteq \rr {2d}$ and $\left( (x_n, \xi_n) \right)_{n \in \no \setminus 0} \subseteq \rr {2d}$ such that $|y_n| + |\eta_n|^{\frac1\sigma} = 1$, 
$|x_n| \leqs \frac2{n}$, $|\xi_n| \leqs \frac{2^\sigma}{n}$, and 
\begin{equation*}
\left| y_n - x_n \right| + \left| \eta_n - \xi_n \right|^{\frac1\sigma} < \frac12
\end{equation*}
for all $n \geqs 1$.
Then a subsequence of $\left( (y_n, \eta_n) \right)_{n \in \no \setminus 0}$ converges, that is 
$\lim_{n \to \infty} (y_n, \eta_n) = (y,\eta)$ for some $(y, \eta) \in \rr {2d}$ that satisfies $|y| + |\eta|^{\frac1\sigma} = 1$. 
(The notation for the subsequence is the same as the original sequence.)
We get the contradiction 
\begin{equation*}
1 = \limsup_{n \to \infty} \left| y_n - x_n \right| + \left| \eta_n - \xi_n \right|^{\frac1\sigma} \leqs \frac12.
\end{equation*}
\end{proof}

Next result says that ellipticity in a subset implies ellipticity in a neighborhood of the subset. 

\begin{prop}\label{prop:ellipticlarger}
Let $\sigma \in \ro_+$ and $0 < \rho \leqs 1$. 
If $r \in \ro$, $a \in G_\rho^{r,\sigma}$ and $a$ is elliptic in $\emptyset \neq \Omega \subseteq \rr {2d}$,
then there exists $\mu > 0$ such that $a$ is elliptic in $\Omega_{\rho,\mu} \subseteq \rr {2d}$ defined by \eqref{eq:anisoneighb}. 
\end{prop}

\begin{proof}
Suppose $(x,\xi) \in \Omega_{\rho,\mu}$ for $\mu > 0$.  
Let $(y,\eta) \in \Omega$ satisfy 
\begin{equation}\label{eq:anisoneighb2}
|x-y| < \mu \theta_\sigma^\rho (y,\eta) \quad \mbox{and} \quad  |\xi-\eta| < \mu \theta_\sigma^{\rho \sigma}(y,\eta)
\end{equation}
as in \eqref{eq:anisoneighb}.
Taking $\mu > 0$ sufficiently small we may by Lemma \ref{lem:boundaniso}
assume that 
\begin{equation}\label{eq:boundyetaupper}
\theta_\sigma (y, \eta) \leqs C_1 \theta_\sigma(x,\xi)
\end{equation}
for some constant $C_1 \geqs 1$ that does not depend on $(x,\xi) \in \Omega_{\rho,\mu}$. 

From 
\eqref{eq:anisoneighb2},
\eqref{eq:quasitriangleineq} and $\rho \leqs 1$ we get
\begin{align*}
|x| & \leqs |y| + \mu \theta_\sigma^\rho (y,\eta) \leqs (1 + \mu) \theta_\sigma(y,\eta), \\
|\xi|^\frac1\sigma & \leqs C_\sigma \left( |\eta|^\frac1\sigma + \mu^\frac1\sigma \theta_\sigma^{\rho} (y,\eta) \right)
\leqs C_\sigma \left( 1 + \mu^\frac1\sigma \right) \theta_\sigma(y,\eta). 
\end{align*}
Thus by possibly increasing $C_1 = C_1(\sigma,\mu)$, still independent of $(x,\xi)$, 
we may strengthen \eqref{eq:boundyetaupper} into
\begin{equation}\label{eq:propyetaxxi}
C_1^{-1} \theta_\sigma(x,\xi) \leqs \theta_\sigma (y, \eta) \leqs C_1 \theta_\sigma(x,\xi).
\end{equation}
By the assumed ellipticity we obtain
\begin{equation}\label{eq:elliptic1}
| a(y,\eta) | \geqs C_2 \, \theta_\sigma (y,\eta)^r \geqs C_2 C_1^{- |r| } \, \theta_\sigma (x,\xi)^r
\end{equation}
with $C_2 > 0$ if $|(x,\xi)| \geqs R$ for some $R > 0$. 

If $0 \leqs t \leqs 1$ then by Lemma \ref{lem:triangleineq} and \eqref{eq:propyetaxxi} 
\begin{equation}\label{eq:weightanisoupper}
\theta_\sigma \Big( t (x,\xi) + (1-t) (y,\eta) \Big)
\lesssim \theta_\sigma (x,\xi). 
\end{equation}
If $|(x,\xi)| \geqs R$ we may due to  \eqref{eq:propyetaxxi} 
assume that $|y| + |\eta|^{\frac1\sigma} \geqs 1$, after possibly increasing $R > 0$.
Now \eqref{eq:anisoneighb2}, $\rho \leqs 1$
and $|y| + |\eta|^{\frac1\sigma} \geqs 1$ 
yield
\begin{align*}
| x - y | & < \mu \theta_\sigma(y,\eta) \leqs 2 \mu \left( |y| + |\eta|^{\frac1\sigma} \right), \\
| \xi - \eta | & < \mu \theta_\sigma^\sigma (y,\eta) \leqs 2^\sigma \mu \left( |y| + |\eta|^{\frac1\sigma} \right)^{\sigma}
\end{align*}
from which we may infer using
Lemma \ref{lem:distanceaniso} the lower bound
\begin{equation}\label{eq:weightanisolower}
\begin{aligned}
& \theta_\sigma \Big( t (x,\xi) + (1-t) (y,\eta) \Big) 
= 1 + | y - t(y-x) | + | \eta - t(\eta-\xi) |^{\frac1\sigma} \\
& = 1 + \left( |y| + |\eta|^{\frac1\sigma} \right) \left( \left| \frac{y}{|y| + |\eta|^{\frac1\sigma}} - \frac{t(y-x)}{|y| + |\eta|^{\frac1\sigma}} \right| + \left| \frac{\eta}{ \left( |y| + |\eta|^{\frac1\sigma} \right)^\sigma} - \frac{t(\eta-\xi)}{ \left( |y| + |\eta|^{\frac1\sigma} \right)^\sigma} \right|^{\frac1\sigma} \right) \\
& \geqs \frac12 \theta_\sigma(y,\eta) 
\end{aligned}
\end{equation}
after possibly decreasing $\mu > 0$.

The mean value theorem combined with \eqref{eq:anisoneighb2} 
yields for some $0 \leqs t \leqs 1$
\begin{align*}
& \left| a(x,\xi) - a(y,\eta) \right| \\
& \leqs \left| \la \nabla_x a \big( t (x,\xi) + (1-t) (y,\eta) \big) , x - y \ra + \la \nabla_\xi a \big( t (x,\xi) + (1-t) (y,\eta) \big) , \xi - \eta \ra  \right| \\
& \leqs C_3 \theta_\sigma (x,\xi)^{r-\rho} \mu \theta_\sigma^\rho(y,\eta) + C_3 \theta_\sigma (x,\xi)^{r - \rho \sigma} \mu \theta_\sigma^{\rho \sigma}(y,\eta) \\
& \leqs C \mu \, \theta_\sigma (x,\xi)^r
\end{align*}
where $C_3,C > 0$ depends on $a, \sigma, \rho, \mu$ only, 
in the second inequality using 
\eqref{eq:symbolderivative1}, \eqref{eq:propyetaxxi}, \eqref{eq:weightanisoupper} and \eqref{eq:weightanisolower}.

Taking into account \eqref{eq:elliptic1} we obtain finally 
\begin{align*}
\left| a(x,\xi) \right| 
& \geqs 
\left| a(y,\eta) \right| -  \left| a(x,\xi) - a(y,\eta) \right| \\
& \geqs \left( C_2 C_1^{-|r|} - C \mu \right) \theta_\sigma (x,\xi)^r
\geqs \frac12 C_2 C_1^{-|r|} \theta_\sigma (x,\xi)^r
\end{align*}
provided $\mu > 0$ is sufficiently small, and provided $|(x,\xi)| \geqs R$. 
We have shown that $a$ is elliptic in $\Omega_{\rho,\mu}$. 
\end{proof}

The following definition concerns the behavior of a tempered distribution in a subset of phase space. 

\begin{defn}\label{def:smooth}
Let $\sigma \in \ro_+$, $0 < \rho \leqs 1$ and let $\Omega \subseteq \rr {2d}$. 
The distribution $u \in \cS'(\rr d)$ is $(\sigma,\rho)$-smooth in $\Omega$ if there exists $r \in \ro$ and $a \in G_\rho^{r,\sigma}$, elliptic in $\Omega$, such that $a^w(x,D) u \in \cS(\rr d)$.
\end{defn}

To simplify we often write smooth rather than $(\sigma,\rho)$-smooth when the parameters are understood from the context. 

\begin{rem}\label{rem:smooth}
We make a few observations.

\begin{enumerate}

\item By the invariance of the symbols classes $G_\rho^{r,\sigma}$ with respect to change of quantization parameter $\tau \in \ro$ in \eqref{eq:tquantization}, and \eqref{eq:calculusquantization}, Definition \ref{def:smooth} is invariant with respect to replacement by $a^w(x,D)$ by $a_\tau(x,D)$ for any $\tau \in \ro$. 

\item Any $a \in \cS(\rr {2d})$ is elliptic on $\emptyset$, and $a^w(x,D) u \in \cS(\rr d)$ for all $a \in \cS(\rr {2d})$ and all $u \in \cS'(\rr d)$. 
Thus any $u \in \cS'(\rr d)$ is smooth in $\emptyset$.

\item
If $\emptyset \neq \Omega \subseteq \rr {2d}$ is bounded
then there exists $r > 0$ such that 
$\Omega \subseteq \rB_r$. 
Let $a \in C_c^\infty(\rr {2d})$ be a cutoff function such that $\supp a \subseteq \rB_{2 r}$
and $a |_{\rB_r} \equiv 1$. 
Then $a \in G_\rho^{0,\sigma}$, $a^w(x,D) u \in \cS(\rr d)$ for any $u \in \cS'(\rr d)$ and $a$ is elliptic in $\Omega$. 
It follows that any $u \in \cS'(\rr d)$ is smooth in a bounded set $\Omega$. 	
	
\item
If $u \in \cS(\rr d)$ and $a = 1 \in G_\rho^{0,\sigma}$ then $a^w(x,D) u = u \in \cS(\rr d)$ and $a$ is elliptic in any subset $\Omega \subseteq \rr {2d}$. 
Thus $u \in \cS(\rr d)$ is smooth in any set in $\rP(\rr {2d})$. 
\item If $u \in \cS'(\rr d)$, $\emptyset \neq \Omega \subseteq T^* \rr d$, $0 < \rho_1 \leqs \rho_2 \leqs 1$ and 
$u$ is $(\sigma,\rho_2)$-smooth in $\Omega$ then it is also $(\sigma,\rho_1)$-smooth in $\Omega$, due to $G_{\rho_2}^{r,\sigma} \subseteq G_{\rho_1}^{r,\sigma}$. 

\item If $u \in \cS'(\rr d)$ is smooth in $\emptyset \neq \Omega \subseteq T^* \rr d$ then by Proposition \ref{prop:ellipticlarger}
it is automatically smooth in $\Omega_{\rho,\mu}$ defined by \eqref{eq:anisoneighb} for some $\mu > 0$. 

\end{enumerate}
\end{rem}

The following result will be essential. It says that smoothness, with rational anisotropy parameter $\sigma$, in a nonempty subset of $\rr {2d}$ 
may be defined by a cutoff symbol of order zero as in Proposition \ref{prop:cutoffsymbol}.

\begin{prop}\label{prop:smoothorderzero}
Let $\sigma \in \qo_+$, $0 < \rho \leqs 1$, $\lambda > 0$, 
$\emptyset \neq \Omega \subseteq \rr {2d}$, 
let $\Omega_{\rho,\lambda}$ be defined by \eqref{eq:anisoneighb}, 
and suppose that
$u \in \cS'(\rr d)$ is smooth in $\Omega_{\rho,\lambda}$. 
If $0 < \delta < b_\sigma$  where $b_\sigma$ is defined by \eqref{eq:bsigma} and $0 < \ep < \delta < \min (1,\lambda)$ then
$q_{\ep ,\delta,\rho,\Omega}^w(x,D)u \in \cS(\rr d)$
where $q_{\ep, \delta,\rho,\Omega} = q_{\ep, \delta,\rho,\sigma,\mu,\Omega} \in G^{0,\sigma}$ is defined by \eqref{eq:cutoffsymboldef},
$\sigma = \frac{k}{m}$ with $k, m \in \no \setminus 0$ coprime, 
$\mu$ satisfies \eqref{eq:emptyintersection1},
and $q_{\ep, \delta,\rho,\Omega}$ is
elliptic in $\Omega_{\rho,\ep}$. 
\end{prop}

\begin{proof}
By assumption there exist $r \in \ro$, $a \in G_\rho^{r, \sigma}$ and $C,R > 0$ such that $a^w u \in \cS$ and $|a (x,\xi)| \geqs C \theta_\sigma(x,\xi)^r$
when $(x,\xi) \in \Omega_{\rho,\lambda} \setminus \rB_R$. 
	
We use the pseudodifferential calculus for the anisotropic Shubin symbols $G_\rho^{r, \sigma}$, 
modified as indicated in Section \ref{subsec:psidiocalc} from $\rho = 1$ \cite{Rodino4} into $0 < \rho \leqs 1$. 
Defining $a_1 = \overline a \wpr a = |a|^2 + a_2 \in G_\rho^{2 r, \sigma}$ we thus have $a_2 \in G_\rho^{2 r- \rho(1+\sigma), \sigma}$. 
Let $0 < \ep < \delta < \min (1,\lambda)$ 
and let $\gamma,\nu \in \ro$ satisfy $\delta < \gamma < \nu < \min (1,\lambda)$. 
If $\chi = q_{\gamma, \nu, \rho,\sigma,\wt \mu,\Omega}$ 
according to \eqref{eq:cutoffsymboldef}, 
with $\sigma=\frac{k}{m}$ 
and $\wt \mu > 0$ sufficiently small,
then $\chi \in G_\rho^{0, \sigma}$ by Proposition \ref{prop:cutoffsymbol}. 
We set for $C_1 > 0$
\begin{equation*}
b = a_1 \chi + (1 - \chi) C_1 w_{k,m}^{\frac{2r}{k}}. 
\end{equation*}
By \cite[Lemma~3.6]{Cappiello6} we have $b \in G_\rho^{2 r, \sigma}$. 
If $z \notin \Omega_{\rho,\nu}$ then $\chi(z) = 0$ so $b (z) = C_1 w_{k,m}(z)^{\frac{2 r}{k}} \asymp \theta_\sigma(z)^{2 r}$
using \eqref{eq:weightequivalence}. 
If $z \in \Omega_{\rho,\nu} \setminus \rB_R \subseteq \Omega_{\rho,\lambda} \setminus \rB_R$ then for some and $C_3 > 0$ and $C_2 = C_2(C_1) > 0$, 
again using \eqref{eq:weightequivalence}, we have
\begin{align*}
|b (z)| & = \left| a_1(z) \chi (z) + (1 - \chi(z)) C_1 w_{k,m}^{\frac{2 r}{k}}(z)  \right| \\
& \geqs \chi(z) |a(z)|^2 + (1 - \chi(z)) C_1 w_{k,m}^{\frac{2 r}{k}}(z) - \chi(z) |a_2(z)| \\
& \geqs \chi(z) C^2 \theta_\sigma(z)^{2 r} + (1-\chi(z)) C_2 \theta_\sigma(z)^{2 r} - C_3 \theta_\sigma(z)^{2 r - \rho (1+\sigma)}. 
\end{align*}
We can pick $C_1 > 0$ such that $C_2 \leqs C^2$ which gives
\begin{align*}
|b (z)| & \geqs C_2 \theta_\sigma(z)^{2 r} \left( 1 - C_3 C_2^{-1} \theta_\sigma(z)^{- \rho (1+\sigma)} \right) 
\geqs \frac12 C_2 \theta_\sigma(z)^{2 r}
\end{align*}
for $z \in \Omega_{\rho,\nu} \setminus \rB_R$
after possibly increasing $R > 0$. 

This argument shows that $b  \in G_\rho^{2 r, \sigma}$ is an elliptic symbol, and then 
the proof of 
\cite[Lemma~6.3]{Rodino4},
with the proof modified straightforwardly from the assumption $\rho = 1$ into $0 < \rho \leqs 1$, 
implies that there exists a parametrix $c \in G_\rho^{-2 r, \sigma}$ that satisfies
$c \wpr b = 1 + r$ with $r \in \cS(\rr {2d})$. 
	
If $q = q_{\ep, \delta, \rho, \sigma, \mu, \Omega}$ is defined by \eqref{eq:cutoffsymboldef} 
then $q \in G_\rho^{0, \sigma}$ by Proposition \ref{prop:cutoffsymbol}. 
The latter proposition combined with Lemma \ref{lem:closureinclusion} 
and the assumption $\delta < b_\sigma$
implies that $\supp q \subseteq \overline{\left( \Omega_{\rho,\delta} \right)} \subseteq \Omega_{\rho,\gamma}$, cf. Remark \ref{rem:lowerboundsmall}.
From $1 = c \wpr b - r$ we obtain
\begin{equation*}
q = q \wpr c \wpr a_1 + q \wpr c \wpr (b - a_1) - q \wpr r
\end{equation*}
where $q \wpr r \in \cS(\rr {2d})$. 
Since $b(z) - a_1(z) = 0$ when $z \in \Omega_{\rho,\gamma}$ we have
$\supp (b - a_1) \subseteq \rr {2d} \setminus \Omega_{\rho,\gamma}$
which gives
$\supp q \cap \supp (b - a_1) = \emptyset$. 
The calculus then implies $q \wpr c \wpr (b - a_1) \in \cS(\rr {2d})$. 
If we set $r_1 = q \wpr c \wpr (b - a_1) - q \wpr r \in \cS(\rr {2d})$ then
\eqref{eq:SchwartzCont}
gives
\begin{align*}
q^w u 
& = q^w c^w  a_1^w u + r_1^w u \\
& = q^w c^w  {\overline a}^w a^w u + r_1^w u \in \cS
\end{align*}
due to $r_1 \in \cS(\rr {2d})$ and $a^w u \in \cS$. 
Finally we observe that $q$ is elliptic in $\Omega_{\rho,\ep}$ by construction. 
\end{proof}

The same proof adapted to the Kohn--Nirenberg calculus gives the following result.
The proof uses a version of \cite[Lemma~6.3]{Rodino4} adapted to the Kohn--Nirenberg calculus. 

\begin{cor}\label{cor:smoothorderzero}
Let $\sigma \in \qo_+$, $0 < \rho \leqs 1$, $\lambda > 0$, 
$\emptyset \neq \Omega \subseteq \rr {2d}$, let $\Omega_{\rho,\lambda}$
be defined by \eqref{eq:anisoneighb}, and suppose that
$u \in \cS'(\rr d)$ is smooth in $\Omega_{\rho,\lambda}$. 
If $0 < \delta < b_\sigma$  where $b_\sigma$ is defined by \eqref{eq:bsigma} and $0 < \ep < \delta < \min (1,\lambda)$ then
$q_{\ep ,\delta,\rho,\Omega}(x,D)u \in \cS(\rr d)$
where $q_{\ep, \delta,\rho,\Omega} = q_{\ep, \delta,\rho,\sigma, \mu,\Omega} \in G^{0,\sigma}$ is defined by \eqref{eq:cutoffsymboldef},
$\sigma = \frac{k}{m}$ with $k, m \in \no \setminus 0$ coprime, 
$\mu$ satisfies \eqref{eq:emptyintersection1},
and $q_{\ep, \delta,\rho,\Omega}$ is
elliptic in $\Omega_{\rho,\ep}$. 
\end{cor}

\subsection{The filter of anisotropic singularities}\label{subsec:filter}

\begin{defn}\label{def:generalfilter}
A filter of subsets of $\rr {2d}$ is denoted by $\fF \subseteq \rP(\rr {2d})$ and defined by the following properties
\cite{Schaefer1}. 

\begin{enumerate}
	
\item $\fF \neq \emptyset$;
		
\item If $F \in \fF$ and $F \subseteq G \subseteq \rr {2d}$ then $G \in \fF$; 
	
\item If $F, G \in \fF$ then $F \cap G \in \fF$.
	
\end{enumerate}
\end{defn}

The following concept generalizes \cite[Definition~6.9]{Cappiello7} where we use subsets $\Omega_\Lambda \subseteq T^* \rr d$ that are anisotropically annular in the sense of 
\begin{equation}\label{eq:AnisoAnnular}
\Omega_\Lambda = \{ (x,\xi) \in T^* \rr d: \ |x|^{2k} + |\xi|^{2m} \in \Lambda \}
\end{equation}
where $\Lambda \subseteq \ro_+$ is given and $\sigma = \frac{k}{m}$.
Here we relax this condition to arbitrary subsets of phase space. 

\begin{defn}\label{def:anisofilter}
Let 
$\sigma \in \qo_+$ and $0 < \rho \leqs 1$.
The \emph{filter of $(\sigma,\rho)$-anisotropic global singularities} of $u \in \cS'(\rr d)$ is defined as 
\begin{equation*}
\fF_{\sigma,\rho}(u) = \{ \Omega \subseteq \rr {2d}: \ u \mbox{ is $(\sigma,\rho)$-smooth in } \rr {2d}\setminus \Omega \} \subseteq \rP(\rr {2d}).
\end{equation*}
\end{defn}

In order to simplify notation we write $\fF_\rho(u) = \fF_{\sigma,\rho}(u)$ when the parameter $\sigma$ is given by the context.

First we show that the assumption of rational anisotropy parameter $\sigma$ implies that $\fF_{\sigma,\rho} (u) \subseteq \rP(\rr {2d})$ is indeed a filter. 

\begin{lem}\label{lem:filter}
If $\sigma \in \qo_+$, $0 < \rho \leqs 1$, $u \in \cS'(\rr d)$ and $\fF_{\sigma,\rho} (u) \subseteq \rP(\rr {2d})$ is defined as in Definition \ref{def:anisofilter}
then $\fF_{\sigma,\rho} (u)$ is a filter. 
\end{lem}

\begin{proof}
In Remark \ref{rem:smooth} (3) we have showed that $\rr {2d} \setminus \Omega \in \fF_\rho(u)$ for any bounded set $\Omega \subseteq \rr {2d}$ and any $u \in \cS'(\rr d)$. 
Thus $\fF_\rho(u) \neq \emptyset$
for any $u \in \cS'(\rr d)$ so property $(1)$ of a filter holds (cf. Definition \ref{def:generalfilter}). 
	
Suppose $\Omega_1 \in \fF_\rho (u)$ and $\Omega_1 \subseteq \Omega_2$. 
Then there exists $a \in G_\rho^{r,\sigma}$ such that $a^w u \in \cS(\rr d)$ and $a$ is elliptic in 
$\rr {2d} \setminus \Omega_1$.
Due to 
$\rr {2d} \setminus \Omega_2 \subseteq \rr {2d} \setminus \Omega_1$
the symbol $a$ is also elliptic in $\rr {2d} \setminus \Omega_2$. 
Thus $u$ is smooth in $\rr {2d} \setminus \Omega_2$ and $\Omega_2 \in \fF_\rho (u)$, 
which shows property $(2)$ of a filter. 
	
Finally we prove property $(3)$ of a filter, and then we need the assumption that $\sigma$ is rational. 
Assume $\Omega_1, \Omega_2 \in \fF_\rho (u)$. 
By Proposition \ref{prop:smoothorderzero} there exist 
$a,b \in G_\rho^{0,\sigma}$ such that $a^w u \in \cS(\rr d)$, $b^w u \in \cS(\rr d)$, $0 \leqs a, b \leqs 1$, $a$ is elliptic in 
$\rr {2d} \setminus \Omega_1$ and $b$ is elliptic in $\rr {2d} \setminus \Omega_2$.
Then $a+b \in G_\rho^{0,\sigma}$ is elliptic in $\rr {2d} \setminus \left( \Omega_1 \cap \Omega_2 \right)$
and $(a+b)^w(x,D) u \in \cS(\rr d)$. It follows that $\Omega_1 \cap \Omega_2 \in \mathcal  F_\rho (u)$
which proves property $(3)$ of a filter. 
\end{proof}

\begin{rem}\label{rem:filtersing1}
Definition \ref{def:generalfilter} (1) and (2) imply that $\rr {2d} \in \fF_\rho(u)$ for any $u \in \cS'(\rr d)$. 
\end{rem}

\begin{rem}\label{rem:filtersing2}
By Remark \ref{rem:smooth} (3) we have $\rr {2d} \setminus \rB_R \in \fF_\rho(u)$ for any $R > 0$ and any $u \in \cS'(\rr d)$.
\end{rem}

\begin{rem}\label{rem:filtersing3}
Suppose $u \in \cS'(\rr d)$ and $\emptyset \neq \Omega \in \fF_\rho (u)$. 
Then $u$ is smooth in $\rr {2d} \setminus \Omega$. 
By Proposition \ref{prop:ellipticlarger} $u$ is also smooth in the open set $(\rr {2d} \setminus \Omega)_{\rho,\mu} \supseteq \rr {2d} \setminus \Omega$
for some $\mu > 0$. It follows that $\rr {2d} \setminus \left(\rr {2d} \setminus \Omega \right)_{\rho,\mu} \in \fF_\rho (u)$. 
Hence for every nonempty $\Omega \in \fF_\rho (u)$ there exists a closed set $\Omega_1 \subseteq \Omega$ such that $\Omega_1 \in \fF_\rho(u)$. 
\end{rem}

\begin{rem}\label{rem:filtersing4}
If $0 < \rho_1 \leqs \rho_2 \leqs 1$ then by 
Remark \ref{rem:smooth} (5) we have $\fF_{\rho_2}(u) \subseteq \fF_{\rho_1}(u)$.
\end{rem}

\begin{lem}\label{lem:charschwartz}
Let $\sigma \in \qo_+$ and $0 < \rho \leqs 1$.
If $u \in \cS'(\rr d)$ then $\emptyset \in \fF_\rho(u)$ if and only if $u \in \cS(\rr d)$. 
\end{lem}

\begin{proof}
In Remark \ref{rem:smooth} (4) we have shown that $u \in \cS(\rr d)$ implies $\emptyset \in \fF_\rho (u)$. 
	
Suppose on the other hand $\emptyset \in \fF_\rho (u)$. 
Then there exists $a \in G_\rho^{r,\sigma}$ such that $a^w(x,D) u \in \cS(\rr d)$ and $a$ is elliptic in $\rr {2d}$. 
There exists a parametrix $b \in G_\rho^{-r,\sigma}$ which satisfies 
$b \wpr a = 1 + r$ where $r \in \cS(\rr {2d})$, cf. \cite[Lemma~6.3]{Rodino4}. 
We conclude
\begin{equation*}
u = b^w a^w u - r^w u \in \cS
\end{equation*}
due to \eqref{eq:SchwartzCont}, 
and $r^w: \cS' \to \cS$ being regularizing.
\end{proof}

\begin{rem}\label{rem:property4}
In \cite{Schaefer1} a filter $\fF$ is assumed to satisfy properties (1), (2), (3) in Definition \ref{def:generalfilter}, and furthermore 
$\emptyset \notin \fF$. 
We exclude the last property from our definition of filter
in order to allow $u \in \cS(\rr d)$, cf. 
Lemma \ref{lem:charschwartz}.
\end{rem}

By the proof of Lemma \ref{lem:charschwartz} we also have the following criterion of when a tempered distribution is a Schwartz function.

\begin{cor}\label{cor:charschwartz}
Let $\sigma \in \qo_+$ and $0 < \rho \leqs 1$.
If $u \in \cS'(\rr d)$ then $\rB_R \in \fF_\rho(u)$ for some $R > 0$ if and only if $u \in \cS(\rr d)$. 
\end{cor}

\begin{rem}\label{rem:Schwartzfilter}
Let $u \in \cS'(\rr d)$. Then $\fF_\rho (u) = \rP(\rr {2d})$ if and only if $u \in \cS(\rr d)$. 
This is a consequence of
Lemma \ref{lem:charschwartz} and property (2) in Definition \ref{def:generalfilter}.
\end{rem}

\subsection{Intersection of filter sets}\label{subsec:intersection}

Let $\sigma \in \qo_+$ and $0 < \rho \leqs 1$.
By Lemma \ref{lem:charschwartz} we have 
\begin{equation}\label{eq:emptyintersection2}
\bigcap_{\Omega \in \fF_\rho(u)} \Omega = \emptyset
\end{equation}
if $u \in \cS(\rr d)$. 
Remark \ref{rem:filtersing2} shows that \eqref{eq:emptyintersection2} in fact holds for any $u \in \cS'(\rr d)$.

We investigate strengthenings  of \eqref{eq:emptyintersection2} that implies that $u \in \cS(\rr d)$. 
Note first that Remark \ref{rem:filtersing3} and \eqref{eq:emptyintersection2} imply
that \eqref{eq:emptyintersection2} is equivalent to
\begin{equation}\label{eq:emptyintersection3}
\bigcap_{\stackbin{\Omega \in \fF_\rho(u)}{\Omega \ \rm{closed}}} \Omega = \emptyset. 
\end{equation}

Then we strengthen \eqref{eq:emptyintersection3} into 
\begin{equation}\label{eq:emptyintersection4}
\bigcap_{\stackbin{\Omega \in \gG}{\Omega \ \rm{closed}}} \Omega = \emptyset
\end{equation}
where $\gG \subseteq \fF_\rho(u)$ is a subset of the filter of singularities. 

For a collection of subsets $\{ A_j \}_{j \in J} \subseteq \rP(\rr {2d})$ the \emph{finite intersection property} \cite{Simmons1} 
means that
\begin{equation*}
\bigcap_{j \in J_F} A_j \neq \emptyset \quad \mbox{ for all finite } J_F \subseteq J.
\end{equation*}
We assume that the finite intersection property implies non-empty intersection: 
\begin{equation*}
\bigcap_{j \in J_F} A_j \neq \emptyset \quad \mbox{ for all finite } J_F \subseteq J \quad \Longrightarrow \quad \bigcap_{j \in J} A_j \neq \emptyset. 
\end{equation*}
This property holds for a collection of closed subsets of a compact set in a topological space \cite{Simmons1}. 

We will use the contrapositive implication for the subset $\gG \subseteq \fF_\rho(u)$ of 
the filter $\fF_\rho(u)$, restricted to closed subsets, of $u \in \cS'(\rr d)$.
Thus we assume that $\gG$ has the property  
\begin{equation}\label{eq:FIPclosedcompact}
\bigcap_{\stackbin{\Omega \in \gG}{\Omega \ \rm{closed}}} 
\Omega = \emptyset \quad \Longrightarrow \quad 
\bigcap_{\stackbin{\Omega \in \gG \cap \fF}{\Omega \ \rm{closed}}} 
\Omega = \emptyset \quad
\mbox{ for some finite } \fF \subseteq \gG.  
\end{equation}
This implication is true e.g. when the subsets in $\gG$ are $\sigma$-conic. 
In fact this is a consequence of the compactness of $\sr {2d-1} \subseteq \rr {2d}$.

\begin{prop}\label{prop:emptyintersection}
Suppose $\sigma \in \qo_+$, $0 < \rho \leqs 1$ and let $u \in \cS'(\rr d)$. 
If $\gG \subseteq \fF_\rho(u)$ satisfies \eqref{eq:FIPclosedcompact} 
then \eqref{eq:emptyintersection4} implies that $u \in \cS(\rr d)$. 
\end{prop}

\begin{proof}
From the assumption \eqref{eq:emptyintersection4} and \eqref{eq:FIPclosedcompact}  it follows that 
\begin{equation*}
\bigcap_{\stackbin{\Omega \in \gG \cap \fF}{\Omega \ \rm{closed}}} \Omega = \emptyset
\end{equation*}
for some finite subset $\fF \subseteq \gG$. By Property (3) of the filter $\fF_\rho(u)$, cf. Definition \ref{def:generalfilter},  we then have $\emptyset \in \fF_\rho(u)$
which by Lemma \ref{lem:charschwartz} implies that $u \in \cS(\rr d)$.
\end{proof}

The following result says that the anisotropic Gabor wave front set $\WFgs (u)$ can be recovered from the filter of anisotropic singularities $\fF_{\sigma,\rho}(u)$, 
by taking the intersection of the subsets in the filter that are $\sigma$-conic closed subsets in $\rr {2d} \setminus 0$, cf. \eqref{eq:sigmaconic}. 

\begin{prop}\label{prop:filterWF}
Let $\sigma \in \qo_+$ and $0 < \rho \leqs 1$.
If $u \in \cS'(\rr d)$ then
\begin{equation}\label{eq:filterWF}
\bigcap_{\stackbin[ \Omega \ \sigma- \rm conic ]{\Omega \in \fF_\rho(u) \cap \rP(\rr {2d} \setminus 0)}{ \Omega \ \rm{ closed \ in } \ \rr {2d} \setminus 0 }} \Omega \quad = \WFgs(u). 
\end{equation}
\end{prop}

\begin{proof}
We prove the equality of complements in $\rr {2d} \setminus 0$
\begin{equation}\label{eq:filterWFcompl}
\bigcup_{\stackbin[ \Omega \ \sigma- \rm conic ]{\Omega \in \fF_\rho(u) \cap \rP(\rr {2d} \setminus 0)}{ \Omega \ \rm{ closed \ in } \ \rr {2d} \setminus 0 }}\left( \rr {2d} \setminus \left( \Omega \cup \{ 0 \} \right) \right)
= \rr {2d} \setminus \left( \WFgs(u) \cup \{ 0 \} \right).  
\end{equation}

If $\WFgs(u) = \rr {2d} \setminus 0$ then 
$\rr {2d} \setminus \left( \WFgs(u) \cup \{ 0 \} \right) = \emptyset$
and trivially
\begin{equation}\label{eq:filterWFcompl1}
\rr {2d} \setminus \left( \WFgs(u) \cup \{ 0 \} \right)
\subseteq 
\bigcup_{\stackbin[ \Omega \ \sigma- \rm conic ]{\Omega \in \fF_\rho(u) \cap \rP(\rr {2d} \setminus 0)}{ \Omega \ \rm{ closed \ in } \ \rr {2d} \setminus 0 }}\left( \rr {2d} \setminus \left( \Omega \cup \{ 0 \} \right) \right). 
\end{equation}

Suppose there exists $z_0 \in \rr {2d} \setminus \left( \WFgs(u) \cup \{ 0 \} \right)$. By \cite[Proposition~3.5]{Wahlberg4} and its proof 
there exists $a \in G^{0,\sigma} \subseteq G_\rho^{0,\sigma}$ such that $0 \leqs a \leqs 1$, $a^w u \in \cS$
and 
\begin{equation*}
a(z) = 1, \quad z \in \Gamma, \quad |z| \geqs R, 
\end{equation*}
with $R > 0$, where $\Gamma \subseteq \rr {2d} \setminus 0$ is an open $\sigma$-conic subset such that $z_0 \in \Gamma$. 
We may conclude that 
$u \in \cS'(\rr d)$ is smooth in $\Gamma \cup \{ 0 \}$, that is $\Omega := \rr {2d} \setminus \left( \Gamma \cup \{ 0 \} \right)\in \fF_\rho(u)$.
Since $\Omega$ is $\sigma$-conic, closed in $\rr {2d} \setminus 0$, and $z_0 \notin \Omega$, we have shown 
\eqref{eq:filterWFcompl1}
also when the left hand side is non-empty. 

Next we take $z_0 \in \rr {2d} \setminus 0$ and assume that $z_0 \notin \Omega$ for some $\sigma$-conic subset $\Omega \subseteq \rr {2d} \setminus 0$, closed in $\rr {2d} \setminus 0$, 
such that $\Omega \in \fF_\rho(u)$. 
By Propositions \ref{prop:cutoffsymbol} and \ref{prop:smoothorderzero} there exists a symbol $a \in G_\rho^{0,\sigma}$ such that $0 \leqs a \leqs 1$, $a^w u \in \cS$ and $a(z) = 1$ when $z \in \rr {2d} \setminus \Omega$. 
By the proof of \cite[Proposition~3.5]{Wahlberg4}, generalized from $\rho = 1$ into $0 < \rho \leqs 1$, it follows that $z_0 \notin \WFgs(u)$, which proves 
\begin{equation}\label{eq:filterWFcompl2}
\rr {2d} \setminus \left( \WFgs(u) \cup \{ 0 \} \right)
\supseteq 
\bigcup_{\stackbin[ \Omega \ \sigma- \rm conic ]{\Omega \in \fF_\rho(u) \cap \rP(\rr {2d} \setminus 0)}{ \Omega \ \rm{ closed \ in } \ \rr {2d} \setminus 0 }}\left( \rr {2d} \setminus \left( \Omega \cup \{ 0 \} \right) \right). 
\end{equation}
Taken together \eqref{eq:filterWFcompl1} and \eqref{eq:filterWFcompl2} prove \eqref{eq:filterWFcompl} which is equivalent to \eqref{eq:filterWF}.
\end{proof}

\subsection{Characterization of filters of anisotropic singularities}\label{subsec:characterizationfilter}

In this subsection we characterize filters of anisotropic singularities by means of the STFT, 
in a spirit that resembles the definition of anisotropic Gabor wave front sets. 
The criterion we will use in Theorem \ref{thm:characsing} is independent of the window function as we prove in the following result. 

\begin{prop}\label{prop:windowinvariance}
Suppose $\sigma \in \qo_+$ and $0 < \rho \leqs 1$. 
Let $\emptyset \neq \Omega \subseteq \rr {2d}$, let $\ep > \delta > 0$, 
let $\Omega_{\rho,\ep}$ and $\Omega_{\rho,\delta}$ be defined by \eqref{eq:anisoneighb}, 
let $u \in \cS'(\rr d)$, and let $\fy, \psi \in \cS(\rr d) \setminus 0$. 
If 
\begin{equation}\label{eq:quickdecay1}
\sup_{z \in \Omega_{\rho,\ep}}
\eabs{z}^N |V_\fy u(z)| < \infty \quad \forall N \geqs 0
\end{equation}
then 
\begin{equation}\label{eq:quickdecay2}
\sup_{z \in \Omega_{\rho,\delta}}
\eabs{z}^N |V_\psi u(z)| < \infty \quad \forall N \geqs 0. 
\end{equation}
\end{prop}

\begin{proof}
First we show the following claim. 
There exists $\mu > 0$ such that 
if $(x,\xi) \in \Omega_{\rho,\delta}$ and 
\begin{equation}\label{eq:smallupperbound1}
|(y,\eta)|^{ \frac{1}{\rho} \max\left( \sigma, \frac{1}{\sigma} \right)} 
\leqs \mu |(x,\xi)|
\end{equation}
then $(x-y, \xi-\eta) \in \Omega_{\rho,\ep}$. 

In fact the assumption $(x,\xi) \in \Omega_{\rho,\delta}$ implies by \eqref{eq:anisoneighb} that 
there exists $(z,\zeta) \in \Omega$ such that 
\begin{equation}\label{eq:Omega0}
|x - z| < \delta \theta_\sigma^\rho(z,\zeta) \quad \mbox{ and } \quad
|\xi - \zeta| < \delta \theta_\sigma^{\rho \sigma}(z,\zeta)
\end{equation}
which yields 
\begin{equation*}
|x| < |z| + \delta \theta_\sigma^\rho(z,\zeta) 
\quad \mbox{ and } \quad
|\xi| < |\zeta| + \delta \theta_\sigma^{\rho \sigma}(z,\zeta).
\end{equation*}

By the assumption \eqref{eq:smallupperbound1}, \eqref{eq:quasitriangleineq}, $\rho \leqs 1$,
and \eqref{eq:sobolevweightestimate1}
we obtain for some $C_\sigma > 0$
\begin{align*}
|(y,\eta)|^{\max\left( \sigma, \frac{1}{\sigma} \right)} 
& < \mu^\rho \left( 2^{\frac{\rho}{2}} \eabs{(z,\zeta)}^\rho + \left( 2 \delta \right)^\rho \left( \theta_\sigma (z,\zeta) \right)^{ \rho \max \left( 1,\sigma \right) } \right) \\
& \leqs \mu^\rho \left( C_\sigma + \left( 2 \delta \right)^\rho \right) \left( \theta_\sigma (z,\zeta) \right)^{ \rho 
\max \left( 1,\sigma \right) }
\end{align*}
which in turn implies 
\begin{align*}
|(y,\eta)|
& \leqs \mu^{\rho \, {\min\left( \sigma, \frac{1}{\sigma} \right)} } C
\left( \theta_\sigma (z,\zeta) \right)^{ \rho \min\left( 1,\sigma \right)}
\end{align*}
where $C = C_{\sigma,\rho,\delta} > 0$. 

Combined with \eqref{eq:Omega0} this gives
\begin{align*}
|x - y - z|
& \leqs |y| + | x - z |
< \left( \mu^{ \rho \, {\min\left( \sigma, \frac{1}{\sigma} \right)} } C + \delta \right) \theta_\sigma^\rho(z,\zeta)
< \ep \theta_\sigma^\rho (z,\zeta), \\
|\xi - \eta - \zeta|
& \leqs |\eta| + | \xi - \zeta |
< \left( \mu^{ \rho \, {\min\left( \sigma, \frac{1}{\sigma} \right)} } C + \delta \right) \theta_\sigma^{\rho \sigma} (z,\zeta)
< \ep \theta_\sigma^{\rho \sigma} (z,\zeta) 
\end{align*}
provided $\mu > 0$ is chosen sufficiently small. 
Hence $(x-y, \xi-\eta) \in \Omega_{\rho,\ep}$, so we have proved the claim.

We use \cite[Lemma~11.3.3]{Grochenig1} which implies the inequality
\begin{equation*}
|V_\psi u(z)| 
\leqs (2 \pi)^{-\frac{d}{2}} \| \fy \|_{L^2}^{-2}
\int_{\rr {2d}} |V_\fy u (z-w) | \, |V_\psi \fy (w) | \, \dd w, \quad z \in \rr {2d}. 
\end{equation*}
Let $N \geqs 0$. 
We estimate using \eqref{eq:Peetre}
\begin{equation*}
\eabs{z}^N |V_\psi u(z)| 
\lesssim 
\int_{\rr {2d}} \eabs{z-w}^N |V_\fy u (z-w) | \, \eabs{w}^N |V_\psi \fy (w) | \, \dd w 
= I_1 + I_2
\end{equation*}
where we split the integral as 
\begin{align*}
I_1 & = \int_{|w|^{ \frac{1}{\rho} \max\left( \sigma, \frac{1}{\sigma} \right)} \leqs \mu |z|} \eabs{z-w}^N |V_\fy u (z-w) | \, \eabs{w}^N |V_\psi \fy (w) | \, \dd w, \\
I_2 & = \int_{|w|^{ \frac{1}{\rho} \max\left( \sigma, \frac{1}{\sigma} \right)} > \mu |z|} \eabs{z-w}^N |V_\fy u (z-w) | \, \eabs{w}^N |V_\psi \fy (w) | \, \dd w.
\end{align*}
Let $z \in \Omega_{\rho,\delta}$. By the claim we have $z - w \in \Omega_{\rho,\ep}$ in $I_1$. 
The assumption \eqref{eq:quickdecay1} and $V_\psi \fy \in \cS(\rr {2d})$ imply that $I_1$ is upper bounded
uniformly over $z \in \Omega_{\rho,\delta}$.

For $I_2$ we use instead the bound \eqref{eq:STFTtempered} for some $r \geqs 0$ which 
together with \eqref{eq:Peetre}
give the estimate
\begin{align*}
I_2 
& \leqs \int_{|w|^{ \frac{1}{\rho} \max\left( \sigma, \frac{1}{\sigma} \right)} > \mu |z|} 
\eabs{z-w}^{N+ r} \, \eabs{w}^{ N} |V_\psi \fy (w) | \, \dd w \\
& \lesssim \int_{|w|^{ \frac{1}{\rho} \max\left( \sigma, \frac{1}{\sigma} \right)} > \mu |z|} 
\eabs{z}^{N+ r} \, \eabs{w}^{ 2 N + r} |V_\psi \fy (w) | \, \dd w \\
& \lesssim \int_{|w|^{ \frac{1}{\rho} \max\left( \sigma, \frac{1}{\sigma} \right)} > \mu |z|} 
\eabs{w}^{ 2 N + r + \frac{1}{\rho} (N + r) \max\left( \sigma, \frac{1}{\sigma} \right)} |V_\psi \fy (w) | \, \dd w \\
& \leqs \int_{\rr {2d}} 
\eabs{w}^{ 2 N + r + \frac{1}{\rho} (N + r) \max\left( \sigma, \frac{1}{\sigma} \right)} |V_\psi \fy (w) | \, \dd w 
\end{align*}
which is upper bounded uniformly over $z \in \Omega_{\rho,\delta}$, again due to $V_\psi \fy \in \cS(\rr {2d})$.
We have shown \eqref{eq:quickdecay2}. 
\end{proof}

\begin{lem}\label{lem:distance}
Suppose $\sigma \in \ro_+$ and $0 < \rho \leqs 1$. 
Let $\emptyset \neq \Omega \subseteq \rr {2d}$, 
let $0 < \ep < \delta$, and let $\Omega_{\rho,\ep}$ and $\Omega_{\rho,\delta}$ be defined by \eqref{eq:anisoneighb}. 
There exists $C = C_{\sigma,\rho,\ep,\delta} > 0$ such that 
\begin{equation}\label{eq:separationineq}
| (x,\xi) |^{ \rho \min\left( \sigma, \frac{1}{\sigma} \right)}
\leqs C \left| (y,\eta) - (x,\xi) \right|
\end{equation}
for all $(x,\xi) \in \Omega_{\rho,\ep}$ and all $(y,\eta) \in \rr {2d} \setminus \Omega_{\rho,\delta}$. 
\end{lem}

\begin{proof}
The assumption $(x,\xi) \in \Omega_{\rho,\ep}$ implies by \eqref{eq:anisoneighb} that 
there exists $(z,\zeta) \in \Omega$ such that 
\begin{equation}\label{eq:Omega1}
|x - z| < \ep \theta_\sigma^\rho(z,\zeta) 
\quad \mbox{ and } \quad
|\xi - \zeta| < \ep \theta_\sigma^{\rho \sigma} (z,\zeta), 
\end{equation}
and the assumption $(y,\eta) \in \rr {2d} \setminus \Omega_{\rho,\delta}$ implies that 
\begin{equation}\label{eq:Omega2}
|y - z| \geqs \delta \theta_\sigma^\rho (z,\zeta) 
\quad \mbox{or} \quad 
|\eta - \zeta| \geqs \delta \theta_\sigma^{\rho \sigma}(z,\zeta)
\end{equation}
must hold. 
Combining \eqref{eq:Omega1} and  \eqref{eq:Omega2} yield the consequence 
\begin{equation*}
|y - x| > (\delta - \ep) \theta_\sigma^\rho (z,\zeta) 
\quad \mbox{or} \quad 
|\eta - \xi| > (\delta - \ep) \theta_\sigma^{\rho \sigma}(z,\zeta)
\end{equation*}
from which we may conclude 
\begin{equation}\label{eq:separation1}
|(y,\eta) - (x,\xi) | > (\delta - \ep) \left( \theta_\sigma (z,\zeta) \right)^{ \rho \min \left( 1, \sigma \right)}.
\end{equation}

On the other hand we obtain from \eqref{eq:Omega1}, $0 < \rho \leqs 1$
and \eqref{eq:sobolevweightestimate1} 
for some $C_\sigma > 0$
\begin{align*}
|(x,\xi) | 
& < | (z,\zeta) | + \ep \theta_\sigma^{\rho} (z,\zeta) + \ep \theta_\sigma^{\rho \sigma} (z,\zeta) \\ 
& \leqs C_\sigma \left( \theta_\sigma (z,\zeta) \right)^{\max(1,\sigma)} + \ep \theta_\sigma^\rho (z,\zeta) + \ep \theta_\sigma^{\rho\sigma} (z,\zeta) \\
& \leqs \left( C_\sigma + 2 \ep \right) \left( \theta_\sigma (z,\zeta) \right)^{\max(1,\sigma)}. 
\end{align*}
Combining this with $\max(1,\sigma) \min \left( \sigma, \frac1\sigma \right) = \min(1,\sigma)$ and \eqref{eq:separation1} we reach the claim \eqref{eq:separationineq}. 
\end{proof}

We are now in a position where we can characterize $(\sigma,\rho)$-smoothness of a tempered distribution in terms of super-polynomial decay 
behavior of the STFT in the case of rational anisotropy parameter $\sigma > 0$. 
This result corresponds to 
\cite[Definition~4.1]{Rodino4} (cf. Section \ref{subsec:WFgaboraniso}) 
and \cite[Proposition~3.5]{Wahlberg4}
for anisotropic Gabor wave front sets. 
Indeed the latter result shows that the anisotropic Gabor wave front set, defined in \cite[Definition~4.1]{Rodino4} 
using super-polynomial decay behavior of the STFT, may equivalently be defined using the concept of smoothness.  

In this paper we have taken an opposite track: The definition of $(\sigma,\rho)$-smoothness of a tempered distribution
is conceptually similar to \cite[Proposition~3.5]{Wahlberg4}, and the following result shows that the concept may be characterized by means of 
the STFT. 

\begin{thm}\label{thm:characsing}
Suppose $\sigma \in \qo_+$ and $0 < \rho \leqs 1$. 
Let $\emptyset \neq \Omega \subseteq \rr {2d}$, let $u \in \cS'(\rr d)$, and let $\fy \in \cS(\rr d) \setminus 0$. 
If $0 < \delta < 1$, 
$\Omega_{\rho,\delta}$ is defined by \eqref{eq:anisoneighb},
and $u$ is smooth in $\Omega_{\rho,\delta}$ then for any $0 < \ep < \delta$ we have
\begin{equation}\label{eq:decayestimate1}
\sup_{z \in \Omega_{\rho,\ep}}
\eabs{z}^N |V_\fy u (z)| < \infty \quad \forall N \geqs 0.
\end{equation}
Conversely if \eqref{eq:decayestimate1} holds true with $0 < \ep < 1$
then $u$ is smooth in $\Omega_{\rho,\delta}$ for any $\delta$ such that $0 < \delta <  \min(b_\sigma,\ep)$
where $b_\sigma$ is defined in \eqref{eq:bsigma}. 
\end{thm}

\begin{proof}
Suppose $0 < \delta < 1$ and $u$ is smooth in $\Omega_{\rho,\delta}$ defined by \eqref{eq:anisoneighb}. 
By Proposition \ref{prop:ellipticlarger} $u$ is automatically smooth in $\left( \Omega_{\rho,\delta} \right)_{\rho,\lambda}$
for some $0 < \lambda < 1$. Let $0 < \gamma < \nu < \lambda$ and let 
$a = q_{\gamma,\nu,\rho,\sigma,\mu,\Omega_{\rho,\delta} } \in G_\rho^{0,\sigma}$ 
be defined as in Proposition \ref{prop:cutoffsymbol}
with $\sigma = \frac{k}{m}$, $k,m \in \no \setminus 0$ coprime, 
and $\mu > 0$ satisfying 
\begin{equation*}
\left( \rr {2d} \setminus \left( \Omega_{\rho,\delta} \right)_{\rho,\nu} \right)_{\rho,\mu} \bigcap \left( \left( \Omega_{\rho,\delta} \right) _{\rho,\gamma} \right)_{\rho,\mu} = \emptyset
\end{equation*}
which is possible by Proposition \ref{prop:disjoint}.
By Corollary \ref{cor:smoothorderzero} we have $a(x,D) u \in \cS$ if $\nu$ is sufficiently small, and moreover
\begin{equation}\label{eq:suppsymbol1}
\supp (1- a) \subseteq 
\rr {2d} \setminus \left( \Omega_{\rho,\delta} \right)_{\rho,\gamma}
\subseteq \rr {2d} \setminus \Omega_{\rho,\delta}.
\end{equation}

The fact that $a(x,D) u \in \cS$ and \eqref{eq:STFTschwartz} imply
\begin{equation*}
\sup_{z \in \rr {2d}}
\eabs{z}^N |V_\fy ( a(x,D) u) (z)| < \infty \quad \forall N \geqs 0.  
\end{equation*}
To show \eqref{eq:decayestimate1} it hence suffices to show 
\begin{equation}\label{eq:reduceddecay1}
\sup_{z \in \Omega_{\rho,\ep}}
\eabs{z}^N |V_\fy \left( ( 1 - a ) (x,D) u \right) (z)| < \infty
\end{equation}
for any $N \geqs 0$. 

Let $c \in G_\rho^{r,\sigma}$ for $r \in \ro$. 
We obtain from \eqref{eq:STFTinverse} for $z \in \rr {2d}$
\begin{align*}
V_\fy \left( c(x,D) u \right) (z)
& = ( 2 \pi )^{- \frac{d}{2}} \left( c(x,D) u, \Pi(z ) \fy \right)
= ( 2 \pi )^{- \frac{d}{2}} \left( u, c(x,D)^* \Pi(z ) \fy \right) \\
& = ( 2 \pi )^{- d} \| \fy \|_{L^2}^{-2} 
\int_{\rr {2d}} V_\fy u (w)  \overline{ \left( c(x,D)^* \Pi(z ) \fy, \Pi(w) \fy \right)}  \, \dd w \\
& = ( 2 \pi )^{- d} \| \fy \|_{L^2}^{-2} 
\int_{\rr {2d}} V_\fy u (w)  \left( c(x,D) \Pi(w ) \fy , \Pi(z) \fy \right)  \, \dd w. 
\end{align*}
We may thus write 
\begin{equation}\label{eq:STFTkernel}
V_\fy \left( c (x,D) u \right) (z)
= \int_{\rr {2d}} K(z,w) V_\fy u(w) \, \dd w, \quad z \in \rr {2d}, 
\end{equation}
with integral kernel
\begin{equation}\label{eq:integralkernel1}
K(z,w) = ( 2 \pi )^{- d} \| \fy \|_{L^2}^{-2} \left( c(x,D) \Pi(w ) \fy , \Pi(z) \fy \right). 
\end{equation}

Using $\wh{M_\eta T_y \fy} = T_\eta M_{-y} \wh \fy$ we get from \eqref{eq:tquantization} with $\tau = 0$
\begin{equation*}
c(x,D) (M_\eta T_y ) \fy (u)
= ( 2 \pi )^{- \frac{d}{2} }
\int_{\rr d} e^{i \left( \la u, \zeta \ra - \la y, \zeta - \eta \ra \right)} c(u,\zeta) \wh \fy(\zeta - \eta) \, \dd \zeta
\end{equation*}
which inserted into \eqref{eq:integralkernel1} yields for $c = 1 - a$ and $x, \xi, y, \eta \in \rr d$
\begin{equation}\label{eq:integralkernel2}
\begin{aligned}
& K(x, \xi ; y, \eta) \\
& = (2 \pi)^{-\frac{3d}{2}} \| \fy \|_{L^2}^{-2} e^{i \la y, \eta \ra}
\int_{\rr {2d}} e^{i \left( \la u - y, \zeta \ra - \la u, \xi \ra \right)} \left( 1- a (u,\zeta) \right) 
\wh \fy(\zeta - \eta) \overline{\fy(u - x)} 
\, \dd u \, \dd \zeta \\
& = (2 \pi)^{-\frac{3d}{2}} \| \fy \|_{L^2}^{-2} e^{i \la x, \eta - \xi \ra}
\int_{\rr {2d}} e^{i \left( \la u, \eta - \xi + \zeta \ra + \la \zeta, x - y \ra \right)} \left( 1- a (u + x,\zeta + \eta) \right) 
\wh \fy(\zeta) \overline{\fy(u)} 
\, \dd u \, \dd \zeta. 
\end{aligned}
\end{equation}

Writing for any $L,M \in \no$
\begin{align*}
& e^{i \left( \la u, \eta - \xi + \zeta \ra + \la \zeta, x - y \ra \right)} 
=  \eabs{\eta - \xi + \zeta}^{- 2M }( 1 - \Delta_u )^M e^{i \left( \la u, \eta - \xi + \zeta \ra + \la \zeta, x - y \ra \right)} \\
& =  \eabs{ x - y }^{ - 2 L} \eabs{\eta - \xi + \zeta}^{- 2M }( 1 - \Delta_u )^M e^{i \la u, \eta - \xi + \zeta \ra} 
( 1 - \Delta_\zeta )^L e^{i \la \zeta, x - y \ra } \\
\end{align*}
and integrating by parts we get
\begin{equation}\label{eq:kernelexpression1}
|K(x, \xi ; y, \eta)| 
= C \eabs{ x - y }^{ - 2 L}  \left| \int_{\rr {2d}} 
e^{i \left( \la u, \eta-\xi \ra + \la \zeta, x - y \ra \right)}
\lambda_{L,M} (x,\xi,\eta,u,\zeta) \, \dd u \, \dd \zeta \right|
\end{equation}
where $C > 0$ and 
\begin{align*}
& \lambda_{L,M} (x,\xi,\eta,u,\zeta) \\
& = ( 1 - \Delta_\zeta )^L \left( e^{i \la u, \zeta \ra } \eabs{\eta - \xi + \zeta}^{- 2M }
( 1 - \Delta_u )^M 
\left( \left( 1- a (u + x,\zeta + \eta) \right) \wh \fy(\zeta) \overline{\fy(u)} \right) \right).
\end{align*}

Let $(x,\xi) \in \Omega_{\rho,\ep}$ and $L \in \no$.
In the integral \eqref{eq:kernelexpression1} we have $(u + x,\zeta + \eta) \in \rr {2d} \setminus \Omega_{\rho,\delta}$
due to \eqref{eq:suppsymbol1}.
By Lemma \ref{lem:distance} we thus have if $|(x,\xi)| \geqs 1$
\begin{equation*}
\eabs{ (x,\xi) }^L 
\lesssim \eabs{(u,\zeta + \eta - \xi)}^{\frac1\rho L \max \left( \sigma, \frac{1}{\sigma} \right)}.
\end{equation*}
This allows us to estimate \eqref{eq:kernelexpression1}, taking into account $a \in G_\rho^{0,\sigma}$, $\fy \in \cS(\rr d)$, 
and using \eqref{eq:Peetre}, for any $L, M, \ell \in \no$
\begin{equation}\label{eq:kernelestimate1}
\begin{aligned}
& \eabs{ (x,\xi) }^L |K(x, \xi ; y, \eta)| \\
& \lesssim \eabs{ x - y }^{ - 2 L}  
\int_{\rr {2d}}
\eabs{(u,\zeta + \eta - \xi)}^{\frac1\rho L \max \left( \sigma, \frac{1}{\sigma} \right)}
\eabs{\eta - \xi + \zeta}^{- 2M } \eabs{\zeta}^{-\ell}  \eabs{u}^{-\ell} \, \dd u \, \dd \zeta \\
& \lesssim 
\eabs{ x - y }^{ - 2 L}  
\eabs{ \xi - \eta }^{ \frac1\rho L \max \left( \sigma, \frac{1}{\sigma} \right) - 2 M}  
\int_{\rr {2d}}
 \eabs{\zeta}^{\frac1\rho L \max \left( \sigma, \frac{1}{\sigma} \right) + 2 M-\ell}  \eabs{u}^{ \frac1\rho L \max \left( \sigma, \frac{1}{\sigma} \right) - \ell} \, \dd u \, \dd \zeta \\
 & \lesssim 
\eabs{ (x - y,\xi-\eta) }^{ - L}  
\end{aligned}
\end{equation}
provided $M \geqs \frac{L}{2} \left( 1 + \frac{1}{\rho} \max \left( \sigma, \frac{1}{\sigma} \right) \right)$ and $\ell$ is sufficiently large. 

Let $N \in \no$.
Finally we obtain from \eqref{eq:STFTkernel} with $c = 1 - a$ and \eqref{eq:kernelestimate1},
using \eqref{eq:Peetre} and \eqref{eq:STFTtempered} for some $r \geqs 0$, 
for $z \in \Omega_{\rho,\ep}$
\begin{align*}
\eabs{z}^N |V_\fy \left( ( 1 - a ) (x,D) u \right) (z)|
& \lesssim \int_{\rr {2d}} \eabs{z}^N |K( z, w)| \, |V_\fy u(w)| \, \dd w \\
& \lesssim \int_{\rr {2d}} \eabs{z}^{N-L} \eabs{z-w}^{-L} \eabs{w}^{r + 2 d + 1} \eabs{w}^{- 2 d - 1} \, \dd w \\
& \lesssim \int_{\rr {2d}} \eabs{z}^{N + r + 2 d + 1 - L} \eabs{z-w}^{r + 2 d + 1 - L} \eabs{w}^{- 2 d - 1} \, \dd w < \infty
\end{align*}
if $L \in \no$ is sufficiently large. 
We have shown \eqref{eq:reduceddecay1} for all $N \geqs 0$ which proves \eqref{eq:decayestimate1}.

To prove the opposite claim we assume that \eqref{eq:decayestimate1} holds for some $0 < \ep < 1$. 
Let $0 < \delta < \min(b_\sigma,\ep)$, let $0 < \delta < \gamma  < \nu < \lambda < \min(b_\sigma,\ep)$ and define $a = q_{\gamma,\nu,\rho,\sigma,\mu,\Omega} \in G_\rho^{0,\sigma}$ as in Proposition \ref{prop:cutoffsymbol}.
Lemma \ref{lem:closureinclusion} then implies that $\supp a \subseteq \Omega_{\rho,\lambda}$. 

Define the localization (or Anti-Wick) operator \cite{Nicola1,Shubin1} $A_a$ as 
\begin{equation}\label{eq:locop}
A_a u(y) = (2 \pi)^{- \frac{d}{2} } \int_{\rr {2d}} a(x,\xi) V_\psi u(x, \xi) M_\xi T_x \psi (y) \, \dd x \, \dd \xi
\end{equation}
where $\psi(x) 
= \pi^{ - \frac{d}{4} } e^{- \frac12 |x|^2} \in \cS(\rr d)$ is a Gaussian.

By \cite[Proposition~1.7.9]{Nicola1} we have $A_a = b^w(x,D)$ where $b = \pi^{-d} a * e^{- | \cdot |^2}$
We estimate for $\alpha, \beta \in \nn d$ using
\eqref{eq:Peetreaniso} and \eqref{eq:symbolderivative1} 
\begin{align*}
\left| \pdd x \alpha \pdd \xi \beta b (x,\xi) \right|
& = \pi^{-d} \left| \int_{\rr {2d}} \pdd x \alpha \pdd \xi \beta a(x-y,\xi-\eta) e^{-| y |^2 - | \eta |^2} \, \dd y \, \dd \eta \right| \\
& \lesssim \int_{\rr {2d}} \theta_\sigma (x-y,\xi-\eta)^{- \rho ( |\alpha| + \sigma |\beta|)} e^{-| y |^2 - | \eta |^2} \, \dd y \, \dd \eta \\
& \lesssim \theta_\sigma (x,\xi)^{- \rho ( |\alpha| + \sigma |\beta|)}
\int_{\rr {2d}} \theta_\sigma (y,\eta)^{ \rho ( |\alpha| + \sigma |\beta|)} e^{-| y |^2 - | \eta |^2} \, \dd y \, \dd \eta \\
& \lesssim \theta_\sigma (x,\xi)^{- \rho ( |\alpha| + \sigma |\beta|)}
\end{align*}
which shows that $b \in G_\rho^{0,\sigma}$. 

Let $0 < r \leqs \gamma - \delta$
and let $(x,\xi) \in \Omega_{\rho,\delta}$. 
If $|(y,\eta)| \leqs r$ then $(x-y,\xi-\eta) \in \Omega_{\rho,\gamma}$ 
follows from \eqref{eq:anisoneighb},
and therefore $a(x-y,\xi-\eta) = 1$ from its construction.
Hence we get if $(x,\xi) \in \Omega_{\rho,\delta}$
\begin{align*}
\left| b (x,\xi) \right|
& = \pi^{-d} \int_{\rr {2d}} a(x-y,\xi-\eta) e^{-| y |^2 - | \eta |^2} \, \dd y \, \dd \eta \\
& \geqs \pi^{-d} \int_{\rB_r} a(x-y,\xi-\eta) e^{-| y |^2 - | \eta |^2} \, \dd y \, \dd \eta \\
& = \pi^{-d} \int_{\rB_r} e^{-| y |^2 - | \eta |^2} \, \dd y \, \dd \eta 
= C_r > 0
\end{align*}
uniformly over $(x,\xi) \in \Omega_{\rho,\delta}$. 
This shows that the symbol $b$ is elliptic in $\Omega_{\rho,\delta}$.

By the assumption \eqref{eq:decayestimate1} and Proposition \ref{prop:windowinvariance} we have 
\begin{equation*}
\sup_{z \in \Omega_{\rho,\lambda}}
\eabs{z}^N |V_\psi u(z)| < \infty \quad \forall N \geqs 0. 
\end{equation*}
From \eqref{eq:locop}, $\supp a \subseteq \Omega_{\rho,\lambda}$ and $\psi \in \cS$ we get for any $\alpha,\beta \in \nn d$
\begin{align*}
\left| y^\alpha \pdd y \beta b^w(x,D) u (y) \right|
& \lesssim \int_{\Omega_{\rho,\lambda}} \left| V_\psi u(x, \xi) \right| \eabs{y}^{|\alpha|} \left| \pdd y \beta \left( M_\xi T_x \psi (y) \right) \right| \, \dd x \, \dd \xi \\
& \lesssim \sum_{|\gamma| \leqs |\beta|} \int_{\Omega_{\rho,\lambda}} \left| V_\psi u(x, \xi) \right|
\eabs{x}^{|\alpha|} \eabs{\xi}^{|\gamma|} \eabs{y - x}^{|\alpha|} \left| \partial^{\beta-\gamma} \psi(y-x) \right|
\, \dd x \, \dd \xi \\
& < C_{\alpha,\beta}
\end{align*}
where $C_{\alpha,\beta} > 0$ for all $y \in \rr d$, which shows that $b^w(x,D) u \in \cS$.
We have shown that $u$ is smooth in $\Omega_{\rho,\delta}$ as claimed. 
\end{proof}

Combining Proposition \ref{prop:ellipticlarger} with Theorem \ref{thm:characsing} gives the following characterization of the filter of anisotropic singularities in terms of the STFT. 

\begin{cor}\label{cor:filtercharacterization}
Suppose $\sigma \in \qo_+$ and $0 < \rho \leqs 1$, let
$u \in \cS'(\rr d)$, let $\fy \in \cS(\rr d) \setminus 0$, and let $\emptyset \neq \Omega \subseteq \rr {2d}$. 
Then $\rr {2d} \setminus \Omega \in \fF_\rho(u)$
if and only if there exists $\ep > 0$ such that 
\begin{equation}\label{eq:decayestimate2}
\sup_{z \in \Omega_{\rho,\ep}}
\eabs{z}^N |V_\fy u (z)| < \infty \quad \forall N \geqs 0.
\end{equation}
\end{cor}

For a given subset $\Omega \subseteq T^* \rr d$ we may define the filter $\fF_\Omega = \{ \Gamma \in \rP (\rr{2d}): \ \Omega \subseteq \Gamma \} \subseteq \rP (\rr{2d})$ of subsets containing $\Omega$.
Next we show that for a given Lebesgue measurable $\Omega \subseteq T^* \rr d$ there exists $u \in \cS'(\rr d)$ such that 
$\Omega_{\rho,\ep} \in \fF_\rho(u)$ for any $\ep > 0$. 
This construction reveals an aspect of flexibility of filters.
We do however not address the question whether a filter of anisotropic singularities may be any subset of $\rP(\rr {2d})$ that satisfies 
the filter axioms in Definition \ref{def:generalfilter}. 

We use the indicator function denoted $\chi_\Omega$ of a set $\emptyset \neq \Omega \subseteq T^* \rr d$, 
assumed Lebesgue measurable,
and
\begin{equation}\label{eq:distributionindicator}
u = V_\fy^* \chi_\Omega = (2 \pi)^{- \frac{d}{2} } \int_{\Omega}  \Pi(z) \fy \, \dd z \in \cS'(\rr d)
\end{equation}
where $\fy \in \cS(\rr d) \setminus 0$, cf. \eqref{eq:STFTadjoint}.
If $\Omega$ has Lebesgue measure zero then $u = 0$.
The following result is true for any $\emptyset \neq \Omega \subseteq \rr {2d}$
but non-trivial only if the Lebesgue measure of $\Omega$ is positive.

\begin{prop}\label{prop:distributionindicator}
Suppose $\sigma \in \qo_+$ and $0 < \rho \leqs 1$. 
Let $\emptyset \neq \Omega \subseteq \rr {2d}$ be Lebesgue measurable, 
and let $\fy \in \cS(\rr d) \setminus 0$. 
There exists $u \in \cS'(\rr d)$ such that $\Omega_{\rho,\ep} \in \fF_\rho (u)$ for any $0 < \ep < 1$
where $\Omega_{\rho,\ep}$ is defined by \eqref{eq:anisoneighb}. 
\end{prop}

\begin{proof}
If $0 < \delta < \ep$ then it follows from Proposition \ref{prop:disjoint} that there exists $\mu > 0$ such that 
\begin{equation}\label{eq:complementinclusion}
\left( \rr {2d} \setminus \Omega_{\rho,\ep} \right)_{\rho,\mu}
\subseteq \rr {2d} \setminus \Omega_{\rho,\delta}
\end{equation}
with reference to \eqref{eq:anisoneighb}. 
If $u$ is defined by \eqref{eq:distributionindicator} then
\begin{equation}\label{eq:STFTindicator}
|V_\fy u (w)| \lesssim \int_{\Omega} |V_\fy \fy (z-w)| \, \dd z 
\end{equation}
and Lemma \ref{lem:distance} gives if $z \in \Omega$ and $w \in \rr {2d} \setminus \Omega_{\rho,\delta}$
\begin{equation}\label{eq:bracketestimate1}
\eabs{ z } \lesssim \eabs{ w - z }^{\frac1\rho \max \left( \sigma, \frac{1}{\sigma} \right)}.
\end{equation}

Combining \eqref{eq:STFTindicator}, \eqref{eq:bracketestimate1}, \eqref{eq:Peetre} and \eqref{eq:STFTschwartz}
we obtain if $w \in \rr {2d} \setminus \Omega_{\rho,\delta}$ for any $N \geqs 0$
\begin{align*}
\eabs{ w }^N |V_\fy u (w)| 
& \lesssim \int_{\Omega} \eabs{z}^N \eabs{ w - z }^N \, |V_\fy \fy (z-w)| \, \dd z \\
& \lesssim \int_{\Omega} \eabs{ w - z }^{N \left( 1 + \frac1\rho \max \left( \sigma, \frac{1}{\sigma} \right) \right)} \, |V_\fy \fy (z-w)| \, \dd z < \infty.
\end{align*}

Theorem \ref{thm:characsing} and \eqref{eq:complementinclusion} now shows that $u$ is smooth in $\rr {2d} \setminus \Omega_{\rho,\ep}$, that is $\Omega_{\rho,\ep} \in \fF_\rho(u)$.
\end{proof}

\begin{rem}\label{rem:TemperedNotSchwartz}
Let $\emptyset \neq \Omega \subseteq T^* \rr d$ be Lebesgue measurable, 
and let $u \in \cS'(\rr d)$ be defined by 
\eqref{eq:distributionindicator}. 
If $\Omega$ is bounded then $u \in \cS(\rr d)$. This is an immediate consequence of \eqref{eq:STFTindicator}, 
\eqref{eq:Peetre} and \eqref{eq:STFTschwartz}.
If on the other hand $u \in \cS'(\rr d) \setminus \cS(\rr d)$ then $\Omega$ hence must be unbounded. 
By Remark \ref{rem:Schwartzfilter} there exists 
$\Gamma \subseteq \rr {2d}$ such that $\Gamma \notin \fF_\rho (u)$. 
From Proposition \ref{prop:distributionindicator}  and property (2) of a filter (cf. Definition \ref{def:generalfilter}) 
it follows that $\Omega_{\rho,\ep} \setminus \Gamma \neq \emptyset$ for any $0 < \ep < 1$.
\end{rem}

\subsection{Metaplectic properties of filters of anisotropic singularities}\label{subsec:metaplectic}

By the proof of \cite[Lemma~4.7]{Wahlberg2} we have for $u \in \cS'(\rr d)$, $\fy \in \cS(\rr d) \setminus 0$ and $\chi \in \Sp(d,\ro)$
\begin{equation}\label{eq:STFTmetaplectic}
\left| V_{ \mu(\chi) \fy} ( \mu(\chi) u )( \chi z)\right| = \left| V_\fy u ( z)\right|, \quad z \in \rr {2d}, \quad u \in \cS'(\rr d). 
\end{equation}
where $\chi \mapsto \mu(\chi)$ is the metaplectic representation (cf. Section \ref{subsec:psidiocalc}), 
and $\mu(\chi)$ is a homeomorphism both on $\cS$ and on $\cS'$ for each $\chi \in \Sp(d,\ro)$.

For the Gabor wave front set, where $\sigma = 1$, this leads to the symplectic invariance \cite[Proposition~2.2]{Hormander2}
\begin{equation}\label{eq:WFgmetaplectic}
\WFg ( \mu(\chi) u) = \chi \WFg(u), \quad \chi \in \Sp(d,\ro), \quad u \in \cS'(\rr d).
\end{equation}

When the anisotropy parameter satisfies $0 < \sigma \neq 1$ the metaplectic invariance \eqref{eq:WFgmetaplectic} is no longer valid. 
It is nevertheless true for certain symplectic matrices $\chi \in \Sp(d,\ro)$ \cite[Proposition~4.3]{Rodino4}. 
In fact according to \cite[Proposition~4.10]{Folland1} each matrix $\chi \in \Sp(d,\ro)$ is a finite product of matrices in $\Sp(d,\ro)$ of the form
\begin{equation*}
\J, \quad 
\chi_A = \left(
  \begin{array}{cc}
  A^{-1} & 0 \\
  0 & A^{T}
  \end{array}
\right), 
\quad
\left(
  \begin{array}{cc}
  I & 0 \\
  B & I
  \end{array}
\right), 
\end{equation*}
for $A \in \GL(d,\ro)$ and $B \in \rr {d \times d}$ symmetric, where 
\begin{equation}\label{eq:Jdef}
\J =
\left(
\begin{array}{cc}
0 & I_d \\
- I_d & 0
\end{array}
\right) \in \rr {2d \times 2d}.
\end{equation}

The corresponding metaplectic operators are 
$\mu(\J) = \cF$, $\mu(\chi_A) = \mu_A$ with 
\begin{equation*}
\mu_A f(x) 
= |A|^{\frac12} f(Ax), 
\end{equation*}
if $A \in \GL(d,\ro)$, and 
\begin{equation*}
\mu 
\left(
  \begin{array}{cc}
  I & 0 \\
  B & I
  \end{array}
\right) f(x)
= e^{\frac{i}{2} \la B x, x \ra} f(x)
\end{equation*}
if $B \in \rr {d \times d}$ is symmetric. 

If $u \in \cS'(\rr d)$ and $\sigma > 0$ it follows from \cite[Proposition~4.3]{Rodino4} that
\begin{equation}\label{eq:WFgsmetaplectic1}
\WFgs (\wh u) = \J \WF_{\rm g}^{\frac1\sigma} (u), 
\end{equation}
and if $A \in \GL(d,\ro)$ 
then 
\begin{equation}\label{eq:WFgsmetaplectic2}
\WFgs ( \mu_A u) = \chi_A \WFgs (u). 
\end{equation}

The following result is a version of \eqref{eq:WFgsmetaplectic1} and \eqref{eq:WFgsmetaplectic2} for filters of anisotropic singularities. 
In the sequel we use the notation for $\Psi \subseteq \rP(\rr {2d})$
and $F: \rr {2d} \to \rr {2d}$
\begin{equation*}
F \Psi = \{ F \Omega \subseteq \rr {2d}: \, \Omega \in \Psi \} \subseteq \rP(\rr {2d}). 
\end{equation*}

\begin{prop}\label{prop:metaplectic}
Let $\sigma \in \qo_+$, $0 < \rho \leqs 1$, 
$u \in \cS'(\rr d)$ and $A \in \GL(d,\ro)$. 
Then
\begin{equation}\label{eq:filtermetaplectic1}
\fF_{\frac1\sigma,\rho} (\wh u) = \J \fF_{\sigma,\rho} (u)
\end{equation}
and 
\begin{equation}\label{eq:filtermetaplectic2}
\fF_{\sigma,\rho} (\mu_A u) = \chi_A \fF_{\sigma,\rho} (u). 
\end{equation}
\end{prop}

\begin{proof}
Let $\fy \in \cS(\rr d) \setminus 0$. 
Suppose that $u \in \cS'(\rr d)$ is $(\sigma,\rho)$-smooth in $\emptyset \neq \Omega \subseteq T^* \rr d$. 
Then it follows from \eqref{eq:STFTmetaplectic},
Proposition \ref{prop:ellipticlarger} and Theorem \ref{thm:characsing}
that there exists $\ep > 0$ such that 
\begin{equation}\label{eq:STFTdecaymetaplectic}
\sup_{z \in \chi (\Omega_{\rho,\ep} )} \eabs{z}^N \left| V_{ \mu(\chi) \fy} ( \mu(\chi) u )(z)\right| < \infty \quad \forall N \geqs 0
\end{equation}
for any $\chi \in \Sp(d,\ro)$.  

Let $(x,\xi) \in \left( \J \Omega \right)_{\rho, C_\sigma^{-2 \max(1,\sigma)} \ep}^\frac1\sigma$ defined by \eqref{eq:anisoneighb}
where $C_\sigma \geqs 1$ is defined by \eqref{eq:qtrieqconstant}. 
By \eqref{eq:anisoneighb} there exists $(y,\eta) \in \J \Omega$ such that 
\begin{equation*}
| x - y | < \ep C_\sigma^{-2 \max(1,\sigma)} \theta_\frac1\sigma^\rho (y,\eta) 
\quad \mbox{ and } \quad 
|\xi - \eta| < \ep C_\sigma^{-2 \max(1,\sigma)} \theta_\frac1\sigma^{ \frac\rho\sigma} (y,\eta).
\end{equation*}
Using \eqref{eq:quasitriangleineq} and $0 < \rho \leqs 1$ we obtain
\begin{align*}
& | x - y | < \ep \left( 1 + |y|^\frac1\sigma + |\eta| \right)^{\rho \sigma}
= \ep \theta_{\sigma} ( \J^{-1} (y,\eta) )^{\rho \sigma}, \\
& |\xi - \eta| < \ep \left( 1 + |y|^\frac1\sigma + |\eta| \right)^{\rho}
= \ep \theta_{\sigma} ( \J^{-1}(y,\eta) )^{\rho}.
\end{align*}
Taking into account $\J^{-1} (y,\eta) = (-\eta, y) \in \Omega$ 
it follows from \eqref{eq:anisoneighb} that $\J^{-1} (x,\xi) = (-\xi, x) \in \Omega_{\rho, \ep}^\sigma$. 

We have now shown
\begin{equation*}
\left( \J \Omega \right)_{\rho, C_{\sigma}^{-2 \max(1,\sigma)} \ep}^\frac1\sigma
\subseteq \J \left( \Omega_{\rho,\ep}^\sigma \right)
\end{equation*}
so it follows from $\mu(\J) = \cF$, \eqref{eq:STFTdecaymetaplectic} and Theorem \ref{thm:characsing} that $\wh u$ is $(\frac1\sigma,\rho)$-smooth in $\J \Omega$. 
By the same argument and $\mu(\J^{-1}) = \cF^{-1}$ the opposite implication is true: If  
$\wh u$ is $(\frac1\sigma,\rho)$-smooth in $\J \Omega$ then 
$u$ is $(\sigma,\rho)$-smooth in $\Omega$.
We have proved \eqref{eq:filtermetaplectic1}.

Secondly it follows from \eqref{eq:anisoneighb} that if $\ep,\delta > 0$ satisfy 
$\delta \alpha^{1 + \max(\sigma,\frac1\sigma)} \leqs \ep$ where $\alpha = \max \left( \| A^{-1} \|, \| A \| \right)$, then
\begin{equation*}
( \chi_A \Omega )_{\rho,\delta} 
\subseteq
\chi_A ( \Omega_{\rho,\ep} ).
\end{equation*}
It follows from Theorem \ref{thm:characsing} and \eqref{eq:STFTdecaymetaplectic}
\begin{equation}\label{eq:filtermatrixinclusion}
\chi_A  \fF_{\sigma,\rho} ( u ) \subseteq \fF_{\sigma,\rho} ( \mu_A u).
\end{equation}
From this we get 
\begin{equation*}
\chi_{A^{-1}}  \fF_{\sigma,\rho} ( \mu_A u ) \subseteq \fF_{\sigma,\rho} ( u)
\end{equation*}
due to $\mu_{A^{-1}} \mu_A u = u$,
and the fact that $\chi_{A^{-1}} = \chi_A^{-1}$ proves the inclusion opposite to \eqref{eq:filtermatrixinclusion}. 
We have proved \eqref{eq:filtermetaplectic2}.
\end{proof}

\subsection{Examples of filters of anisotropic singularities}\label{subsec:examples}

The following result concerns sufficient conditions for subsets of $\rr {2d}$ to belong, or not belong, to the anisotropic filter of singularities of $\delta_0$ and $1$ respectively. 

\begin{prop}\label{prop:filterdeltaandone}
Suppose $\sigma \in \qo_+$ and $0 < \rho \leqs 1$. 

\begin{enumerate}[\rm (i)]

\item
If 
\begin{equation}\label{eq:Omegaassumption1}
\Omega \subseteq \{ (y,\eta) \in \rr {2d}: \ \eabs{\eta} \leqs C \eabs{y}^\sigma \} 
\end{equation}
with $C > 0$
then $\rr {2d} \setminus \Omega \in \fF_{\sigma,\rho} (\delta_0)$, and 
there exists $\Gamma \subseteq \rr {2d}$ such that $\Gamma \notin \fF_{\sigma,\rho} (\delta_0)$ and 
$\Omega \cup \Gamma \subsetneq \rr {2d}$.

\item
If 
\begin{equation}\label{eq:Omegaassumption2}
\Omega \subseteq \{ (y,\eta) \in \rr {2d}: \ \eabs{y}^\sigma \leqs C \eabs{\eta} \} 
\end{equation}
with $C > 0$
then $\rr {2d} \setminus \Omega \in \fF_{\sigma,\rho} (1)$, and 
there exists $\Gamma \subseteq \rr {2d}$ such that $\Gamma \notin \fF_{\sigma,\rho} (1)$ and 
$\Omega \cup \Gamma \subsetneq \rr {2d}$.

\end{enumerate}

\end{prop}

\begin{proof}
$\rm{(i)}$
If $\fy \in \cS(\rr d) \setminus 0$ then
\begin{equation}\label{eq:STFTdelta}
V_\fy \delta_0(x,\xi) = (2 \pi)^{- \frac{d}{2}} \overline{\fy (-x)}. 
\end{equation}

Let $\ep > 0$ and let $(x,\xi) \in \Omega_{\rho,\ep}$. 
By \eqref{eq:anisoneighb} there exists $(y,\eta) \in \Omega$ such that 
\begin{equation}\label{eq:anisoneighb3}
|x-y| < \ep \theta_\sigma^\rho (y,\eta), \qquad | \xi - \eta| < \ep \theta_\sigma^{\rho \sigma} (y,\eta). 
\end{equation}
If $|y| \leqs 1$ then $|\eta| \leqs C 2^{\frac{\sigma}{2}}$ by assumption \eqref{eq:Omegaassumption1}, 
and \eqref{eq:anisoneighb3} gives
\begin{equation}\label{eq:phasevarbound1}
\eabs{(x,\xi)} \lesssim 1 \leqs \eabs{x}.
\end{equation}

If $|y| \geqs 1$ then by 
\eqref{eq:Omegaassumption1}, \eqref{eq:anisoneighb3} and $\rho \leqs 1$ 
we have 
\begin{equation}\label{eq:lowerboundx}
\begin{aligned}
|x| & \geqs |y| - \ep \left( 1 + |y| + |\eta|^{\frac1\sigma} \right)^\rho \\
& \geqs |y| - \ep |y|^\rho (2 + 2^\frac12 C^\frac1\sigma  )^\rho \\
& \geqs |y| \left( 1 - \ep (2 + 2^\frac12 C^\frac1\sigma  )^{\rho} \right) \geqs \frac12 |y|
\end{aligned}
\end{equation}
provided $\ep > 0$ is sufficiently small.
We also obtain from \eqref{eq:Omegaassumption1}, \eqref{eq:anisoneighb3} and $\rho \leqs 1$
\begin{equation}\label{eq:upperboundxi}
\begin{aligned}
|\xi| & \leqs |\eta| + \ep \left( 1 + |y| + |\eta|^{\frac1\sigma} \right)^{\rho \sigma} \\
& \leqs \eabs{y}^{\sigma} \left( C + \ep (2 + 2^\frac12 C^\frac1\sigma )^{\rho \sigma} \right). 
\end{aligned}
\end{equation}
Combining \eqref{eq:lowerboundx} and \eqref{eq:upperboundxi} we get $|\xi| \lesssim \eabs{x}^\sigma$ and 
consequently
\begin{equation}\label{eq:phasevarbound2}
\eabs{(x,\xi)} \lesssim \eabs{x}^{\max(\sigma,\frac1\sigma ) }.
\end{equation}

Finally we obtain from \eqref{eq:phasevarbound1}, \eqref{eq:phasevarbound2}, \eqref{eq:STFTdelta}
and $\fy \in \cS(\rr d)$
\begin{equation*}
\sup_{(x,\xi) \in \Omega_{\rho,\ep}} \eabs{(x,\xi)}^N |V_\fy \delta_0(x,\xi)| 
\lesssim \sup_{\rr d} \eabs{x}^{N \max(\sigma,\frac1\sigma ) } |\fy(-x)| < \infty
\end{equation*}
for any $N \geqs 0$. 
Corollary \ref{cor:filtercharacterization} now shows $\rr {2d} \setminus \Omega \in \fF_{\sigma,\rho} (\delta_0)$. 
Since $\delta_0 \in \cS' \setminus \cS$, by Remark \ref{rem:TemperedNotSchwartz}
 there exists $\Gamma \subseteq \rr {2d}$ such that 
 $\Gamma \notin \fF_{\sigma,\rho} (\delta_0)$ and 
$\Omega \cup \Gamma \subsetneq \rr {2d}$.
We have shown claim (i).

Claim (ii) follows from $\wh \delta_0 = (2 \pi)^{- \frac{d}2}$, claim (i), and Eq. \eqref{eq:filtermetaplectic1} in Proposition \ref{prop:metaplectic}. 
\end{proof}

\subsection{Sufficient conditions for $C^\infty$}\label{subsec:suffsmooth}

Here we give a filter version of \cite[Proposition~3]{Wahlberg3} which treats the (isotropic) Gabor wave front set. 

\begin{prop}\label{prop:suffsmooth}
Suppose $\sigma \in \qo_+$, $0 < \rho \leqs 1$ and $u \in \cS'(\rr d)$.
If $C > 0$, 
\begin{equation}\label{eq:Omegaassumption3}
\{ (y,\eta) \in T^* \rr d: \ \eabs{y}^\sigma \leqs C \eabs{\eta} \} 
\subseteq \Omega
\end{equation}
and $\rr {2d} \setminus \Omega \in \fF_{\sigma,\rho} (u)$ then $u \in C^\infty(\rr d)$ and there exists $p \geqs 0$ such that 
\begin{equation}\label{eq:smoothupperbound}
| \pd \alpha u (x)|
\leqs C_\alpha \eabs{x}^{p + \sigma |\alpha|}, \quad x \in \rr d, \quad \alpha \in \nn d, \quad C_\alpha > 0.
\end{equation}
\end{prop}

\begin{proof}
Assumption \eqref{eq:Omegaassumption3} implies
\begin{equation}\label{eq:Omegacompassumption1}
\rr {2d} \setminus \Omega 
\subseteq 
\{ (y,\eta) \in T^* \rr d: \ \eabs{\eta} \leqs C^{-1} \eabs{y}^\sigma \}.
\end{equation}
Let $\fy \in \cS(\rr d)$ satisfy $\| \fy \|_{L^2} = 1$. 
We use \eqref{eq:STFTinverse}
and show that the integral for $\partial^\alpha u$ is absolutely convergent for any $\alpha \in \nn d$. 
Thus we write formally
\begin{equation}\label{eq:STFTreconstruction}
\partial^\alpha u (x) = (2\pi)^{-\frac{d}{2}} \sum_{\beta \leqslant \alpha} \binom{\alpha}{\beta} \int_{\rr {2d}} V_\fy u(y,\eta) \, (i \eta )^\beta e^{i \langle \eta,x \rangle} \partial^{\alpha-\beta} \fy( x - y ) \, \dd y \, \dd \eta. 
\end{equation}

We split the domain of the integral \eqref{eq:STFTreconstruction} in two pieces.
For the first integral over $\rr {2d} \setminus \Omega$ we obtain using 
\eqref{eq:Peetre}, \eqref{eq:STFTtempered} for some $r \geqs 0$, $\fy \in \cS$, and 
\eqref{eq:Omegacompassumption1}, for any $n \geqs 0$
\begin{equation}\label{eq:integralpiece1}
\begin{aligned}
& \left| \int_{\rr {2d} \setminus \Omega} V_\fy u(y,\eta) \, (i \eta)^\beta e^{i \langle \eta,x \rangle} \partial^{\alpha-\beta} \fy(x - y) \, \dd y \, \dd \eta \right| \\
& \leqs \int_{\rr {2d} \setminus \Omega} \eabs{(y,\eta)}^{r + 2 d + 1} \, \eabs{\eta}^{|\alpha|} \, 
|\partial^{\alpha-\beta} \fy(x - y)| \,
\eabs{(y,\eta)}^{ - 2 d - 1} \, 
\dd y \, \dd \eta \\
& \leqs \int_{\rr {2d} \setminus \Omega} \eabs{y}^{(r + 2 d + 1)(1+\sigma) + \sigma |\alpha|} \,  
\eabs{x-y}^{-n} \, \eabs{(y,\eta)}^{ - 2 d - 1} \, \dd y \, \dd \eta \\
& \lesssim \eabs{x}^{(r + 2 d + 1)(1+\sigma) + \sigma |\alpha|} 
\end{aligned}
\end{equation}
provided $n = (r + 2 d + 1)(1+\sigma) + \sigma |\alpha|$. 

For the remaining part of the integral we use the assumption that $u$ is smooth in $\Omega$
and Corollary \ref{cor:filtercharacterization}. 
Thus for any $N \geqs 0$
\begin{equation}\label{eq:integralpiece2}
\begin{aligned}
\left| \int_{\Omega} V_\fy u(y,\eta) \, (i \eta)^\beta e^{i \langle \eta,x \rangle} \partial^{\alpha-\beta} \fy(x-y) \, \dd y \, \dd \eta \right| 
& \lesssim \int_{\Omega} \eabs{(y,\eta)}^{-N} \, \eabs{\eta}^{|\alpha|} \, |\partial^{\alpha-\beta} \psi(x-y)| \, \dd y \, \dd \eta \\
& \lesssim \int_{\rr {2d}} \eabs{(y,\eta)}^{-N + |\alpha|}  \, \dd y \, \dd \eta 
\lesssim 1
\end{aligned}
\end{equation}
provided $N > |\alpha| + 2d+1$. 
Combining \eqref{eq:integralpiece1} and \eqref{eq:integralpiece2} shows in view of \eqref{eq:STFTreconstruction} that $u \in C^\infty(\rr d)$ and the estimate \eqref{eq:smoothupperbound} follows with $p = (r + 2 d + 1)(1+\sigma)$. 
\end{proof}

\section{Microlocality and microellipticity}\label{sec:microlocmicroell}

First we show a microlocal result for the filter of anisotropic singularities and 
pseudodifferential operators with anisotropic Shubin symbols. 

\begin{thm}\label{thm:microlocal}
If $\sigma \in \qo_+$, $0 < \rho \leqs 1$, $r \in \ro$, $a \in G_\rho^{r,\sigma}$ and $u \in \cS'(\rr d)$ then 
\begin{equation*}
\fF_\rho(u) \subseteq \fF_\rho( a^w(x,D) u). 
\end{equation*}
\end{thm}

\begin{proof}
We first observe that $\rr {2d} \in \fF_\rho(u) \cap \fF_\rho(a^w u)$ 
by Properties (1) and (2) of a filter, cf. Definition \ref{def:generalfilter}.
Next we show that if $u$ is smooth in $\emptyset \neq \Omega \subseteq \rr {2d}$, then $a^w u$ is also 
smooth in $\Omega$. 
Taking the complement $\rr {2d} \setminus \Omega$ proves the statement. 

Thus we assume that $u$ is smooth in $\emptyset \neq \Omega \subseteq \rr {2d}$.
By Proposition \ref{prop:smoothorderzero} there exists $0 < \ep_1 < \delta_1 < 1$ and $q_1 = q_{\ep_1,\delta_1,\rho,\Omega} \in G_\rho^{0,\sigma}$
defined by \eqref{eq:cutoffsymboldef} such that $q_1^w u \in \cS(\rr d)$. 
Let $0 < \ep_2 < \delta_2 < \ep_1$ and set $q_2 = q_{\ep_2,\delta_2,\rho,\Omega} \in G_\rho^{0,\sigma}$
again defined by \eqref{eq:cutoffsymboldef}. 
By Remark \ref{rem:supportintersection}
we have
\begin{equation*}
\supp q_2 \cap \supp(1-q_1) = \emptyset
\end{equation*}
if $\delta_2$ is sufficiently small.
The calculus then gives $r : = q_2 \wpr a \wpr ( 1 - q_1) \in \cS(\rr {2d})$ and 
\begin{equation*}
q_2 \wpr a = q_2 \wpr a \wpr q_1 + r. 
\end{equation*}

From $q_1^w u \in \cS$, \eqref{eq:SchwartzCont} and the regularization $r^w: \cS' \to \cS$, 
we conclude 
\begin{equation*}
q_2^w a^w u = q_2^w a^w q_1^w u + r^w u \in \cS(\rr d). 
\end{equation*}
Since $q_2$ is elliptic on $\Omega_{\rho,\ep_2} \supseteq \Omega$ we have shown that $a^w u$ is smooth in $\Omega$. 
\end{proof}

Theorem \ref{thm:microlocal} combined with Proposition \ref{prop:filterWF} gives a short proof of a particular case of the microlocal inclusion \cite[Proposition~5.1]{Rodino4}, cf. \eqref{eq:microlocalityWFs}: 

\begin{cor}\label{cor:microlocal}
Let $\sigma \in \qo_+$ and $0 < \rho \leqs 1$. 
If $r \in \ro$ and $a \in G_\rho^{r,\sigma}$ then 
\begin{equation}\label{eq:microlocal}
\WFgs ( a^w(x,D) u) \subseteq \WFgs ( u).  
\end{equation}
\end{cor}

\begin{proof}
Suppose $0 \neq z \notin \WFgs ( u)$. 
By Proposition \ref{prop:filterWF} there exists $\Omega \in \fF_\rho(u) \cap \rP(\rr {2d} \setminus 0)$ which is closed in $\rr {2d} \setminus 0$, 
$\sigma$-conic, and $z \notin \Omega$. 
By Theorem \ref{thm:microlocal} we get $\Omega \in \fF_\rho( a^w u) \cap \rP(\rr {2d} \setminus 0)$
which by means of another appeal to Proposition \ref{prop:filterWF} yields $z \notin \WFgs ( a^w u)$. 
We have shown \eqref{eq:microlocal}.
\end{proof}

The next result is a version of \cite[Lemma~6.3]{Rodino4} adapted from $\sigma$-conic to general subsets of phase space, 
cf. the proof of Proposition \ref{prop:smoothorderzero}.
We omit the proof since it is conceptually identical to the proof of \cite[Lemma~6.3]{Rodino4}. 

\begin{lem}\label{lem:microellipticparametrix}
Let $\sigma \in \ro_+$, $0 < \rho \leqs 1$, $r \in \ro$, $a \in G_\rho^{r,\sigma}$, 
let $\emptyset \neq \Omega \subseteq \rr {2d}$ be open, 
and suppose $a$ is elliptic in $\Omega$. 
The there exists $R > 0$ such that 
for any $\chi \in G_\rho^{0,\sigma}$
such that 
$\supp \chi \subseteq \Omega \setminus \rB_R$ 
there exists $b \in G_\rho^{-r,\sigma}$ such that 
\begin{equation*}
b \wpr a - \chi \in \cS(\rr {2d}).
\end{equation*}
\end{lem}

The following result is a microelliptic inclusion for the filter of anisotropic singularities, cf. \eqref{eq:microellipticityWFs}. 

\begin{thm}\label{thm:microelliptic}
Let $\sigma \in \qo_+$, $0 < \rho \leqs 1$ and $r \in \ro$. 
If $a \in G_\rho^{r,\sigma}$ and $u \in \cS'(\rr d)$ then 
\begin{equation*}
\left\{ \Omega \in \fF_\rho( a^w(x,D) u): \ a \mbox{ is elliptic on } \ \rr {2d} \setminus \Omega \right\}
\subseteq \fF_\rho(u). 
\end{equation*}
\end{thm}

\begin{proof}
Take $\rr {2d} \setminus \Omega \in \fF_\rho( a^w(x,D) u)$ such that $a$ is elliptic in $\Omega$. 
By Definition \ref{def:smooth} there exists $b \in G_\rho^{t,\sigma}$ such that $b^w a^w u \in \cS$
and $b$ is elliptic in $\Omega$. According to Proposition \ref{prop:ellipticlarger} we may assume that both $a$ and $b$ are 
elliptic in $\Omega_{\rho,\mu}$ defined by \eqref{eq:anisoneighb} for some $\mu > 0$. 
By the calculus the symbol $b_1 = b \wpr a \in G_\rho^{r + t,\sigma}$ is elliptic in $\Omega_{\rho,\mu}$. 
Let $0 < \ep < \delta < \min(\mu,b_\sigma)$ with $b_\sigma$ defined by \eqref{eq:bsigma}, 
and pick $q = q_{\ep,\delta,\rho,\Omega} \in G_\rho^{0,\sigma}$ as in 
Proposition 
\ref{prop:cutoffsymbol}
so that $q$ is elliptic in $\Omega_{\rho,\ep}$, 
and $\supp q \subseteq \overline{\left(\Omega_{\rho,\delta}\right)} \subseteq \Omega_{\rho,\mu}$
by Lemma \ref{lem:closureinclusion}. 

We may now use Lemma \ref{lem:microellipticparametrix} applied to $b_1$ and $q$. 
Thus there exists $b_2 \in G_\rho^{- r - t,\sigma}$ such that $r = b_2 \wpr b_1 - q \in \cS(\rr {2d})$.
Hence 
\begin{equation*}
q^w u 
= b_2^w b_1^w u - r^w u 
= b_2^w b^w a^w u - r^w u \in \cS
\end{equation*}
due to $b^w a^w u \in \cS$, 
\eqref{eq:SchwartzCont}
and the regularization $r^w: \cS' \to \cS$. 
It follows that $u$ is smooth in $\Omega \subseteq \Omega_{\rho,\ep}$
which shows that $\rr {2d} \setminus \Omega \in \fF_\rho(u)$. 
\end{proof}

\begin{cor}\label{cor:microelliptic}
Suppose $\sigma \in \qo_+$, $0 < \rho \leqs 1$, 
$a \in G_\rho^{r,\sigma}$, $a$ is elliptic, and $u \in \cS'(\rr d)$. 
Then 
\begin{equation*}
\fF_\rho (u) = \fF_\rho( a^w(x,D) u). 
\end{equation*}
\end{cor}

\begin{proof}
The inclusion $\fF_\rho (u) \subseteq \fF_\rho ( a^w(x,D) u)$ 
is a consequence of Theorem \ref{thm:microlocal}, and conversely 
$\fF_\rho ( a^w(x,D) u) \subseteq \fF_\rho (u)$ follows from 
Theorem \ref{thm:microelliptic}.
\end{proof}

\section{Application to evolution equations of anisotropic Schr\"odinger type}\label{sec:schrodinger}

For $T > 0$ we consider the initial value Cauchy problem
\begin{equation}\label{eq:anharmonicCP}
\begin{cases} 
\partial_t u + i a^w(x,D) u = 0, \qquad x \in \rr d, \qquad t \in [-T,T] \setminus 0,  \\ 
u(0,x) = u_0(x)
\end{cases}. 
\end{equation}	
As a tool we will use the corresponding inhomogeneous equation
\begin{equation}\label{eq:anharmonicCPinhom}
\begin{cases} 
\partial_t u + i a^w(x,D) u = f(t,x), \qquad x \in \rr d, \qquad t \in [-T,T] \setminus 0,  \\ 
u(0,x) = u_0(x)
\end{cases}.
\end{equation}	
The Weyl symbol $a$ of the hamiltonian will be assumed to belong to an anisotropic Shubin class $G_\rho^{r,\sigma}$. 

\begin{rem}\label{rem:cutoffsymbol}
In order to get rid of the possible lack of smoothness at the origin of symbols, 
we will use a cutoff function $\psi_\mu(x,\xi) = \fy(|x|^2 + |\xi|^2) \in C^\infty(\rr {2d})$ where $\fy \in C^{\infty}(\ro)$,
$0 \leqs \fy \leqs 1$, $\fy(t) = 0$ for $t \leqs \frac{\mu^2}{4}$ and $\fy(t) = 1$ for $t \geqs \mu^2$
for a given $\mu > 0$. 
Thus $\psi_\mu \big|_{\rB_{\frac{\mu}{2}}} \equiv 0$ and $\psi_\mu \big|_{\rr {2d} \setminus \rB_\mu} \equiv 1$.
\end{rem} 

We show results on propagation of the filter of anisotropic singularities from the initial datum ($t = 0$) to $t \in [-T,T]$, expressed with the hamiltonian flow corresponding to the real-valued principal symbol of $a$. 
Thus our results are anchored in many other results on propagation of singularities of different forms for evolution equations of various types \cite{Cappiello6,Cappiello7,Hormander1,Ito1,Nicola2,PravdaStarov1,Wahlberg4}.

For a smooth function $a: \rr {2d} \to \ro$ the corresponding hamiltonian flow is the solution to 
Hamilton's system of equations
\begin{equation} \label{eq:Hamiltoneq}
\begin{cases} 
x'(t)  = \nabla_\xi a( x(t), \xi(t) ) \\ 
\xi'(t)  = -\nabla_x a( x(t), \xi(t) )  \\ 
x(0)=x  \\ 
\xi(0)=\xi
\end{cases}
\end{equation}
where $(x,\xi) \in T^* \rr d$ is initial datum and $t \in \ro$. The Hamilton flow is denoted $( x(t), \xi(t) ) = \chi_t (x,\xi)$ for $t \in \ro$. 

The following result is a generalization of \cite[Theorem~8.3]{Cappiello6}.

\begin{thm}\label{thm:propagationfilter1}
Let $\sigma \in \qo_+$, $0 < \rho \leqs 1$, 
and suppose that $a \in G^{1 + \sigma,\sigma}$, 
$a \sim \sum_{j = 0}^{\infty} a_j$, 
where 
$a_0 \in C^\infty(\rr {2d} \setminus 0)$ is real-valued, 
\begin{equation}\label{eq:a0homogeneity}
a_0( \lambda x, \lambda^\sigma \xi)  = \lambda^{1+\sigma } a_0(x,\xi), \quad \lambda > 0, \quad (x,\xi) \in T^* \rr d \setminus 0,  
\end{equation}
and $a_j \in G^{(1+\sigma) (1-2 j ), \sigma}$ for $j \geqs 1$. 
The hamiltonian flow $\chi_t: T^* \rr d \setminus 0 \to T^* \rr d \setminus 0$
corresponding to $a_0$ is then defined for $-T \leqs t \leqs T$ for some $T > 0$.
If $u_0 \in \cS'(\rr d)$ then there exists a unique solution $u=u(t,\cdot) = \cK_t u_0 \in \cS'(\rr d)$ to \eqref{eq:anharmonicCP} for all $t \in \ro$, and we have
\begin{equation}\label{eq:propcritical}
\fF_\rho ( u(t,\cdot) ) = \chi_t \left( \fF_\rho ( u_0 ) \right), \quad t \in [-T,T]. 
\end{equation}
\end{thm}

\begin{proof}
By \cite[Proposition~6.2]{Cappiello6} there exists $T > 0$ such that 
the hamiltonian flow $\chi_t: T^* \rr d \setminus 0 \to T^* \rr d \setminus 0$
corresponding to the hamiltonian $a_0$ is well defined for $-T \leqs t \leqs T$.
According to \cite[Remark~6.3]{Cappiello6} we may extend 
$\chi_t: T^* \rr d  \to T^* \rr d$ by $\chi_t (0,0) = (0,0)$ maintaining its continuity and bijectivity, possibly losing its smoothness. 

If $u_0 \in \cS$ then by \cite[Corollary~7.11]{Cappiello6} we have a unique solution $u \in C( [-T,T], \cS)$
to \eqref{eq:anharmonicCP}. 
By Remark \ref{rem:Schwartzfilter} it follows that $\fF_\rho ( u(t,\cdot) ) = \fF_\rho ( u_0 ) = \rP(\rr {2d})$, so then 
both sides of \eqref{eq:propcritical} are equal to $\rP(\rr {2d})$. 
Thus we may assume that $u_0 \in \cS' \setminus \cS$ and $\fF_\rho ( u_0 ) \neq \rP(\rr {2d})$.
By \eqref{eq:SSSchwartz} there exists $s \in \ro$ such that $u_0 \in M_{\sigma,s}$. 
From \cite[Corollary~7.10]{Cappiello6} we obtain the existence of a unique solution 
$u(t) = \cK_t u_0 \in C( [-T,T], M_{\sigma,s})$ 
to \eqref{eq:anharmonicCP}. 

Let $\emptyset \neq \Omega \subseteq \rr {2d}$ satisfy
$\rr {2d} \setminus \Omega \in \fF_\rho ( u_0 )$. 
By Proposition \ref{prop:smoothorderzero}
there exists $q_0 = q_{\ep,\delta,\rho,\Omega} \in G_\rho^{0,\sigma}$ with $0 < \ep < \delta < 1$
such that $q_0^w u_0 \in \cS$ and 
\begin{equation}\label{eq:q0one}
q_0 |_{\Omega} \equiv 1.
\end{equation}

The proof of \cite[Proposition~8.2]{Cappiello6}, 
may be modified straight-forwardly from the assumption $\rho = 1$ into $0 < \rho \leqs 1$.  
Thus for any $\mu > 0$ there exists $q(t) \in C([-T,T], G_\rho^{0,\sigma})$ such that 
$q(t) \sim \sum_{j = 0}^{\infty} q_j (t)$ with $q_j (t) \in C([-T,T], G_\rho^{- 2 j \rho(1+\sigma),\sigma})$ for $j \geqs 0$, 
$q_0(t) (x,\xi) = \psi_\mu (x,\xi) q_0( \chi_{-t} (x,\xi))$, 
$q_j(0) = 0$ for all $j \geqs 1$,
and $r(t) \in L^\infty([-T,T], \cS(\rr {2d}))$ where 
\begin{equation*}
r(t)^w = q(t)^w \left( \partial_t + i a^w \right)
- \left( \partial_t +i a^w \right) q(t)^w. 
\end{equation*}

This gives 
\begin{equation*}
0 = q(t)^w \left( \partial_t + i a^w \right) u(t) 
= \left( \partial_t +i a^w \right) q(t)^w u(t) + r(t)^w u(t)
\end{equation*}
that is 
\begin{equation*}
\left( \partial_t +i a^w \right) q(t)^w u(t)  = - r(t)^w u(t) := f(t). 
\end{equation*}
By \eqref{eq:symbolintersection} and \eqref{eq:ShubinSobolevCont}
we have for any $r \in \ro$
\begin{equation*}
\sup_{|t| \leqs T} \| f(t) \|_{M_{\sigma,r+s}} 
\lesssim \sup_{|t| \leqs T} \| u(t) \|_{M_{\sigma,s}} < \infty
\end{equation*}
which by \eqref{eq:SSSchwartz} and the phrase following it implies that $f \in L^\infty ( [-T,T], \cS(\rr d)) \subseteq L^1 ( [-T,T], \cS(\rr d))$. 
Thus $q(t)^w u(t)$ solves the inhomogeneous equation \eqref{eq:anharmonicCPinhom}, 
and for the initial value we have
\begin{equation*}
q(0)^w u(0) 
= q_0(0)^w u(0) = \left( \psi_\mu q_0 \right)^w u_0 
= q_0^w u_0 + \left( (\psi_\mu-1) q_0 \right)^w u_0  \in \cS
\end{equation*}
due to $q_0^w u_0 \in \cS$ and $\psi_\mu-1 \in C_c^\infty(\rr {2d})$. 

At this point we may apply 
\cite[Corollary~7.11]{Cappiello6} which gives 
$q(t)^w u(t) \in \cS(\rr d)$ for all $t \in [-T,T]$. 
We note that $q_0 (t) (\chi_t(x,\xi)) = q_0(x,\xi)$ if $|\chi_t(x,\xi)| \geqs \mu$
which due to \eqref{eq:q0one} means that $q_0(t)|_{\left( \chi_t \Omega \right) \setminus \rB_\mu} \equiv 1$.
Hence $q(t) \in G_\rho^{0,\sigma}$ is elliptic in $\chi_t \Omega$, 
since the lower order terms $\{ q_j(t) \}_{j \geqs 1}$ in $q(t)$ decay in $\rr {2d}$. 
We have now shown that $u(t,\cdot)$ is smooth in $\chi_t \Omega$, that is, 
$\rr {2d} \setminus \chi_t \Omega = \chi_t ( \rr {2d} \setminus \Omega) \in \fF_\rho( u(t,\cdot) )$. 

From the argument we may conclude that $\chi_t \left( \fF_\rho(u_0) \right) \subseteq \fF_\rho( u(t,\cdot) )$. 
The opposite inclusion follows from $\cK_t^{-1} =  \cK_{-t}$ and $\chi_t^{-1} =  \chi_{-t}$. 
\end{proof}

\begin{rem}\label{rem:propsinganiso1}
We may deduce \cite[Theorem~8.3]{Cappiello6} as a consequence of 
Theorem \ref{thm:propagationfilter1}, 
Proposition \ref{prop:filterWF} and \cite[Proposition~6.2]{Cappiello6}. 
Indeed the latter result implies in particular 
that the hamiltonian flow, of a real-valued symbol $a_0$ that satisfies \eqref{eq:a0homogeneity}, commutes with the anisotropic dilation \eqref{eq:anisodilation}.
\end{rem}

\begin{example}\label{ex:homogensymbol}
We give two examples of principal symbols that satisfy the criterion \eqref{eq:a0homogeneity}. 

\begin{enumerate}

\item Let $m \geqs 2$, $a(x,\xi) = p_m(\xi)$ where 
\begin{equation*}
p_m(\xi) = \sum_{|\alpha| = m} c_\alpha \xi^\alpha, \quad c_\alpha \in \ro,
\end{equation*}
is a homogeneous polynomial with real coefficients of degree $m$. Then $a \in C^\infty(\rr {2d})$ satisfies 
\eqref{eq:a0homogeneity} with $\sigma = \frac1{m-1}$. 
The operator is $a^w(x,D) = p_m(D)$ is a homogeneous differential operator of order $m$ with constant real coefficients.

\item If $\alpha_j, \beta_j \in \ro$ for $1 \leqs j \leqs d$, $k, m \in \no \setminus 0$, and 
\begin{equation*}
a(x, \xi) = \left( \left( \sum_{j=1}^d \alpha_j x_j^2 \right)^k + \left( \sum_{j=1}^d \beta_j \xi_j^2 \right)^m \right)^{\frac12 \left( \frac1k + \frac1m \right)}
\end{equation*}
then $a \in C^\infty(\rr {2d} \setminus 0)$ satisfies 
\eqref{eq:a0homogeneity} with $\sigma = \frac{k}{m}$. 

\end{enumerate}

\end{example}

Next we show that a hamiltonian Weyl symbol of order smaller than $\rho(1+\sigma)$ and uniformly bounded imaginary part gives invariance of the filter of anisotropic singularities. 

\begin{thm}\label{thm:nonpropagationfilter1}
Let $\sigma \in \qo_+$,  $0 < \rho \leqs 1$,
$a \in G_1^{r,\sigma}$ with $r < \rho(1+\sigma)$, 
and suppose that
\begin{equation}\label{eq:boundedimag}
\sup_{(x,\xi) \in \rr {2d}} |\im \, a(x,\xi)| \leqs C
\end{equation}
for some $C > 0$.
Consider the Cauchy problem \eqref{eq:anharmonicCP}. 
If $u_0 \in \cS'(\rr d)$ then there exists a unique solution $u=u(t,\cdot) \in \cS'(\rr d)$ to \eqref{eq:anharmonicCP} with $T=+\infty$, 
and  
\begin{equation}\label{eq:propsubcritical}
\fF_\rho ( u (t,\cdot) ) = \fF_\rho ( u_0 ), \quad t \in \ro. 
\end{equation}
\end{thm}

\begin{proof}
By \eqref{eq:SSSchwartz} there exists $s \in \ro$ such that $u_0 \in M_{\sigma,s}$, and then 
\cite[Corollary~7.10]{Cappiello6} shows that the equation \eqref{eq:anharmonicCP} has a unique solution $u \in C(\ro, M_{\sigma,s})$.

As in the proof of Theorem \ref{thm:propagationfilter1} we may assume that $u_0 \notin \cS$ 
and $\fF ( u_0 ) \neq \rP(\rr {2d})$.
Let $\emptyset \neq \Omega \subseteq \rr {2d}$ satisfy
$\rr {2d} \setminus \Omega \in \fF_\rho ( u_0 )$. 
By Proposition \ref{prop:smoothorderzero} there exist $0 < \ep_1 < \delta_1 < 1$
and $q_1 = q_{\ep_1,\delta_1,\rho,\Omega} \in G_\rho^{0,\sigma}$ 
such that $q_1^w u_0 \in \cS$ and $q_1 |_{\Omega} \equiv 1$.
From $u \in C(\ro, M_{\sigma,s})$ and 
\eqref{eq:ShubinSobolevCont}
we obtain
\begin{equation}\label{eq:modspace2}
q_1^w u \in C( \ro , M_{\sigma,s}). 
\end{equation} 

Eq.~\eqref{eq:modspace2} is the starting point for an iterative procedure. 
First we deduce from \eqref{eq:modspace2} the improved regularity
\begin{equation}\label{eq:regularization1}
q_1^w u \in C( \ro , M_{\sigma,s + \alpha}) 
\end{equation} 
where $\alpha = \rho(1+\sigma) - r > 0$ by assumption. 

In fact, cf. \cite[Theorem 23.1.4]{Hormander2} and \cite{Nicola2}, we may 
regard $u_1 = q_1^w u$ as the solution of the non-homogeneous Cauchy problem 
\begin{equation}\label{eq:SchrodEqIter1}
\begin{cases} 
P u_1 = f_1 \\ 
u_1(0,\cdot) = u_{1,0}
\end{cases}
\end{equation}
where $P = \partial_t + i a^w(x,D)$
and $u_{1,0} = q_1^w u_0 \in \cS \subseteq M_{\sigma, s + \alpha}$. 
To express $f_1$ we write, using $P u = 0$,
\begin{equation*}
f_1 = P u_1 = P q_1^w u = P q_1^w u - q_1^w Pu = [ P, q_1^w] u = [ i a^w, q_1^w ] u.
\end{equation*}
By the proof of \cite[Proposition~8.2]{Cappiello6} the commutator $b_1^w = [ i a^w, q_1^w ]$
has symbol $b_1 \in G_\rho^{r - \rho(1+\sigma),\sigma} = G_\rho^{- \alpha,\sigma}$ of negative order $-\alpha$. 

From $u \in C( \ro , M_{\sigma,s})$ and 
\eqref{eq:ShubinSobolevCont}
we obtain
\begin{equation*}
f_1  = P u_1 = b_1^w u \in C( \ro , M_{\sigma,s + \alpha} ). 
\end{equation*}
Now \eqref{eq:SchrodEqIter1}, $u_{1,0} \in M_{\sigma, s + \alpha}$ and 
\cite[Corollary~7.10]{Cappiello6}
show that \eqref{eq:regularization1} holds. 

In the next step we take $0 < \ep_2 < \delta_2 < \ep_1$
with $\delta_2 < b_\sigma$ where $b_\sigma$ is defined by \eqref{eq:bsigma},
and define $q_2 = q_{\ep_2, \delta_2,\rho,\Omega} \in G_\rho^{0,\sigma}$ as in 
Proposition \ref{prop:cutoffsymbol}.
By Remark \ref{rem:supportintersection}
we have $\supp q_2 \cap \supp (1 - q_1 ) = \emptyset$ 
which by the calculus yields $r = q_2 \wpr ( 1 - q_1 ) \in \cS(\rr {2d})$. 
Combined with $q_1^w u_0 \in \cS$
and 
\eqref{eq:ShubinSobolevCont}
we find
\begin{equation}\label{eq:schwartzmodspace3}
q_2^w u_0 = q_2^w q_1^w u_0 + r^w u_0 \in \cS
\end{equation}
since $r^w: \cS' \to \cS$ is regularizing. 

Set $u_2 = q_2^w u$ and consider the non-homogeneous Cauchy problem 
\begin{equation}\label{eq:SchrodEqIter2}
\begin{cases} 
P u_2 = f_2 \\ 
u_2(0,\cdot) = u_{2,0}
\end{cases} 
\end{equation}
where $u_{2,0} = q_2^w u_0 \in M_{\sigma, s + 2 \alpha}$ by 
\eqref{eq:SSSchwartz}. 
As above we obtain
\begin{equation*}
f_2 = P q_2^w u - q_2^w Pu = 
b_2^w u
= b_2^w q_1^w u + b_2^w (1 - q_1)^w u
\end{equation*}
with $b_2^w = [P, q_2^w] = [i a^w, q_2^w] \in G_\rho^{-\alpha,\sigma}$ as above. 

Since $\supp b_2 \subseteq \supp q_2$ we may conclude that $\supp b_2 \cap \supp (1 - q_1 ) = \emptyset$,
which implies that $b_2^w (1 - q _1)^w:  \cS' \to \cS$ is regularizing.  
Now \eqref{eq:regularization1} and again \eqref{eq:ShubinSobolevCont}
give $f_2 \in C( \ro , M_{\sigma,s + 2 \alpha} )$. 
From
\eqref{eq:SchrodEqIter2}, $u_{2,0} \in M_{\sigma, s + 2 \alpha}$ and \cite[Corollary~7.10]{Cappiello6}
we obtain
$u_2 = q_2^w u \in C( \ro , M_{\sigma,s + 2 \alpha} )$. 

This bootstrap argument shows that we have
\begin{equation}\label{eq:regularityk}
q_k^w u \in C( \ro , M_{\sigma,s + k \alpha} )
\end{equation}
for any $k \in \no \setminus 0$, with $q_k = q_{\ep_k, \delta_k,\rho,\Omega}$, 
in each step decreasing $\ep_k$ and $\delta_k$ such that $0 < \ep_k < \delta_k < \ep_{k-1}$
where $\ep_{k-1}$ has been chosen in the preceding step, taking $\delta_k < b_\sigma$ for each $k$. 
We do this keeping $\ep_k > 2 \mu > 0$ for all $k \in \no$ for a fixed $\mu > 0$. 
Let $0 < \ep < \mu$ and set $q = q_{\ep, \mu,\rho,\Omega} \in G_\rho^{0,\sigma}$ as in Proposition \ref{prop:cutoffsymbol}.
Then $0 < \ep < \mu < 2 \mu < \ep_k < \delta_k$ for all $k \in \no \setminus 0$. 

By Lemma \ref{lem:closureinclusion} we have for any $k \in \no \setminus 0$
\begin{equation*}
\supp q \cap \supp (1 - q_k ) 
\subseteq \overline{ \left( \Omega_{\rho,\mu} \right)} \cap \left( \rr {2d} \setminus \Omega_{\rho,\ep_k} \right)
\subseteq \Omega_{\rho,2 \mu } \cap \left( \rr {2d} \setminus \Omega_{\rho,\ep_k} \right)
= \emptyset
\end{equation*}
which yields $r_k := q \wpr (1-q_k) \in \cS(\rr {2d})$.
Hence for any $t \in \ro$ and any $k \in \no \setminus 0$ we get from \eqref{eq:regularityk}
\begin{equation*}
q^w u (t) = q^w q_k^w u (t) + r_k^w u (t) \in M_{\sigma,s + k \alpha}
\end{equation*}
which implies
\begin{equation*}
q^w u(t,\cdot) \in \bigcap_{s \in \ro} M_{\sigma,s} = \cS(\rr d). 
\end{equation*}
Since $q|_{\Omega} \equiv 1$ we have shown that $\rr {2d} \setminus \Omega \in \fF_\rho( u(t,\cdot) )$. 
It follows $\fF_\rho (u_0) \subseteq \fF_\rho ( u(t,\cdot) )$. 
The opposite inclusion follows from $\cK_t^{-1} =  \cK_{-t}$. 
\end{proof}

\begin{rem}\label{rem:propsinganiso2}
Again we may deduce the conclusion \cite[Theorem~1.4~Eq.~(1.11)]{Cappiello7} as a consequence of 
Theorem \ref{thm:nonpropagationfilter1} and
Proposition \ref{prop:filterWF}. 
\end{rem}

\begin{example}\label{ex:nonpropagation}
We give two examples of symbols that satisfy the 
assumptions of Theorem \ref{thm:nonpropagationfilter1}
(cf. Example \ref{ex:homogensymbol}).

\begin{enumerate}

\item Let $m \geqs 2$, $\sigma \in \qo_+$, $b_1(x,\xi) = p(\xi)$ where 
\begin{equation*}
p(\xi) = \sum_{|\alpha| \leqs m} c_\alpha \xi^\alpha, \quad c_\alpha \in \ro,
\end{equation*}
is a polynomial with real coefficients of degree $m$. It follows that $b_1 \in G_1^{m \sigma, \sigma}$. 

If $a = b_1 + i b_2$ where $b_2 \in G_1^{0,\sigma}$ is real-valued, then it follows that 
$a$ satisfies the assumptions of Theorem \ref{thm:nonpropagationfilter1}
provided $\sigma < \frac{\rho}{m - \rho} \leqs \frac1{m-1}$. 
The corresponding operator is $a^w(x,D) = p(D) + i b_2^w(x,D)$.

\item If $\alpha_j, \beta_j \in \ro$ for $1 \leqs j \leqs d$, $k, m \in \no \setminus 0$, $\sigma = \frac{k}{m}$, $p > 0$ and 
\begin{equation*}
b_1(x, \xi) = \left( \left( \sum_{j=1}^d \alpha_j x_j^2 \right)^k + \left( \sum_{j=1}^d \beta_j \xi_j^2 \right)^m \right)^p
\end{equation*}
then $\psi_\mu b_1 \in G_1^{2 k p, \sigma}$ if $\mu > 0$.
If $a = b_1 + i b_2$ where $b_2 \in G_1^{0,\sigma}$ is real-valued, then it follows that 
$a$ satisfies the assumptions of Theorem \ref{thm:nonpropagationfilter1} provided $p < \rho p_c$
where the \emph{critical exponent} 
\cite{Cappiello7} is defined as
\begin{equation}\label{eq:pcrit}
p_c = \frac12 \left( \frac1k + \frac1m \right). 
\end{equation}
\end{enumerate}
\end{example}

Finally we devote our interest to Weyl symbols of a hamiltonian of the form 
\begin{equation}\label{eq:hamiltoniansymbolpower}
a(x,\xi) = \psi_\mu(x,\xi) \left( x^{2k} + \xi^{2m} \right)^p, \quad x, \xi \in \ro, 
\end{equation}
where $k,m \in \no \setminus 0$, $p \in \ro \setminus 0$ and $\mu > 0$, cf. Remark \ref{rem:cutoffsymbol}. Note that we work in dimension $d = 1$. 
As usual $\sigma = \frac{k}{m}$.
By the proof of \cite[Lemma~3.6]{Cappiello6} we have $a \in G^{2 k p,\sigma} = G_1^{2 k p,\sigma}$.

\begin{rem}\label{rem:symbolpoweroperator}
Consider the symbol $a_p = a$ defined by \eqref{eq:hamiltoniansymbolpower} as a function of $p > 0$. 
A natural question concerns relations between the Weyl quantization $a_p^w(x,D)$ and $\left(a_1^w(x,D) \right)^p$
where the latter is a power, in the operator theoretic sense, of the generalized anharmonic oscillator $a_1^w(x,D) = x^{2k} + (- \partial^2)^m$. 
Using the ellipticity of $a_1$, $a_1^w(x,D) \geqs 0$, \eqref{eq:sobolevweightestimate1}
and \cite[Theorem~4.3.6]{Nicola1}, 
we reach the following conclusion. 
We have
\begin{equation}\label{eq:comparisonpowers}
\overline{a_1^w(x,D)}^{ \, p} = 
\overline{ \left( a_1^p \right)^w (x,D)} + b^w(x,D)
\end{equation}
where $\overline A$ denotes operator closure, $b \in G_{\rho}^{r,\sigma}$, and
$\rho = \min \left( \sigma, \frac1\sigma \right)$ and $r = 2 k p - (1+\rho) \min(1,\sigma)$. 
Since the orders of the operators 
$\overline{a_1^w(x,D)}^{ \, p}$ and 
$\overline{ \left( a_1^p \right)^w (x,D)}$
are both equal to $2 k p$, \eqref{eq:comparisonpowers} 
says that they are identical modulo a term of lower order, and $\rho \leqs 1$. 
With operator closure understood we thus conclude that $\left( x^{2k} + (- \partial^2)^m \right)^p$ equals 
the Weyl quantization of 
\begin{equation}\label{eq:hamiltoniansymbolpower2}
a(x,\xi) = \psi_\mu (x,\xi)^p  \left( x^{2k} + \xi^{2m} \right)^p
\end{equation}
plus a remainder of lower order. 
\end{rem}

First we show that the evolution equation \eqref{eq:anharmonicCPinhom} is well-posed in the anisotropic modulation spaces \eqref{eq:SSmodnorm}. 

\begin{prop}\label{prop:WellPosedPower}
Let $k,m \in \no \setminus 0$, $\sigma = \frac{k}{m}$, $p \in \ro \setminus 0$, $\mu > 0$ and let $a$ be defined by 
\eqref{eq:hamiltoniansymbolpower}.
If $s \in \ro$, $u_0 \in M_{\sigma,s}(\ro)$ 
and $f \in L^\infty(\ro, M_{\sigma,s})$
then
the equation \eqref{eq:anharmonicCPinhom} has a unique solution $u(t,\cdot) = \cK_t u_0 \in (C \cap L^\infty)(\ro, M_{\sigma,s})$ such that 
$\partial_t u(t,\cdot) \in L^\infty(\ro, M_{\sigma,s - 2 k p})$.
\end{prop}

\begin{proof}
The symbol classes $G^{r,\sigma}$ do not quite fit into the symbols used in \cite{Nicola1}. 
But as explained in \cite[Remark~2.2]{Cappiello7} we may still use the calculus in \cite{Nicola1}
for aspects which do not require the precise anisotropic decay behavior of the derivatives of symbols. 

The symbol \eqref{eq:hamiltoniansymbolpower} is thus elliptic in the symbol class $G^{2 k p, \sigma}$
according to \cite[Definition~1.3.1]{Nicola1}.
By \cite[Theorem~4.2.9]{Nicola1} the closure of $a^w(x,D)$ as an unbounded operator in $L^2(\ro)$ has as spectrum a sequence of real eigenvalues $(\lambda_j)_{j=1}^\infty \subseteq \ro$ diverging to $+\infty$, 
and an associated sequence of eigenfunctions $(\fy_j)_{j=1}^\infty \subseteq \cS(\ro)$ which constitutes 
an orthonormal basis in $L^2(\ro)$. 

The operator $a^w(x,D)$ is not guaranteed to be non-negative, but it does satisfy $a^w(x,D) + C \geqs 0$, 
that is $\left( (a^w(x,D) + C) f, f \right) \geqs 0$ for all $f \in \cS(\ro)$, for some $C > 0$ by \cite[Lemma~4.2.8]{Nicola1}. 
The Weyl symbol of $b^w(x,D) = a^w(x,D) + C$ is $b = a + C$, 
and the eigenvalues of $b^w(x,D)$ are $(\lambda_j + C)_{j=1}^\infty \subseteq \ro_+$ 
whereas its sequence of eigenfunctions remains $(\fy_j)_{j=1}^\infty \subseteq \cS(\ro)$.

As in the proof of \cite[Proposition~3.1]{Cappiello7} it follows that we may express the modulation space norm as
\begin{equation}\label{eq:ModSpaceNorm}
\| u \|_{M_{\sigma,s}}^2
\asymp \sum_{j = 1}^\infty ( \lambda_j + C)^{\frac{s}{k p}} | (u, \fy_j)|^2
\end{equation}
for all $s \in \ro$.

Then we construct as in the proof of 
\cite[Theorem~3.6]{Cappiello7} the series
\begin{equation}\label{eq:SeriesSol}
u(t,x) = \sum_{j=1}^\infty e^{ - i\lambda_j t } \left( c_j +  \int_0^t f_j(\tau) e^{i\lambda_j \tau} \, \dd \tau \right) \fy_j(x)
\end{equation}
where $c_j = ( u_0, \fy_j)$ and $f_j (t) = ( f(t,\cdot), \fy_j ) \in L^\infty(\ro, \co)$ for $j \geqs 1$.
Then $u(0,\cdot) = u_0$, and  
as in the proof of \cite[Theorem~3.6]{Cappiello7} it follows that 
$u(t,\cdot) \in (C \cap L^\infty)(\ro, M_{\sigma,s})$, 
$\partial_t u(t,\cdot) \in L^\infty(\ro, M_{\sigma,s - 2 k p})$, 
and $u(t,\cdot)$ is the unique solution to \eqref{eq:anharmonicCPinhom}. 
\end{proof}

From \eqref{eq:SSSchwartz} we get the following consequence.   

\begin{cor}\label{cor:WellPosedPower}
Let $k,m \in \no \setminus 0$, $p \in \ro \setminus 0$, $\mu > 0$ and let $a$ be defined by 
\eqref{eq:hamiltoniansymbolpower}.
If $u_0 \in \cS'(\ro)$ then there exists a unique solution $u=u(t,\cdot) = \cK_t u_0 \in \cS'(\ro)$ to \eqref{eq:anharmonicCP} for all $t \in \ro$. 
\end{cor}

In our final propagation result 
we use the critical exponent \eqref{eq:pcrit} and we restrict to $p > p_c$. 
In fact the case $p = p_c$ means that $a \in G^{2 k p_c,\sigma} = G^{1+\sigma,\sigma}$ satisfies 
the anisotropic homogeneity \eqref{eq:a0homogeneity} so it is covered by Theorem \ref{thm:propagationfilter1}. 
The case $p < p_c$ implies 
by \cite[Lemma~3.6]{Cappiello6}
that $a \in G^{2 k p,\sigma}$ with order $2 k p < 1 + \sigma$. 
It is covered by Theorem \ref{thm:nonpropagationfilter1}.

In \cite[Theorem~4.7]{Cappiello7} we calculate the hamiltonian flow $\chi_t: \rr 2 \setminus 0 \to \rr 2 \setminus 0$ with $t \in \ro$ corresponding to \eqref{eq:hamiltoniansymbolpower} for $p \neq 0$. 
It is smooth in $\rr 2 \setminus 0$, periodic with respect to $t \in \ro$, and involves incomplete beta functions and their inverses. 
The period $T = T_{k,m,p}(x,\xi)$ depends on the initial datum $(x,\xi) \in \rr 2 \setminus 0$ as 
\begin{equation}\label{eq:periodhamiltonflow}
T = 
\frac{\Gamma \left( \frac1{2k}\right) \Gamma \left( \frac1{2m}\right)}{k m p \,\Gamma \left( \frac1{2k} + \frac1{2m} \right)}\left( x^{2k} + \xi^{2m} \right)^{p_c-p}
\end{equation}
cf. \cite[Remark~4.8]{Cappiello7}. 

We will use the assumption 
\begin{equation}\label{eq:pinterval}
p_c < p \leqs p_c + \frac{1}{4} \min \left( \frac1k, \frac1m \right). 
\end{equation}
This implies that $p \leqs \frac{5}{4}$. If $k = m = 1$ 
then any $1 < p \leqs \frac{5}{4}$ is allowed. 

\begin{lem}\label{lem:symbolflow}
Suppose $k, m \in \no \setminus 0$, $\sigma = \frac{k}{m}$, 
$\mu > 0$,
and let $p$ satisfy \eqref{eq:pinterval}. 
If $a$ is defined by \eqref{eq:hamiltoniansymbolpower} and $\chi_t$ for $t \in \ro$ is the corresponding hamiltonian flow,
then for any $\rho$ such that 
\begin{equation*}
0 < \rho \leqs 1 - 2 \max(k,m) (p-p_c)
\end{equation*}
and any $q \in G_1^{0,\sigma}$ we have $q(t,\cdot) = \psi_\mu (q \circ \chi_{-t} ) \in G_{\rho}^{0,\sigma}$ uniformly for $t \in \ro$.
\end{lem}

\begin{proof}
The assumption \eqref{eq:pinterval} means that 
\begin{equation}\label{eq:lowerboundrho}
\frac{1}{2} \leqs 1 - 2 \max(k,m) (p-p_c) < 1. 
\end{equation}
From \cite[Proposition~4.14]{Cappiello7}
we obtain the estimates for $\alpha, \beta \in \no$ and $t \in \ro$
\begin{equation}\label{eq:compderivative1}
|\pdd x \alpha \pdd \xi \beta q(t,x,\xi)| 
\lesssim \theta_\sigma(x,\xi)^{2 k \max(p - p_c,0) (\alpha+\beta) - \left( \alpha + \sigma \beta \right)}. 
\end{equation}
They imply that $q(t, \cdot) \in G_{\rho}^{0,\sigma}$
for any $\rho$ that satisfies 
\begin{equation*}
0 < \rho \leqs 1 - 2 \max(k,m) (p-p_c). 
\end{equation*}
\end{proof}

\begin{thm}\label{thm:propagationfilter2}
Let $k, m \in \no \setminus 0$, $\sigma = \frac{k}{m}$, 
$\mu > 0$,
and suppose that $a$ is defined by \eqref{eq:hamiltoniansymbolpower}
with $p$ restricted as in \eqref{eq:pinterval}. 
Let $\chi_t: \rr 2 \setminus 0 \to \rr 2 \setminus 0$ denote the hamiltonian flow corresponding to $a$. 
If $u_0 \in \cS'(\ro)$ then there exists a unique solution $u=u(t,\cdot) = \cK_t u_0 \in \cS'(\ro)$ to \eqref{eq:anharmonicCP} for all $t \in \ro$. For any $\rho$ such that 
\begin{equation}\label{eq:rhointerval}
\frac{1}{2} \leqs \rho \leqs 1 - 2 \max(k,m) (p-p_c)
\end{equation}
we have
\begin{align}
\chi_t \left( \fF_{1} ( u_0  ) \right) & \subseteq \fF_{\rho} ( u(t,\cdot) ), \label{eq:propsupercritical1} \quad \mbox{{\rm and}} \\
\fF_{1} ( u(t,\cdot) ) & \subseteq  \chi_t \left( \fF_{\rho} ( u_0  ) \right) , \quad t \in \ro. \label{eq:propsupercritical2}
\end{align}
\end{thm}

\begin{proof}
By \eqref{eq:SSSchwartz} there exists $s \in \ro$ such that $u_0 \in M_{\sigma,s}$, so the equation \eqref{eq:anharmonicCP} has by Proposition \ref{prop:WellPosedPower} a unique solution $u \in (C \cap L^\infty)(\ro, M_{\sigma,s})$. 

If $\rho \geqs \frac{1}{2}$ then
it follows from \eqref{eq:pcrit}
and \eqref{eq:pinterval} that 
\begin{equation}\label{eq:positiveparameter}
\begin{aligned}
\alpha & := 3 \rho (1+\sigma ) - 2 k p \\
& \geqs 3 \rho (1+\sigma ) - 2 k \left( p_c + \frac{1}{4} \min \left( \frac1k, \frac1m \right) \right) \\
& = (1+\sigma ) \left( 3 \rho - 1 \right) - \frac{1}{2} \min \left( 1, \sigma \right) \\
& > (1+\sigma )  \left( 3 \rho - 1 - \frac{1}{2} \right) 
= 3 (1+\sigma )  \left( \rho - \frac{1}{2} \right) 
\geqs 0.
\end{aligned}
\end{equation}

As in the proof of Theorem \ref{thm:propagationfilter1} we may assume that $u_0 \notin \cS$ and $\fF_{1} ( u_0 ) \neq \rP(\rr 2)$.
Let $\emptyset \neq \Omega \subseteq \rr 2$ satisfy
$\rr 2 \setminus \Omega \in \fF_{1} ( u_0 )$. 
By Proposition \ref{prop:smoothorderzero} there exist $0 < \ep_1 < \delta_1 < 1$
and $q_1 = q_{\ep_1,\delta_1,1,\Omega} \in G_1^{0,\sigma}$
such that $q_1^w u_0 \in \cS$ and 
\begin{equation}\label{eq:q1one}
q_1 |_{\Omega} \equiv 1.
\end{equation}
If $\rho$ satisfies \eqref{eq:rhointerval} then 
by Lemma \ref{lem:symbolflow} we have $q_1(t) = q_1(t,\cdot) = \psi_\mu ( q_1 \circ \chi_{-t} ) \in G_{\rho}^{0,\sigma}$ 
uniformly for all $t \in \ro$.

From $u \in (C \cap L^\infty)(\ro, M_{\sigma,s})$ and \eqref{eq:ShubinSobolevCont}
we obtain
\begin{equation}\label{eq:modspace3}
q_1(t)^w u \in L^\infty( \ro , M_{\sigma,s}). 
\end{equation} 
Again we use \eqref{eq:modspace3} as the starting point for an iteration. 
First we deduce from \eqref{eq:modspace3} the improved regularity
\begin{equation}\label{eq:regularization2}
q_1(t) ^w u \in C( \ro, M_{\sigma,s + \alpha}) 
\end{equation} 
where $\alpha > 0$ is defined by \eqref{eq:positiveparameter}.

Indeed we may 
regard $u_1 = q_1(t)^w u$ as the solution of the non-homogeneous Cauchy problem 
\begin{equation}\label{eq:SchrodEqIter3}
\begin{cases} 
P u_1 = f_1 \\ 
u_1(0,\cdot) = u_{1,0}
\end{cases}
\end{equation}
where $P = \partial_t + i a^w(x,D)$
and $u_{1,0} = q_1(0)^w u_0 = (\psi_\mu q_1)^w u_0 =  q_1^w u_0 + ((\psi_\mu -1) q_1)^w u_0\in \cS \subseteq M_{\sigma, s + \alpha}$ due to $q_1^w u_0 \in \cS$ and the compact support of $\psi_\mu -1$ which gives a regularizing operator $((\psi_\mu -1) q_1)^w: \cS' \to \cS$. 
To express $f_1$ we write, using $P u = 0$,
\begin{equation*}
f_1 = P u_1 = P q_1(t) ^w u = P q_1(t)^w u - q_1(t)^w P u 
= \big( \left( \partial_t q_1(t) \right)^w + i [ a^w, q_1(t) ^w ] \big) u.
\end{equation*}
The right hand side operator $b_1(t)^w = \left( \partial_t q_1(t) \right)^w + i [ a^w, q_1(t) ^w ]$
has Weyl symbol 
\begin{equation}\label{eq:b1symbol}
b_1(t) = \partial_t q_1(t) + i \left( a \wpr q_1 (t) - q_1(t) \wpr a \right)
\end{equation}
which we need to investigate. 
We rewrite Hamilton's system of equations \eqref{eq:Hamiltoneq} in the form (cf. \cite[Section~6]{Cappiello6})
\begin{equation}\label{eq:hamiltoncompact}
\begin{aligned}
\left( 
\begin{array}{l}
x' (t)  \\
\xi'(t)
\end{array}
\right)
& = \J \nabla_{x,\xi} a \left( x(t), \xi(t) \right), \quad x(0) = x, \quad \xi(0) = \xi, 
\end{aligned}
\end{equation}
where $\J$ is defined in \eqref{eq:Jdef} with $d = 1$. 

Writing $q_1 (t) ( \chi_t (x, \xi) ) =   \psi_\mu ( \chi_t (x, \xi) ) q_1 ( x,\xi )$, 
differentiation with respect to $t$,
$\partial_t \chi_t(x,\xi) = \J \nabla a ( \chi_t(x,\xi) )$ 
(cf. \eqref{eq:hamiltoncompact})
and the Chain Rule
give for $(x,\xi) \in \rr 2 \setminus 0$
\begin{equation*}
\begin{aligned}
 \{ a, \psi_\mu \} ( \chi_{t} (x,\xi) ) q_1 ( x,\xi )
& = \la \nabla  \psi_\mu ( \chi_{t} (x,\xi) ),  \J \nabla a ( \chi_t(x,\xi) ) \ra q_1 ( x,\xi ) \\
& = \la \nabla  \psi_\mu ( \chi_{t} (x,\xi) ),  \partial_t \chi_t(x,\xi) \ra q_1 ( x,\xi ) \\
& = \left( \partial_t q_1(t) \right) ( \chi_t (x, \xi) ) + \la \nabla q_1 (t) ( \chi_{t} (x,\xi) ), \J \nabla a ( \chi_{t} (x,\xi) ) \ra \\
& = \left( \partial_t q_1(t) \right) ( \chi_t (x, \xi) ) + \{ a , q_1(t) \} ( \chi_{t} (x,\xi) )
\end{aligned}
\end{equation*}
where we use the Poisson bracket notation
\begin{equation*}
\{ a, q_1 (t) \} = \la \nabla_\xi a, \nabla_x  q_1(t) \ra - \la \nabla_x a, \nabla_\xi  q_1(t) \ra = \la \J \nabla_{x,\xi} a, \nabla_{x,\xi} q_1(t) \ra. 
\end{equation*}
Thus for all $(x,\xi) \in \rr 2$
\begin{equation}\label{eq:symbolcompsupp}
\partial_t q_1 (t)(x, \xi)  + \{ a, q_1(t) \}(x,\xi) = \{ a, \psi_\mu \} (x,\xi) q_1 ( \chi_{-t} (x,\xi) ) \in C_c^\infty(\rr 2).
\end{equation}
In fact the right hand side has support in $\overline{\rB}_\mu \subseteq \rr 2$.

By \cite[Theorem~18.5.4]{Hormander1} (cf. \cite[Section~8]{Cappiello6}) we have
\begin{equation*}
i \left( a \wpr q_1(t)  - q_1(t) \wpr a \right) 
\sim \sum_{j=0}^{\infty} \frac{(-1)^{j}}{(2j + 1)! 2^{ 2 j }} \, \{ a,q_1(t) \}_{2j+1}
\end{equation*}
where for $j \geqs 1$
\begin{equation*}
\{ f,g \}_{j} (x,\xi) 
= (-1)^j \left( \la \partial_x, \partial_\eta \ra - \la \partial_y, \partial_\xi \ra  \right)^{j} 
f(x,\xi) g (y,\eta) \Big|_{(y,\eta) = (x,\xi)}
\end{equation*}
are bilinear differential operators that extend the Poisson bracket $\{ f,g \} = \{ f,g \}_1$.  

For $j \geqs 0$ we have (cf. \cite[Eq.~(8.9)]{Cappiello6} with modified assumption $0 < \rho \leqs 1$)
\begin{equation*}
a \in G_\rho^{r,\sigma}, \quad b \in G_\rho^{s,\sigma} \quad \Longrightarrow \quad \{ a, b \}_j \in G_\rho^{r + s - j \rho (1+\sigma).\sigma}. 
\end{equation*}
By \cite[Lemma~3.6]{Cappiello6} we have $a \in G^{2 k p, \sigma} \subseteq G_\rho^{2 k p, \sigma}$ so this gives
\begin{equation}\label{eq:symbolnegorder}
\sum_{j=1}^{\infty} \frac{(-1)^{j}}{(2j + 1)! 2^{ 2 j }} \, \{ a,q_1(t) \}_{2j+1}
\in G_\rho^{ 2 k p - 3 \rho (1+\sigma), \sigma}
= G_\rho^{ - \alpha, \sigma}
\end{equation}
uniformly for $t \in \ro$,
using \eqref{eq:positiveparameter}.

Combining \eqref{eq:b1symbol} , \eqref{eq:symbolcompsupp} and \eqref{eq:symbolnegorder}
we find that 
\begin{align*}
b_1 (t) & = \partial_t q_1(t) + i \left( a \wpr q_1 (t) - q_1(t) \wpr a \right) \\
& \sim \partial_t q_1(t) + \{ a, q_1 (t) \} + \sum_{j=1}^{\infty} \frac{(-1)^{j}}{(2j + 1)! 2^{ 2 j }} \, \{ a,q_1(t) \}_{2j+1} 
\in G_\rho^{ - \alpha, \sigma}
\end{align*}
uniformly for $t \in \ro$,

From \eqref{eq:modspace3} we have
$u_1 \in L^\infty( \ro , M_{\sigma,s})$ and by \eqref{eq:ShubinSobolevCont} we obtain
\begin{equation*}
f_1  = P u_1 = b_1(t)^w u \in L^\infty( \ro , M_{\sigma,s + \alpha} ). 
\end{equation*}
Now \eqref{eq:SchrodEqIter3}, $u_{1,0} \in M_{\sigma, s + \alpha}$ and 
Proposition \ref{prop:WellPosedPower}
show that \eqref{eq:regularization2} holds as claimed. 

In the next step we take $0 < \ep_2 < \delta_2 < \ep_1$
such that $\delta_2 < b_\sigma$ defined in \eqref{eq:bsigma},
and define $q_2 = q_{\ep_2, \delta_2,1,\Omega} \in G_1^{0,\sigma}$ as in 
Proposition \ref{prop:cutoffsymbol}.
By Lemma \ref{lem:symbolflow} we have $q_2(t) = q_2(t,\cdot) = \psi_\mu (q_2 \circ \chi_{-t}) \in G_{\rho}^{0,\sigma}$ uniformly for $t \in \ro$.

For fixed $t$ we have by Remark \ref{rem:supportintersection} $\supp q_2 (t) \cap \supp (1 - q_1 (t) ) = \emptyset$ 
which by the calculus yields $r (t) := q_2 (t) \wpr ( 1 - q_1 (t) ) \in \cS(\rr 2)$. 
Combined with $q_1(0)^w u_0 \in \cS$
and \eqref{eq:ShubinSobolevCont} we find
\begin{equation}\label{eq:schwartzmodspace4}
q_2(0)^w u_0 = q_2(0)^w q_1(0)^w u_0 + r(0)^w u_0 \in \cS
\end{equation}
since $r(0)^w: \cS' \to \cS$ is regularizing. 

Set $u_2 = q_2(t)^w u$ and consider the non-homogeneous Cauchy problem 
\begin{equation}\label{eq:SchrodEqIter4}
\begin{cases} 
P u_2 = f_2 \\ 
u_2(0,\cdot) = u_{2,0}
\end{cases} 
\end{equation}
where $u_{2,0} = q_2(0)^w u_0 \in M_{\sigma, s + 2 \alpha}$ by 
\eqref{eq:SSSchwartz}. 
As above we obtain
\begin{equation*}
f_2 = P q_2(t)^w u - q_2(t)^w Pu = 
b_2(t) ^w u
= b_2(t)^w q_1(t)^w u + b_2(t)^w (1 - q_1(t))^w u
\end{equation*}
with $b_2(t) ^w = [P, q_2(t)^w] = \left( \partial_t q_2(t) \right)^w +  i [a^w, q_2(t)^w]$, 
and $b_2(t) \in G_\rho^{-\alpha,\sigma}$ uniformly for $t \in \ro$ as above. 

From $\supp b_2 (t) \subseteq \supp q_2(t)$
we may conclude that $\supp b_2 (t)\cap \supp (1 - q_1 (t)) = \emptyset$.
This implies that  
$r_2(t) := b_2 (t) \wpr (1 - q _1(t)) \in \cS(\rr 2)$ and $r_2(t)^w:  \cS' \to \cS$ is regularizing. 
We get for $t \in \ro$
\begin{equation*}
f_2 = b_2(t)^w u = b_2(t)^w q_1(t)^w u + r_2(t)^w u \in L^\infty(\ro,M_{\sigma,s + 2 \alpha})
\end{equation*}
due to \eqref{eq:regularization2} and \eqref{eq:ShubinSobolevCont}.
From
\eqref{eq:SchrodEqIter4}, $u_{2,0} \in M_{\sigma, s_0 + 2 \alpha}$ and 
Proposition \ref{prop:WellPosedPower}
we obtain
$u_2 = q_2(t)^w u \in C( \ro , M_{\sigma,s + 2 \alpha} )$. 

This argument shows that
\begin{equation}\label{eq:regularityk2}
q_k(t)^w u \in C( \ro , M_{\sigma,s + k \alpha} )
\end{equation}
for any $k \in \no$, with $q_k = q_{\ep_k, \delta_k,1,\Omega} \in G_1^{0,\sigma}$ and 
$q_k(t) = q_k(t,\cdot) = \psi_\mu (q_k \circ \chi_{-t}) \in G_{\rho}^{0,\sigma}$ for all $t \in \ro$, 
in each step decreasing $\ep_k$ and $\delta_k$ such that $0 < \ep_k < \delta_k < \ep_{k-1}$
and $\delta_k < b_\sigma$ defined by \eqref{eq:bsigma},
where $\ep_{k-1}$ has been chosen in the preceding step. 
We do this keeping $\ep_k > 2 \nu > 0$ for all $k \in \no$ for a fixed $\nu > 0$. 
Let $0 < \ep < \nu$ and set $q = q_{\ep, \nu,1,\Omega} \in G_1^{0,\sigma}$ as in Proposition \ref{prop:cutoffsymbol}.
Then $0 < \ep < \nu < 2 \nu < \ep_k < \delta_k$ for all $k \in \no \setminus 0$. 

By Lemma \ref{lem:closureinclusion} we have for any $k \in \no \setminus 0$
\begin{equation}\label{eq:supportintersectionempty}
\supp q \cap \supp (1 - q_k ) 
\subseteq \overline{ \left( \Omega_{1,\nu} \right)} \cap \left( \rr 2 \setminus \Omega_{1,\ep_k} \right)
\subseteq \Omega_{1,2 \nu } \cap \left( \rr 2 \setminus \Omega_{1,\ep_k} \right)
= \emptyset. 
\end{equation}
We have $q(t) = q(t,\cdot) = \psi_\mu (q \circ \chi_{-t} ) \in G_{\rho}^{0,\sigma}$ by Lemma \ref{lem:symbolflow}
uniformly for all $t \in \ro$, and \eqref{eq:supportintersectionempty} implies $\supp q(t) \cap \supp (1 - q_k (t))  = \emptyset$
for each fixed $t \in \ro$. 

For any $t \in \ro$ this yields $r_k (t) := q (t) \wpr (1-q_k (t) ) \in \cS(\rr 2)$.
Hence for any $t \in \ro$ and any $k \in \no \setminus 0$ we get from \eqref{eq:regularityk2}
and \eqref{eq:ShubinSobolevCont}
\begin{equation*}
q(t)^w u (t) = q(t)^w q_k(t)^w u (t) + r_k(t)^w u (t) \in M_{\sigma,s + k \alpha}
\end{equation*}
which implies
\begin{equation*}
q(t)^w u(t,\cdot) \in \bigcap_{s \in \ro} M_{\sigma,s} = \cS(\ro). 
\end{equation*}

From \eqref{eq:q1one} it follows that 
$q(t)|_{\chi_t \Omega \setminus \rB_\mu} \equiv 1$ so we have shown that $\chi_t \left( \rr {2} \setminus \Omega \right) \in \fF_{\rho}( u(t,\cdot) )$. 
It follows that $\chi_t \left( \fF_{1} (u_0) \right) \subseteq \fF_\rho ( u(t,\cdot) )$ for any $t \in \ro$ so
we have proven \eqref{eq:propsupercritical1}.
The remaining inclusion \eqref{eq:propsupercritical2} is a consequence of \eqref{eq:propsupercritical1},
$\cK_t^{-1} =  \cK_{-t}$ and $\chi_t^{-1} = \chi_t$. 
\end{proof}

\begin{rem}\label{rem:parametergap}
By Remark \ref{rem:filtersing4} we have $\fF_{1} (u) \subseteq \fF_{\rho} (u)$ when $\rho \leqs 1$
but we cannot obtain a filter propagation equality from \eqref{eq:propsupercritical1} and \eqref{eq:propsupercritical2}
for some parameter $0 < \wt \rho \leqs 1$.
\end{rem}

\subsection{Examples}\label{subsec:examples}

\begin{example}\label{ex:harmoscpower}
Let $k = m = 1$ and consider the hamiltonian Weyl symbol \eqref{eq:hamiltoniansymbolpower}, 
with $p_c = 1 < p \leqs \frac{5}{4}$, which is the interval \eqref{eq:pinterval}. 
The propagation result Theorem \ref{thm:propagationfilter2} then holds. 

The solution $( x(t), \xi(t) ) = \chi_t(x,\xi)$ to Hamilton's system of equations \eqref{eq:Hamiltoneq} is
\begin{equation}\label{eq:hamiltonflowharmosc}
\left( 
\begin{array}{l}
x(t) \\
\xi(t) 
\end{array}
\right)
=
\left( 
\begin{array}{ll}
\cos \left( 2 p (x^2 + \xi^2)^{p-1} t \right) & \sin \left( 2 p (x^2 + \xi^2)^{p-1} t \right) \\
- \sin \left( 2 p (x^2 + \xi^2)^{p-1} t \right) & \cos \left( 2 p (x^2 + \xi^2)^{p-1} t \right)
\end{array}
\right)
\left( 
\begin{array}{l}
x \\
\xi
\end{array}
\right)
\end{equation}
provided $(x,\xi) \in \rr 2 \setminus \overline{\rB_\mu}$.
The hamiltonian flow is thus a circular clockwise rotation with period that depend on $x^2 + \xi^2$ as
$T = \frac{\pi}{p} (x^2 + \xi^2)^{1-p}$. 
In fact this is a particular case of \eqref{eq:periodhamiltonflow}. 

The anisotropy parameter is $\sigma = 1$, but the hamiltonian flow does not preserve 
conic subsets of phase space, due to \eqref{eq:periodhamiltonflow} and $p \neq 1$.  
Alternatively this can be seen from the homogeneity $a(\lambda x, \lambda \xi) = \lambda^{2p} a(x,\xi)$, 
$2 p - 1 > 1$ and \cite[Propositions~6.2 and 6.4]{Cappiello6}.

\begin{figure}
\begin{subfigure}{.5\textwidth}
  \centering
  \includegraphics[width=.8\linewidth]{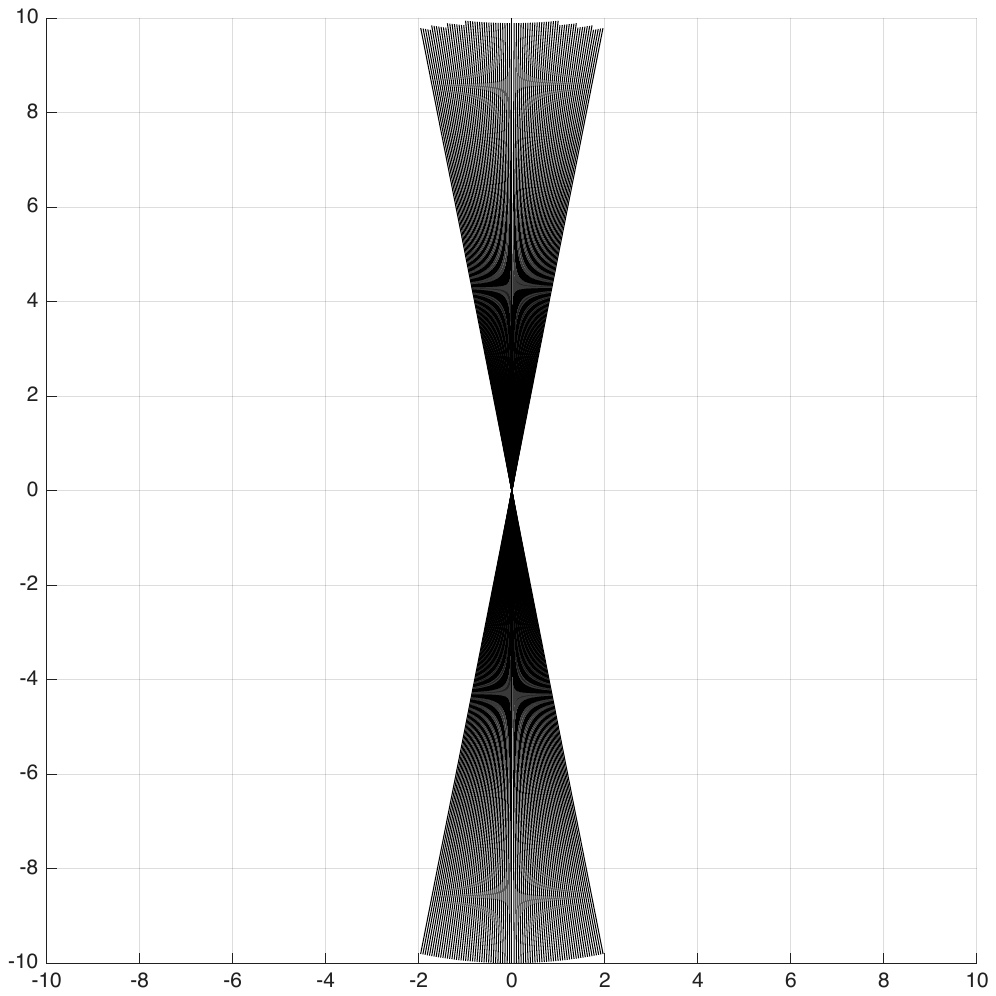}
  \caption{$\Omega_C \subseteq \rr 2$}
  \label{fig:sub1a}
\end{subfigure}%
\begin{subfigure}{.5\textwidth}
  \centering
  \includegraphics[width=.8\linewidth]{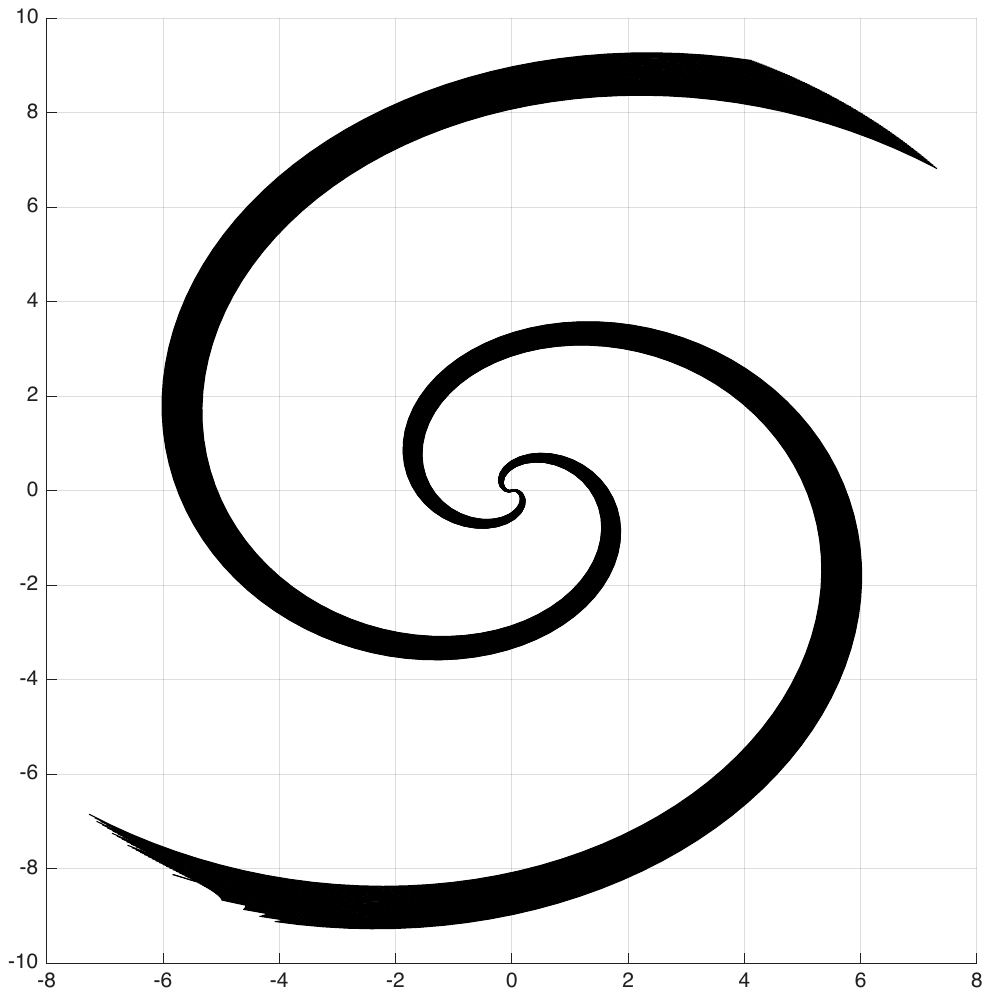}
  \caption{$\chi_t( \Omega_C) \subseteq \rr 2$}
  \label{fig:sub2a}
\end{subfigure}
\caption{Example \ref{ex:harmoscpower} where $\sigma = 1$. (A) A set $\Omega_C$ as in \eqref{eq:neigbhfreqaxisiso}, and (B) $\chi_t( \Omega_C)$ for fixed $t > 0$ with $\chi_t$ defined by \eqref{eq:hamiltonflowharmosc}.}
 \label{fig:harmoscpower}
\end{figure}

By Proposition \ref{prop:filterdeltaandone} we have $\Omega_C \in \fF_{1,1}(\delta_0)$ for any $C > 0$, 
where 
\begin{equation}\label{eq:neigbhfreqaxisiso}
\Omega_C = \{ (x,\xi): C |x| \leqs |\xi| \} \subseteq \rr 2
\end{equation}
is a conic neighborhood of the frequency axis $\{ 0 \} \times \ro \subseteq \rr 2$.
Such a set, shown in Figure \ref{fig:harmoscpower} (A), 
is deformed by the hamilton flow \eqref{eq:hamiltonflowharmosc}, 
for fixed $t > 0$,
into a set of the form shown in Figure \ref{fig:harmoscpower} (B).
Here the parameters are $C=5$, $p = \frac{6}{5}$ and $t = \frac{2}{p}$.
By Theorem \ref{thm:propagationfilter2} we have $\chi_t (\Omega_C) \in \fF_{1,\rho} ( \cK_t \delta_0 )$
provided $\frac12 \leqs \rho \leqs 3 - 2 p$ for any $t \in \ro$.
\end{example}

\begin{example}\label{ex:anharmoscpower}
Let $k = 2$, $m = 1$ and consider the hamiltonian \eqref{eq:hamiltoniansymbolpower}, 
with $p_c = \frac{3}{4} < p \leqs \frac{7}{8}$, which is the interval \eqref{eq:pinterval}. 
Again the propagation result Theorem \ref{thm:propagationfilter2} holds. 
The hamiltonian flow $\chi_t(x,\xi)$ is given by \cite[Theorem~4.7]{Cappiello7}. 
It is periodic with period 
$T = \frac{\Gamma\left( \frac14 \right)^2}{ 2 p \sqrt{2 \pi}} (x^{4} + \xi^2)^{\frac34-p}$, which is again 
a particular case of \eqref{eq:periodhamiltonflow}. 

The anisotropy parameter is $\sigma = 2$, but the hamiltonian flow does not preserve 
$\sigma$-conic subsets of phase space, due to \eqref{eq:periodhamiltonflow} and $p \neq \frac34$.  
Alternatively this can be seen from the anisotropic homogeneity $a(\lambda x, \lambda^2 \xi) = \lambda^{4p} a(x,\xi)$, 
$4 p - 1 > 2$ and \cite[Propositions~6.2 and 6.4]{Cappiello6}.

\begin{figure}
\begin{subfigure}{.5\textwidth}
  \centering
  \includegraphics[width=.8\linewidth]{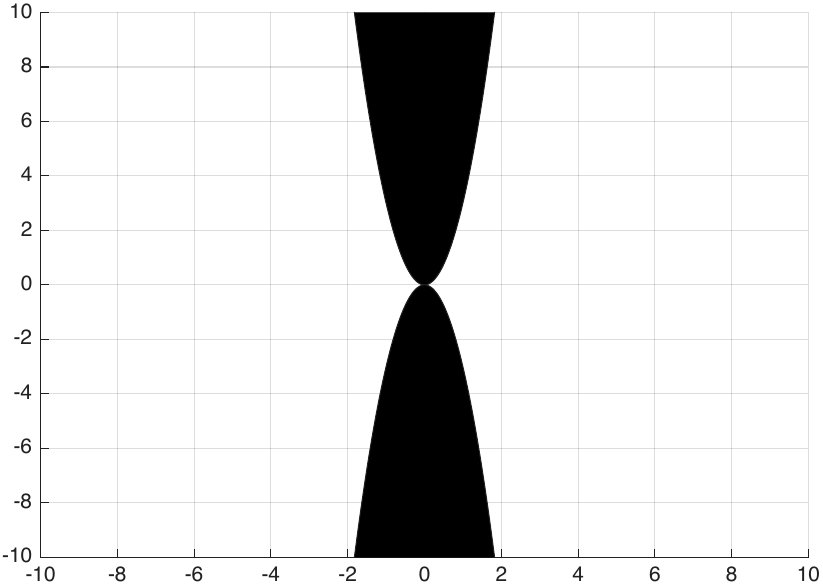}
  \caption{$\Omega_C \subseteq \rr 2$}
  \label{fig:sub1b}
\end{subfigure}%
\begin{subfigure}{.5\textwidth}
  \centering
  \includegraphics[width=.8\linewidth]{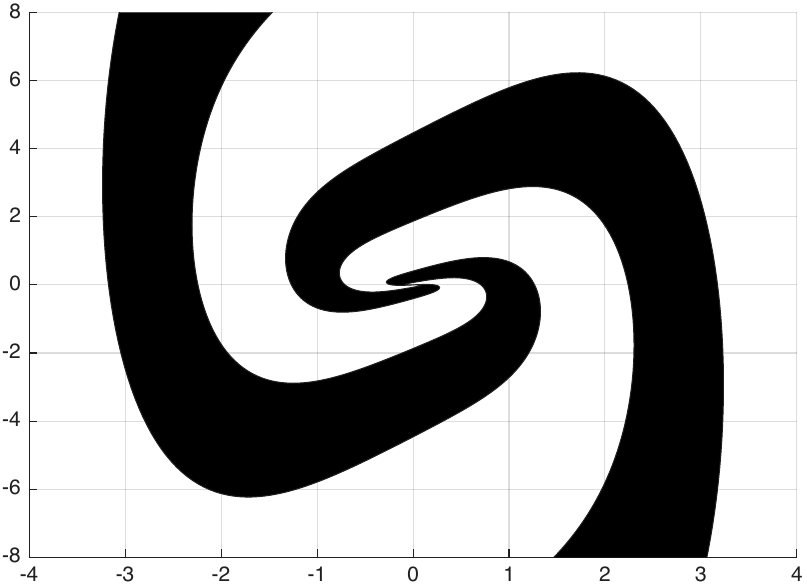}
  \caption{$\chi_t( \Omega_C) \subseteq \rr 2$}  
  \label{fig:sub2b}
\end{subfigure}
\caption{Example \ref{ex:anharmoscpower} where $\sigma = 2$. (A) A set $\Omega_C$ as in \eqref{eq:neigbhfreqaxisaniso}, and (B) $\chi_t( \Omega_C)$ for fixed $t > 0$ with hamiltonian flow $\chi_t$.}
 \label{fig:anharmoscpower}
\end{figure}

By Proposition \ref{prop:filterdeltaandone} we have $\Omega_C \in \fF_{\sigma,1}(\delta_0)$ for any $C > 0$, 
where 
\begin{equation}\label{eq:neigbhfreqaxisaniso}
\Omega_C = \{ (x,\xi): C|x|^2 \leqs |\xi| \} \subseteq \rr 2
\end{equation}
is a $\sigma$-conic neighborhood of the frequency axis $\{ 0 \} \times \ro \subseteq \rr 2$.
Such a set, shown in Figure \ref{fig:anharmoscpower} (A), 
is deformed by the hamiltonian flow $\chi_t$
for fixed $t > 0$,
into a set of the form shown in Figure \ref{fig:anharmoscpower} (B).
Here we approximate the solution $\chi_t$ to \eqref{eq:Hamiltoneq} by means of a numerical solution using MATLAB.  
The parameters are $C=3$, $p = \frac{7}{8}$ and $t = \frac{2}{p}$.
By Theorem \ref{thm:propagationfilter2} we have $\chi_t (\Omega_C) \in \fF_{\sigma,\frac12} ( \cK_t \delta_0 )$
for any $t \in \ro$.
\end{example}

\begin{rem}\label{rem:exponentone}
A natural question is whether the preceding propagation results 
extend to the case of the anharmonic oscillator
\begin{equation}\label{eq:AnHarmCP}
\begin{cases} 
\partial_t u + i \left( x^4 - \partial_x^2 \right) u = 0, \qquad x \in \rr d, \qquad t \in \ro \setminus 0,  \\ 
u(0,x) = u_0(x)
\end{cases}. 
\end{equation}	
In fact, as observed in Example \ref{ex:anharmoscpower}, the validity of Theorem \ref{thm:propagationfilter2}
is limited to powers $\frac34 < p \leqs \frac78$ and \eqref{eq:AnHarmCP} where $p = 1$ is thus excluded. 
For $p=1$ we nevertheless have the propagation result \cite[Theorem~1.5]{Cappiello7}
where the filter sets are restricted to be anisotropically annular as in \eqref{eq:AnisoAnnular}. 
It is an open problem to formulate 
propagation results when $p = 1$ for more general filters of singularities. 
\end{rem}

\section*{Acknowledgments}
The second author is a member of Gruppo Nazionale per l’Analisi Matematica, la Probabilit\`a e le loro Applicazioni (GNAMPA) -- Istituto Nazionale di Alta Matematica
(INdAM).







\begin{thebibliography}{2000}

%
\bibitem{Bambusi1}
D.~Bambusi and B.~Langella, 
\textit{Globally integrable quantum systems and their perturbations}, 
Singularities, asymptotics, and limiting models, Springer INdAM Ser. \textbf{64}, pp. 1--34, Springer, Singapore, 2025.
%
\bibitem{Cappiello5}
M.~Cappiello, R.~Schulz and P.~Wahlberg, \textit{Shubin type Fourier integral operators and evolution equations}, 
J. Pseudo-Differ. Oper. Appl. \textbf{11} (1) (2020) 119--139.
%
\bibitem{Cappiello6}
M.~Cappiello, L.~Rodino and P.~Wahlberg, 
\textit{Propagation of anisotropic Gabor singularities for Schr\"odinger type equations}, 
J. Evol. Equ. \textbf{24} (2) Paper No. 36 (2024), 46 pp.
%
\bibitem{Cappiello7}
M.~Cappiello, L.~Rodino and P.~Wahlberg, 
\textit{Propagation of singularities for anharmonic Schr\"odinger equations}, 
J. Math. Phys. \textbf{66} (4) Paper No. 041503 (2025), 33 pp.
%
\bibitem{Chatzakou1}
M.~Chatzakou, J.~Delgado and M.~Ruzhansky, \textit{On a class of anharmonic oscillators}, 
J. Math. Pures Appl. \textbf{153} (2021) 1--29.
%
\bibitem{Folland1}
G.~B.~Folland, \textit{Harmonic Analysis in Phase Space}, Princeton University Press, 1989.
%
\bibitem{Garello1}
G.~Garello and A.~Morando, \textit{$m$-Microlocal elliptic pseudodifferential operators acting on $L_{\rm loc}^p (\Omega)$}, 
Math. Nachr. \textbf{289} (2016) 1820--1837. 
%
\bibitem{Gramchev1}
T.~Gramchev and P.~Popivanov, 
\emph{Partial differential equations. Approximate solutions in scales of functional spaces}, 
Math. Res. \textbf{108}, Wiley-VCH Verlag Berlin GmbH, Berlin, 2000. 
%
\bibitem{Grochenig1}
K.~Gr\" ochenig, \textit{Foundations of Time-Frequency Analysis}, Birkh\" auser, Boston, 2001.
%
\bibitem{Hormander1}
L.~H\"ormander,
\textit{The Analysis of Linear Partial Differential Operators}, Vol. I, III, IV,
Springer, Berlin, 1990.
%
\bibitem{Hormander2}
L.~H\"ormander, \textit{Quadratic hyperbolic operators}, Microlocal Analysis and Applications, Lecture Notes in Math. \textbf{1495}, Eds. L. Cattabriga, L. Rodino, pp. 118--160, Springer, 1991.
%
\bibitem{Ito1}
S.~Ito and K.~Kato, \textit{Wave front set of solutions to Schr\"odinger equations with perturbed harmonic oscillators}, 
J. Math. Anal. Appl. \textbf{507}  (2) Paper No. 125821 (2022), 17 pp.
%
\bibitem{Lascar1} 
R.~Lascar, \textit{Propagation des singularit\'es d'\'equations pseudodiff\'erentielles quasi homog\`enes}, 
Ann. Inst. Fourier Grenoble \textbf{27} (1977), 79--123. 
%
\bibitem{Lerner1}
N.~Lerner, \textit{Metrics on the Phase Space and Non-Selfadjoint Pseudodifferential Operators}, Birkh\"auser, Basel, 2010.
%
\bibitem{Nicola1}
F.~Nicola and L.~Rodino, \textit{Global Pseudo-Differential Calculus on Euclidean Spaces}, Birkh\"auser, Basel, 2010.
%
\bibitem{Nicola2}
F.~Nicola and L.~Rodino, \textit{Propagation of Gabor singularities for semilinear Schr\"odinger equations}, 
Nonlinear Differential Equations Appl. \textbf{22} (6) (2015) 1715--1732. 
%
\bibitem{PravdaStarov1}
K.~Pravda-Starov, L.~Rodino and P.~Wahlberg, \textit{Propagation of Gabor singularities for Schr\"odinger equations with quadratic Hamiltonians}, Math. Nachr. \textbf{291} (1) (2018) 128--159. 
%
\bibitem{Rodino1}
L. Rodino, 
\textit{Microlocal analyis for spatially inhomogeneous pseudodifferential operators}, 
Ann. Scuola Norm. Sup. Pisa, Ser IV \textbf{9} (1982) 211--253.
%
\bibitem{Rodino2}
L.~Rodino and P.~Wahlberg, \textit{The Gabor wave front set}, Monaths. Math. \textbf{173} (4) (2014) 625--655. 
%
\bibitem{Rodino3}
L.~Rodino and P.~Wahlberg, \textit{Microlocal analysis of Gelfand--Shilov spaces}, 
Ann. Mat. Pura Appl. \textbf{202} (5) (2023) 2379--2420. 
%
\bibitem{Rodino4}
L.~Rodino and P.~Wahlberg, \textit{Anisotropic global microlocal analysis for tempered distributions}, 
Monatsh. Math. \textbf{202} (2) (2023) 397--434. 
%
\bibitem{Schaefer1}
H.~H.~Schaefer and M.~P.~Wolff, \textit{Topological Vector Spaces}, 
Springer Verlag, New York, 1999. 
%
\bibitem{Shubin1}
M.~A.~Shubin, \textit{Pseudodifferential Operators and Spectral Theory}, Springer, 2001.
%
\bibitem{Simmons1}
G.~F.~Simmons, \textit{Introduction to Topology and Modern Analysis}, 
McGraw-Hill, Auckland, 1963. 
%
\bibitem{Turbiner1}
A.~V.~Turbiner, J.~C.~del~Valle~Rosales, \textit{Quantum Anharmonic Oscillator},
World Scientific, Singapore, 2023.
%
\bibitem{Wahlberg2}
P.~Wahlberg, \textit{Propagation of polynomial phase space singularities for 
Schr\"odinger equations with quadratic Hamiltonians}, 
Math. Scand. \textbf{122} (1) (2018) 107--140. 
%
\bibitem{Wahlberg3}
P.~Wahlberg, \textit{The Gabor wave front set of compactly supported distributions}, 
Advances in Microlocal and Time-Frequency Analysis, Appl. Numer. Harmon. Anal., pp. 507--520, Birkh\"auser/Springer, Cham, 2020.
%
\bibitem{Wahlberg4}
P.~Wahlberg, \textit{Propagation of anisotropic Gabor wave front sets}, 
Proc. Edinb. Math. Soc. \textbf{67} (3) (2024) 674--698.
%
\end{thebibliography}
\end{document}